\documentclass[12pt,reqno]{amsart}
\usepackage[activeacute,english]{babel}

\usepackage[utf8]{inputenc}
\usepackage{amsmath} 
\usepackage{amsthm} 
\usepackage{amssymb} 
\usepackage{amscd} 

\usepackage{emptypage}
\usepackage{enumitem} 
\usepackage{mathtools} 
\usepackage{mathrsfs} 
\usepackage{upgreek}

\usepackage{hyperref,cleveref,autonum}

\usepackage{esint} 
\usepackage{graphicx,caption} 
\usepackage{float} 

\usepackage{pgf,tikz,pgfplots}
\usetikzlibrary{arrows}
\usetikzlibrary{intersections}
\usetikzlibrary{fadings}
\usepackage{wrapfig}

\usepackage[margin=0.9in]{geometry}
\newtheorem{theorem}{Theorem}[section]
\newtheorem*{theorem*}{Theorem}
\newtheorem{proposition}[theorem]{Proposition}
\newtheorem{lemma}[theorem]{Lemma}
\newtheorem{corollary}[theorem]{Corollary}
\newtheorem{conjecture}[theorem]{Conjecture}

\theoremstyle{definition}

\theoremstyle{remark}
\newtheorem{remark}[theorem]{Remark}
\newtheorem*{remark*}{Remark}

\newcommand{\con}[1]{\mathbb{#1}}
\newcommand{\C}{\con{C}} 
\newcommand{\R}{\con{R}} 
\newcommand{\N}{\con{N}} 

\newcommand{\Dom}{\mathrm{Dom}}
\newcommand{\Rodzin}{\mathcal{R}}

\numberwithin{equation}{section}
\makeatletter
\def\l@subsection{\@tocline{2}{0pt}{2.5pc}{5pc}{}}
\makeatother

\title[The $\overline\partial$-Robin Laplacian]
{The $\overline\partial$-Robin Laplacian}

\author[J. Duran]{Joaquim Duran}
\address{J. Duran
\newline
Centre de Recerca Matem\`atica, Edifici C, Campus Bellaterra, 08193 Bellaterra, Spain.}
\email{jduran@crm.cat}

\date{\today}

\thanks{The author is supported by the Spanish grants PID2021-123903NB-I00 and RED2022-134784-T funded by MCIN/AEI/10.13039/501100011033, by ERDF ``A way of making Europe", and by the Catalan grant 2021-SGR-00087. This work is supported by the Spanish State Research Agency, through the Severo Ochoa and Mar\'ia de Maeztu Program for Centers and Units of Excellence in R\&D (CEX2020-001084-M), and more specifically by the grant CEX2020-001084-M-20-1. The author acknowledges CERCA Programme/Generalitat de Catalunya for institutional support.}

\begin{document}

\begin{abstract}
We study the family of operators $\{\Rodzin_a\}_{a\in [0,+\infty)}$ associated to the Robin-type problems in a bounded domain $\Omega\subset\R^2$
\begin{equation}
    \begin{cases}
        -\Delta u = f & \text{in } \Omega, \\
        2 \bar \nu \partial_{\bar z} u + au = 0 & \text{on } \partial\Omega,
    \end{cases}
\end{equation}
and their dependency on the boundary parameter $a$ as it moves along $[0,+\infty)$. In this regard, we study the convergence of such operators in a resolvent sense. We also describe the eigenvalues of such operators and show some of their properties, both for all fixed $a$ and as functions of the parameter $a$. As shall be seen in more detail in the paper~\cite{DuranMasSanzPerela2025}, the eigenvalues of these operators characterize the positive eigenvalues of quantum dot Dirac operators.
\end{abstract}

\maketitle
\vspace{-1.1cm}
\tableofcontents

\section{Introduction} \label{sec:Introduction}

In this work, we study the operator associated to the problem 
\begin{equation} \label{eq:BvPRodzinLaplacian}
    \begin{cases}
        -\Delta u = f & \text{in } \Omega, \\
        2 \bar \nu \partial_{\bar z} u + au = 0 & \text{on } \partial\Omega,
    \end{cases}
\end{equation}
where $\Omega$ is a bounded $\mathcal C^2$ domain in $\R^2$, $f\in L^2(\Omega,\C)$ is prescribed, $a\geq0$ is a parameter, and the boundary term $2 \bar \nu \partial_{\bar z}$ involves some complex structure that we describe next. We give special interest to the study of its eigenvalues, and their dependency on the boundary parameter $a$.

We denote the outward normal unit vector to $\partial \Omega$ as $\nu$, and the tangential unit vector as $\tau$, so that $\{\nu, \tau\}$ is positively oriented. In Cartesian coordinates, $\nu=(\nu_1, \nu_2)$ and $\tau = (-\nu_2, \nu_1)$. The normal and tangential derivatives on $\partial \Omega$ shall be denoted $\partial_\nu$ and $\partial_\tau$, respectively. Namely, $\partial_\nu = \nu_1 \partial_1 + \nu_2 \partial_2$ and $\partial_\tau = \nu_1\partial_2 - \nu_2\partial_1$, where $(\partial_1,\partial_2)=\nabla$ denotes the gradient in $\R^2$.

We identify $\R^2\equiv\C$. In particular, we will make the identification $\R^2\ni\nu = (\nu_1,\nu_2)\equiv \nu_1 + i\nu_2\in\C$ and, accordingly, we write $\bar \nu =\nu_1 - i\nu_2$. We also use the complex notation $\partial_z := \frac{1}{2} (\partial_1 -i \partial_2)$ and $\partial_{\bar z} := \frac{1}{2} (\partial_1 +i \partial_2)$. With these identifications, $2 \bar \nu \partial_{\bar z} = \partial_\nu + i \partial_\tau$ and $2 \nu \partial_z = \partial_\nu - i \partial_\tau$, and thus the boundary condition in \eqref{eq:BvPRodzinLaplacian} can be seen as an oblique boundary condition; see, for example, \cite{Behrndt2023,Benhellal2025,Heriban2025} where oblique transmission conditions are studied for Schr\"odinger operators in $\R^2$.

To motivate the problem \eqref{eq:BvPRodzinLaplacian}, let us explore its similarity with the standard problem
\begin{equation} \label{eq:BvPRobinLaplacian}
    \begin{cases}
        -\Delta u = f & \text{in } \Omega, \\
        \partial_\nu u + au = 0 & \text{on } \partial\Omega,
    \end{cases}
\end{equation}
whose associated operator is the widely studied Robin Laplacian \cite{Antunes2013,Behrndt2010,Belgacem2018,Bogli2022,Bossel1986,Henrot2017Chapter4,Daners2006,Daners2007,Freitas2021}, defined\footnote{The reader may look at \Cref{sec:Notation}, where we recall the basic notation used throughout the paper.} by
\begin{equation} \label{eq:RobinLaplacian}
    \begin{split}
        \Dom(-\Delta_a) & := \left\{u\in H^2(\Omega):\,  \partial_\nu u + au = 0 \text{ in } H^{1/2}(\partial \Omega) \right\}, \\
        -\Delta_a u & := -\Delta u \quad \text{for all } u \in \Dom(-\Delta_a).
    \end{split}
\end{equation}
The weak formulation of \eqref{eq:BvPRobinLaplacian} is obtained from the equality
\begin{equation} \label{eq:weakMotivation}
    \int_\Omega f \, \overline v = \int_\Omega -\Delta u \, \overline v \quad \text{for all } v\in \mathcal C^\infty(\overline\Omega),
\end{equation}
when one decomposes the Laplacian as $\Delta = \mathrm{div}\nabla$, integrates by parts, and applies the boundary condition in \eqref{eq:BvPRobinLaplacian}. That is,
\begin{equation}
    \int_\Omega f \, \overline v = \int_\Omega -\mathrm{div}(\nabla u) \, \overline v = \int_\Omega \nabla u \cdot \overline{\nabla v} - \int_{\partial\Omega} \partial_\nu u \, \overline v = \int_\Omega \nabla u \cdot \overline{\nabla v} +a \int_{\partial\Omega} u \, \overline v \quad \text{for all } v\in \mathcal C^\infty(\overline\Omega).
\end{equation}
If in \eqref{eq:weakMotivation} one decomposes, instead, the Laplacian as $\Delta = 4\partial_z\partial_{\bar z}$ and integrates by parts, one obtains 
\begin{equation}
    \int_\Omega f \, \overline v = -4\int_\Omega \partial_z \partial_{\bar z} u \, \overline v = 4 \int_\Omega \partial_{\bar z} u \, \overline{\partial_{\bar z} v} - \int_{\partial\Omega} 2\overline \nu \partial_{\bar z} u \, \overline v \quad \text{for all } v\in \mathcal C^\infty(\overline\Omega).
\end{equation}
In order for the boundary integral to look like the one of the weak formulation of the Robin Laplacian, one must impose the boundary condition $2 \bar \nu \partial_{\bar z} u + au = 0$. This way, one obtains
\begin{equation} \label{eq:DivergenceToRodzin}
    \int_\Omega f \, \overline v = 4 \int_\Omega \partial_{\bar z} u \, \overline{\partial_{\bar z} v} +a \int_{\partial\Omega} u \, \overline v \quad \text{for all } v\in \mathcal C^\infty(\overline\Omega).
\end{equation}

Despite the similarity of the boundary problems \eqref{eq:BvPRodzinLaplacian} and \eqref{eq:BvPRobinLaplacian}, as far as we know the former is hardly studied in the literature; see point $(i)$ below. In contrast, the Robin problem \eqref{eq:BvPRobinLaplacian} is widely understood; see~\cite{Bogli2022,Henrot2017Chapter4} and the open problems therein. For example, it is well known that, for every $a>0$, the Robin Laplacian $-\Delta_a$ is self-adjoint in $L^2(\Omega)$ and has a purely discrete spectrum given by the min-max levels 
\begin{equation} \label{eq:firstRobinEigenvalue}
    \mu_{k}^{\mathrm{Rob}}(a) = \underset{\substack{ F\subset H^1(\Omega, \R) \\ \mathrm{dim}(F)=k }}{\inf} \, \, \underset{u\in F\setminus\{0\}}{\sup} \, \dfrac{\int_\Omega |\nabla u|^2 + a \int_{\partial\Omega} |u|^2}{\int_\Omega |u|^2}.
\end{equation}
Notice that these min-max levels can be taken minimizing among real-valued subspaces $F\subset H^1(\Omega, \R)$ because the Robin Laplacian decomposes into real and imaginary parts ---actually, one can straightforwardly see this from \eqref{eq:BvPRobinLaplacian}. This is not the case for the problem \eqref{eq:BvPRodzinLaplacian}, and thus one can not get rid of the complex structure inherent to it. 

This exhibits that \eqref{eq:BvPRodzinLaplacian} and \eqref{eq:BvPRobinLaplacian} are different problems. Hence, it is not evident whether the properties that the eigenvalues of the Robin Laplacian satisfy are analogously satisfied in the setting of \eqref{eq:BvPRodzinLaplacian}. In this work, we address this study. Our interest on \eqref{eq:BvPRodzinLaplacian} is twofold: 

\begin{enumerate}[label=$(\roman*)$]
    \item Firstly, it is an elliptic, non-coercive, complex valued PDE, a framework which is not quite standard in the literature. We are only aware of the work by Antunes, Benguria, Lotoreichik, and Ourmi\`eres-Bonafos \cite{Antunes2021}, which strongly inspires this work, and the work of Behrndt, Holzmann, and Stenzel, to appear in \cite{Behrndt2025}, of which we were aware while completing the present work. On the one hand, in \cite{Antunes2021} a non-linear version of our problem arises for studying the two-dimensional massless Dirac operator with infinite mass boundary conditions; see \Cref{sec:RelationRodzinDirac} for more details. On the other hand, in \cite{Behrndt2025} the problem
    \begin{equation} 
        \begin{cases}
            -\Delta u = f & \text{in } \Omega, \\
            \partial_\nu u +i \alpha \partial_\tau u + \beta u = 0 & \text{on } \partial\Omega,
        \end{cases}
    \end{equation}
    is studied distinguishing the cases $\alpha>1$, $\alpha=1$, and $\alpha \in [0,1)$, where $\beta\in \R$ is fixed, as then different qualitative spectral properties hold for the associated operator. Notice that in this work, we fix $\alpha=1$ and move $\beta\equiv a\in[0,+\infty)$. 
    \item Secondly, as shown in the paper \cite{DuranMasSanzPerela2025}, the eigenvalues of the problem~\eqref{eq:BvPRodzinLaplacian} turn out to characterize the positive eigenvalues of quantum dot Dirac operators. Using the family $\{\Rodzin_a\}_{a>0}$, we provide positive results supporting the conjecture that, among all $\mathcal C^2$ bounded domains with prescribed area, the first positive eigenvalue of the Dirac operator with infinite mass boundary conditions is minimal for a disk ---which is a hot open problem in spectral geometry \cite[Problem~5.1]{Webpage2019}. See \Cref{sec:RelationRodzinDirac} for a brief explanation of this. Actually, this connection with the quantum dot Dirac operators is the reason that motivated us to study the problem \eqref{eq:BvPRodzinLaplacian}, and originated the present work.
\end{enumerate}

\subsection{Notation} \label{sec:Notation}

In this section, we recall some basic notation regarding the Hilbert spaces and associated norms to be used throughout the paper. 

In the sequel, $\Omega$ denotes a bounded domain in 
$\R^2$ with $\mathcal C^2$ boundary. We denote by $L^2(\Omega)$ the Hilbert space of  functions $\varphi:\Omega\to\C$ endowed with the scalar product and the associated norm
\begin{equation}
    \langle \varphi,\psi\rangle_{L^2(\Omega)}:=\int_{\Omega} \varphi \, \overline\psi \,dx \quad\text{and}\quad
    \|\varphi\|_{L^2(\Omega)}:=\langle \varphi,\varphi\rangle_{L^2(\Omega)}^{1/2}.
\end{equation}
For $k\in \N$, we denote by $H^k(\Omega)$ the Sobolev space of functions in $L^2(\Omega)$ with weak partial derivatives up to order $k$ in $L^2(\Omega)$, and $H_0^k(\Omega)$ denotes the closure with respect to the $H^k(\Omega)-$norm of the set $\mathcal C^\infty_c(\Omega)$ of smooth functions compactly supported in $\Omega$. For $s\in(0,1)$, we denote by $H^s(\Omega)$ the fractional Sobolev space of functions $\varphi\in L^2(\Omega)$ such that 
    $$
        \|\varphi\|_{H^s(\Omega)}:= \Big( \int_{\Omega} |\varphi|^2 \, dx + \int_{\Omega}\int_{\Omega}\frac{|\varphi(x)-\varphi(y)|^2}{|x-y|^{2+2s}} \,dy \,dx \Big)^{1/2} < +\infty,
    $$
and we denote the continuous dual of $H^s(\Omega)$ by $H^{-s}(\Omega)$.

Similarly, $L^2(\partial\Omega)$ denotes the Hilbert space of functions $\varphi:\partial\Omega\to\C$ endowed with the scalar product and the associated norm
\begin{equation}
    \langle \varphi,\psi\rangle_{L^2(\partial\Omega)}:=\int_{\partial\Omega} \varphi \, \overline\psi \,d\upsigma \quad\text{and}\quad
    \|\varphi\|_{L^2(\partial\Omega)}:=\langle \varphi,\varphi\rangle_{L^2(\partial\Omega)}^{1/2},
\end{equation}
where $\upsigma$ denotes the surface (or arc-length) measure on $\partial\Omega$. For $s\in(0,1)$, we denote by $H^s(\partial \Omega)$ the fractional Sobolev space of functions $\varphi\in L^2(\partial\Omega)$ such that 
    $$
        \|\varphi\|_{H^s(\partial \Omega)}:= \Big( \int_{\partial\Omega} |\varphi|^2\,d\upsigma + \int_{\partial\Omega}\int_{\partial\Omega}\frac{|\varphi(x)-\varphi(y)|^2}{|x-y|^{1+2s}} \,d\upsigma(y)\,d\upsigma(x) \Big)^{1/2} < +\infty.
    $$
For the sake of simplifying the notation, we shall omit the measures of integration when no confusion arises.
    
The continuous dual of $H^s(\partial \Omega)$ is denoted by $H^{-s}(\partial \Omega)$. The action of $\varphi \in H^{-s}(\partial \Omega)$ on $\psi \in H^s(\partial \Omega)$ is denoted by the pairing $\langle \varphi, \psi \rangle_{H^{-s}(\partial\Omega), H^s(\partial\Omega)}$, and the norm in $H^{-s}(\partial \Omega)$ is
    $$
        \|\varphi\|_{H^{-s}(\partial \Omega)}:= 
        { \sup_{\|\psi\|_{H^s(\partial \Omega)}\leq1}} 
        \langle \varphi, \psi \rangle_{H^{-s}(\partial\Omega), H^s(\partial\Omega)} .
    $$
Recall that, whenever $\varphi \in L^2(\partial\Omega) \subset H^{-s}(\partial\Omega)$ and $\psi \in H^s(\partial\Omega) \subset L^2(\partial\Omega)$, the pairing satisfies
\begin{equation} \label{eq:Brezis}
    \langle \varphi, \psi \rangle_{H^{-s}(\partial\Omega), H^s(\partial\Omega)} = \overline{ \langle \varphi, \psi \rangle}_{L^2(\partial\Omega)};
\end{equation}
see, for example, \cite[Remark 3 in Section 5.2, and Section 11.4]{Brezis2011}. The reason why there is a complex conjugate in \eqref{eq:Brezis} is that we defined $\langle\cdot,\cdot\rangle_{L^2(\partial\Omega)}$ to be linear with respect to the first entry.

We conclude this section introducing some notation for the spectrum and its decomposition. The spectrum of a self-adjoint operator $A$ will be denoted by $\sigma(A)$. The point spectrum of~$A$ ---namely, the subset of $\sigma(A)$ consisting of eigenvalues of~$A$--- will be denoted by $\sigma_{\mathrm{p}}(A)$. The discrete spectrum of~$A$ ---namely, the subset of $\sigma_{\mathrm{p}}(A)$ consisting of isolated points in $\sigma(A)$ with finite multiplicity--- will be denoted by $\sigma_{\mathrm{d}}(A)$. The essential spectrum of~$A$ ---namely, the complement of $\sigma_{\mathrm{d}}(A)$ in $\sigma(A)$--- will be denoted by $\sigma_{\mathrm{ess}}(A)$.

\subsection{Qualitative behavior when the domain is a disk} \label{sec:QualitativeDisk}

We present here a brief summary of the spectral study of the operator associated to the problem \eqref{eq:BvPRodzinLaplacian} in the case in which the underlying domain $\Omega$ is a disk $D_R\subset \R^2$ of radius $R>0$ centered at the origin; by the translation invariance of the problem \eqref{eq:BvPRodzinLaplacian}, the position of the center of the disk is irrelevant.

In this radially symmetric domain, we can use polar coordinates and separation of variables to obtain equations for the eigenvalues and eigenfunctions of \eqref{eq:BvPRodzinLaplacian}, that is, for the solutions $(\mu,u)$ of
\begin{equation} \label{eq:eigenProblemDisk}
    \begin{cases}
        -\Delta u = \mu u & \text{in } D_R, \\
        2 \bar \nu \partial_{\bar z} u + au = 0 & \text{on } \partial D_R.
    \end{cases}
\end{equation}
The analysis done for the disk provides some intuition on which kind of situations one can expect when studying the operator associated to the problem \eqref{eq:BvPRodzinLaplacian} in a general bounded $\mathcal C^2$ domain $\Omega\subset \R^2$, for which no explicit formulas are available. A more detailed analysis including the proofs of the facts stated in this section can be found in \Cref{sec:EigenBall}. 

Using polar coordinates, the eigenvalue problem \eqref{eq:eigenProblemDisk} can be reduced to a Bessel-type ODE. After imposing the boundary conditions, we obtain the eigenvalue equations
\begin{equation} \label{eq:EigenEquationsDisk}
    \text{either} \quad a = \sqrt \mu \frac{J_{j+1}(\sqrt \mu R)}{J_j(\sqrt \mu R)}, \quad \text{or} \quad a + \frac{2j}{R} = \sqrt \mu \frac{J_{j+1}(\sqrt \mu R)}{J_j(\sqrt \mu R)}, \quad \text{for } j\in \N\cup\{0\},
\end{equation}
where $J_j$ is the $j-$th Bessel function of the first kind. If we solve numerically \eqref{eq:EigenEquationsDisk}, we can obtain $\mu$ in terms of $a$. Representing this graphically we obtain \Cref{fig:EigenRodzinBall}.

\vspace{-0.3cm}
\begin{figure}[h]
    \centering
    \includegraphics[width=0.65\linewidth]{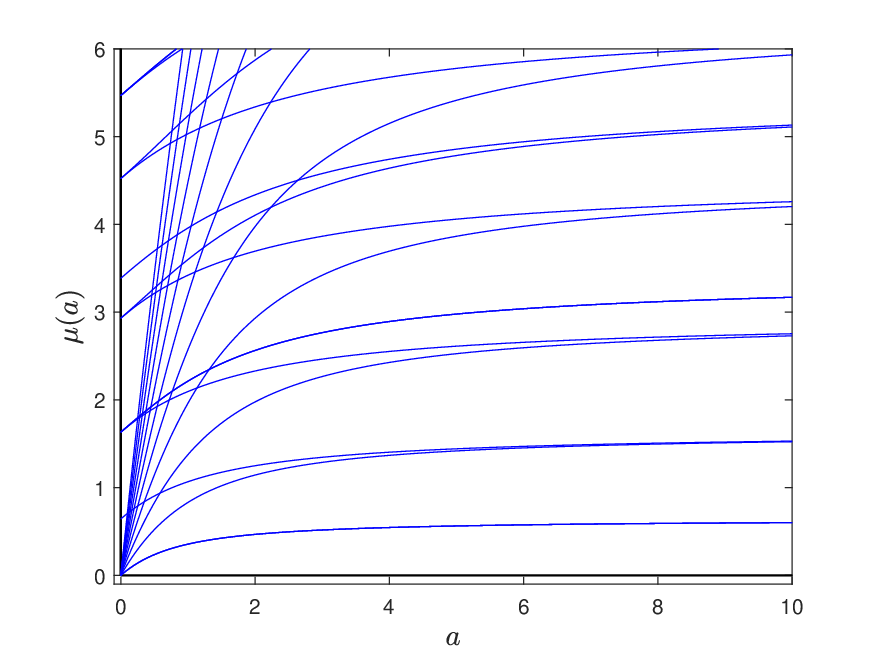}
    \vspace{-0.4cm}
    \caption{Some eigenvalue curves $a\mapsto \mu(a)$ on $D_R$ for $R=3$.}
    \label{fig:EigenRodzinBall}
\end{figure}
\vspace{-0.2cm}
As we see, the eigenvalues can be parametrized in terms of $a$, obtaining a family of increasing, bounded functions $a\mapsto \mu(a)$. These are the so-called \emph{eigenvalue curves}. Actually, in \Cref{prop:EigenDisk} we show that the right hand sides in \eqref{eq:EigenEquationsDisk} are functions of $\mu$ invertible in suitable intervals that are related with the zeros of the Bessel function appearing in the denominator, and the inverses are such eigenvalue curves.

In addition, the limit of any $\mu(a)$ as $a\to+\infty$ is always a positive zero of $J_j(\sqrt{\cdot} R)$ for some $j\in \N\cup\{0\}$. We recall that these are the eigenvalues of the Dirichlet Laplacian in $D_R$ ---the definition of this operator is recalled in \eqref{eq:OpDirichletLaplacian}. On the other hand, the limit of the functions $\mu(a)$ as $a\to 0^+$ is always either one such eigenvalue of the Dirichlet Laplacian in $D_R$, or zero. Actually, we see that several eigenvalue curves converge to $0$ as $a\to 0^+$.

In summary, all the eigenvalues of \eqref{eq:eigenProblemDisk} in $D_R$ can be represented as a set of monotone increasing curves parametrized by $a\in(0,+\infty)$, which may cross among them; see \Cref{fig:EigenRodzinBall}. For a given curve $a\mapsto \mu(a)$, we know which are the possible limits of $\mu(a)$ as $a\to 0^+$ and $a\to+\infty$: with the only exception of some eigenvalues that converge to $0$ as $a\to 0^+$, the limiting values of $\mu(a)$ are the eigenvalues of the Dirichlet Laplacian in $D_R$.

At least for the disk, these properties suggest the existence of a self-adjoint operator $\Rodzin_a$ associated to the problem \eqref{eq:BvPRodzinLaplacian}, and the existence of some limiting operators $\Rodzin_0$ and $\Rodzin_{+\infty}$ to which $\Rodzin_a$ converge as $a\to 0^+$ and $a\to +\infty$, respectively. 

In this paper we show that, for every bounded domain $\Omega\subset \R^2$ with $\mathcal C^2$ boundary, this is indeed the case: in \Cref{thm:IntroRodzinLaplacian} we describe the operator $\Rodzin_a$, in \Cref{thm:NRCDirichletLaplacian} we show that the limiting operator $\Rodzin_{+\infty}$ to which $\Rodzin_a$ converges as $a\to+\infty$ is the Dirichlet Laplacian, in \Cref{prop:SpectrumRodzin0} we describe the operator $\Rodzin_0$, and in \Cref{thm:SRCRodzin0} we show the convergence of~$\Rodzin_a$ to $\Rodzin_0$ as $a\to0^+$. We also show that the spectral properties that we have qualitatively seen from \Cref{fig:EigenRodzinBall} for the disk also hold for such domains. Actually, we characterize in \Cref{thm:PropertiesMukOmega} the ordered eigenvalues of $\Rodzin_a$ in terms of min-max levels of a functional linear in the boundary parameter $a$, and qualitatively describe their behaviour as $a$ moves in \Cref{thm:PropertiesMukOmegaGlobal}.

We conclude this section justifying the restriction $a\geq 0$. In the case in which the underlying domain $\Omega$ is a disk, analogous explicit computations as in \Cref{sec:EigenBall} can be performed when $a<0$, but then the plot of the eigenvalues that we obtain is the one of \Cref{fig:EigenRodzinBallNegative}.

\vspace{-0.3cm}
\begin{figure}[h]
    \centering
    \includegraphics[width=0.65\linewidth]{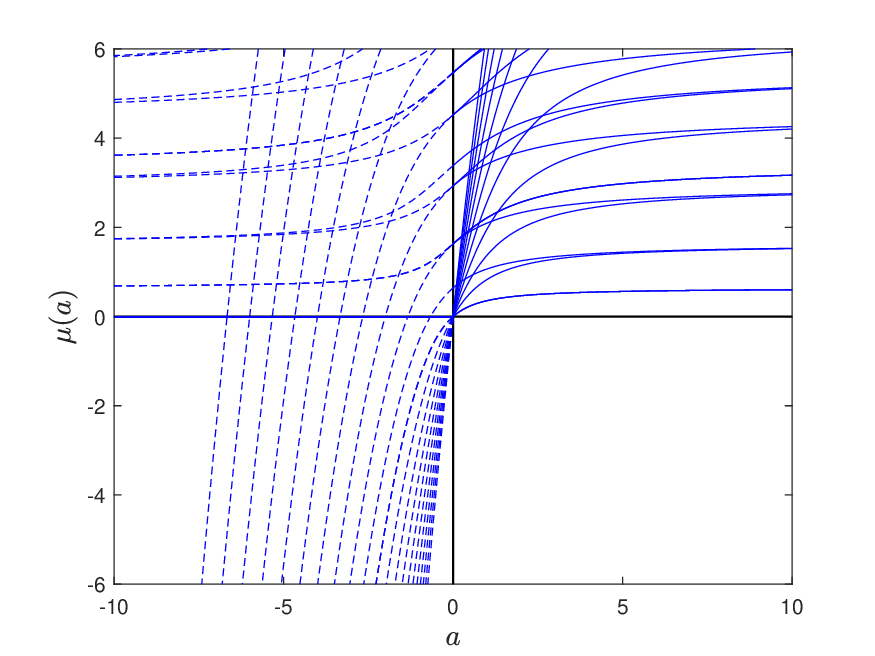}
    \vspace{-0.4cm}
    \caption{Some eigenvalue curves $a\mapsto \mu(a)$ on $D_R$ for $R=3$. The dashed curves correspond to $a<0$. The continuous curves are the ones of \Cref{fig:EigenRodzinBall}.}
    \label{fig:EigenRodzinBallNegative}
\end{figure}
\vspace{-0.2cm}
As we see, the qualitative behaviour is very different from the case in which $a>0$. Actually, \Cref{fig:EigenRodzinBallNegative} seems to suggest that the sequence of eigenvalues in the regime $a<0$ is not bounded from below. This is precisely the content of \Cref{rmk:Negativea}, and exhibits once more that the situation is very different from the standard Robin Laplacian. For this reason, we restrict to $a\in[0,+\infty)$.

\subsection{Main results} \label{sec:MainResults}

From now on, we shall assume that $\Omega$ is a $\mathcal C^2$ bounded domain in $\R^2$. In this section, we state our main results on the operator associated to the problem \eqref{eq:BvPRodzinLaplacian}, its spectral properties, and its behaviour as $a$ moves in $[0,+\infty)$.

The first main result is, indeed, the existence of a unique self-adjoint operator associated to the problem \eqref{eq:BvPRodzinLaplacian} ---for a precise description of such association, see \Cref{sec:AssociatedOperator}. The first part of the following theorem is proven in \Cref{sec:AssociatedOperator}, and the second part is proven in \Cref{sec:LocalSpectralProperties}.

\begin{theorem} \label{thm:IntroRodzinLaplacian}
    For every $a>0$, the operator
    \begin{equation} \label{eq:RodzinLaplacian}
        \begin{split}
            \Dom(\Rodzin_a) & := \left\{u\in H^1(\Omega): \, \partial_{\bar z} u \in H^1(\Omega), \, 2\bar \nu \partial_{\bar z}u + au = 0 \text{ in } H^{1/2}(\partial \Omega) \right\}, \\
            \Rodzin_a u & := -\Delta u \quad \text{for all } u \in \Dom(\Rodzin_a),
        \end{split}
    \end{equation}
    is the unique self-adjoint operator such that, for every $f\in L^2(\Omega)$, there exists a unique function in its domain solving \eqref{eq:BvPRodzinLaplacian} in the weak sense.

    Moreover, for every $\lambda \in \C\setminus\R$, the resolvent $(\Rodzin_a-\lambda)^{-1}$ is a bounded operator from $L^2(\Omega)$ to $H^1(\Omega)$. As a consequence, $(\Rodzin_a-\lambda)^{-1}$ is a compact operator from $L^2(\Omega)$ to $L^2(\Omega)$, and $\sigma(\Rodzin_a)$ is purely discrete.
\end{theorem}

This operator is the counterpart of the one described in \cite[Propostition 31]{Antunes2021}, which has quadratic dependence on the parameter; see \Cref{sec:RelationRodzinDirac} for more details. Since the operator $\Rodzin_a$ will be the core of study of this work, we give it a name: the operator $\Rodzin_a$ defined in \eqref{eq:RodzinLaplacian} is called the \textbf{$\overline\partial$-Robin Laplacian}.

The following theorem gives a description of the eigenvalues $\{\mu_k(a)\}_{k\in\N}$ of $\Rodzin_a$ in terms of min-max levels. The space $E(\Omega)$ appearing throughout the theorem is the Hilbert space
\begin{equation} \label{eq:AmbientHilbertSpace}
    E(\Omega) := \left\{ u \in L^2(\Omega): \, \partial_{\bar z}u \in L^2(\Omega), \, u\in L^2(\partial\Omega) \right\},
\end{equation}
which we shall study in detail in \Cref{sec:HSandSFsetting}.

\begin{theorem} \label{thm:PropertiesMukOmega}
    For every $a>0$, the ordered eigenvalues of $\Rodzin_a$ (repeated according to their multiplicity) are
    \begin{equation} \label{eq:mukOmega}
        \mu_{k}(a) = \underset{\substack{ F\subset E(\Omega) \\ \mathrm{dim}(F)=k }}{\inf} \, \, \underset{u\in F\setminus\{0\}}{\sup} \, \dfrac{4\int_\Omega |\partial_{\bar z} u|^2 + a \int_{\partial\Omega} |u|^2}{\int_\Omega |u|^2}, \quad \text{where } k\in \N.
    \end{equation}
    In particular, $\mu_{k}(a)\in (0, \Lambda_{k})$, where $\Lambda_{k}$ is the $k-$th eigenvalue of the Dirichlet Laplacian. In addition, the min-max level \eqref{eq:mukOmega} is attained by a function $u_{k}(a)\in E(\Omega)$, which is an eigenfunction of $\Rodzin_a$ with eigenvalue $\mu_{k}(a)$. Moreover, $\{u_{k}(a)\}_{k\in \N}$ can be chosen to be pairwise orthogonal in $L^2(\Omega)$.
\end{theorem}

The reader can find in \Cref{sec:LocalSpectralProperties} a constructive proof of \Cref{thm:PropertiesMukOmega} based on the min-max~Theorem appearing in~\cite[Theorem 4.14]{Teschl2014}; see also \cite[Theorem 4.5.1]{Davies1995}. Such version of the min-max Theorem yields that the ordered eigenvalues of $\Rodzin_a$ are given by the min-max levels 
\begin{equation}
    \underset{\substack{ F\subset \mathrm{Dom}(\Rodzin_a) \\ \mathrm{dim}(F)=k }}{\inf} \, \, \underset{u\in F\setminus\{0\}}{\sup} \, \dfrac{\langle \Rodzin_a u, u \rangle_{L^2(\Omega)}}{\int_\Omega |u|^2} = \underset{\substack{ F\subset \mathrm{Dom}(\Rodzin_a) \\ \mathrm{dim}(F)=k }}{\inf} \, \, \underset{u\in F\setminus\{0\}}{\sup} \, \dfrac{4\int_\Omega |\partial_{\bar z} u|^2 + a \int_{\partial\Omega} |u|^2}{\int_\Omega |u|^2},
\end{equation}
where the infimum is taken among subspaces $F$ of $\mathrm{Dom}(\Rodzin_a)$. \Cref{thm:PropertiesMukOmega} evokes a stronger version of the min-max Theorem ---see \cite[Theorem 4.5.3]{Davies1995}---, because the infimum in \eqref{eq:mukOmega} is taken among subspaces $F$ of $E(\Omega) \supsetneq \mathrm{Dom}(\Rodzin_a)$, and $E(\Omega)$ is independent of $a$ ---unlike~$\mathrm{Dom}(\Rodzin_a)$. This will be convenient for compactness arguments. Alternatively, \Cref{thm:PropertiesMukOmega} can be obtained from \cite[Lemma 4.4.1, Corollary 4.4.3, and Theorem 4.5.3]{Davies1995} in combination with \Cref{prop:RodzinLaplacian} below. We thank the anonymous referee for bringing this fact to our attention.

We point out that the min-max level $\mu_k(a)$ of \Cref{thm:PropertiesMukOmega} looks similar to the min-max level $\mu_{k}^{\mathrm{Rob}}(a)$ in \eqref{eq:firstRobinEigenvalue} of the Robin Laplacian. Actually, changing in \eqref{eq:mukOmega} $2\partial_{\bar z} u$ by $\nabla u$ and $E(\Omega)$ by $H^1(\Omega,\R)$, we get \eqref{eq:firstRobinEigenvalue}.

\Cref{thm:PropertiesMukOmega} gives a description of the eigenvalues $\mu_{k}(a)$ of $\Rodzin_a$, for every fixed $a>0$. The following theorem exhibits the qualitative properties that the functions $a \mapsto \mu_{k}(a)$ satisfy; see \Cref{sec:GlobalSpectralProperties} for a proof.

\begin{theorem} \label{thm:PropertiesMukOmegaGlobal}
    For every $a>0$, let $\mu_k(a)$ be the $k-$th eigenvalue of $\Rodzin_a$ (repeated according to the multiplicity). Then, as a function of $a$, for every $k\in \N$ the following holds:
    \begin{enumerate}[label=$(\roman*)$]
        \item \label{item:ContIncBij} $a\mapsto \mu_{k}(a)$ is continuous, strictly increasing, and bijective from $(0,+\infty)$ to $(0,\Lambda_{k})$, where $\Lambda_{k}$ is the $k-$th eigenvalue of the Dirichlet Laplacian in $\Omega$.
        \item \label{item:Analytic} $a\mapsto \mu_{k}(a)$ is piecewise analytic in $(0,+\infty)$, and the set of points $X_k$ where the function is not analytic is finite in every compact subset of $(0,+\infty)$. In addition, there exists a function $u_{k}:(0,+\infty)\to E(\Omega)$ that is continuous in $(0,+\infty)\setminus X_k$ with respect to the $E(\Omega)-$norm and satisfies
        \begin{equation}
            \mu_{k}(a) = \dfrac{4\int_\Omega |\partial_{\bar z} u_{k}(a)|^2 + a \int_{\partial\Omega} |u_{k}(a)|^2}{\int_\Omega |u_{k}(a)|^2} \quad \text{for every } a\in(0,+\infty).
        \end{equation}
        \item \label{item:Derivative} If $a\in (0,+\infty)\setminus X_k$ is a point where $\mu_{k}(a)$ is analytic and $u_{k}(a)$ is as in $(ii)$, then
        \begin{equation} 
            \mu_{k}'(a) := \dfrac{d}{da} \mu_{k}(a) = \frac{\int_{\partial\Omega} |u_{k}(a)|^2}{\int_\Omega |u_{k}(a)|^2} > 0.
        \end{equation}
        \item \label{item:Concavity} For $k=1$, $a\mapsto \mu_{1}(a)$ is strictly concave in $(0,+\infty)$.
    \end{enumerate}
\end{theorem}

Some remarks are in order about properties that hold for the Robin Laplacian, which one could expect to hold as well for the $\overline\partial$-Robin Laplacian in view of their similarity, but that are not stated in \Cref{thm:PropertiesMukOmegaGlobal} and will not be addressed in the present work.

\begin{remark} \label{rmk:simplicity}
    In this work we do not address the question whether the first eigenvalue $\mu_{1}(a)$ is simple or not, which a priori does not seem an easy task as with the Robin Laplacian. However, when the underlying domain is a disk, we can ensure that the first eigenvalue is simple. Actually, from the explicit computations of \Cref{sec:EigenBall}, we can show that the first eigenvalue of the $\overline\partial$-Robin Laplacian $\Rodzin_a$ coincides with the first eigenvalue of the Robin Laplacian $-\Delta_a$, and that the associated eigenfunction for $\Rodzin_a$ is (up to a multiplicative constant) the same as the first eigenfunction for $-\Delta_a$; this is the content of \Cref{cor:MuDiskIsMuRobin} in \Cref{sec:EigenBall}.

    The basic reason why this holds true is that the first eigenfunction of the $\overline\partial$-Robin Laplacian in a disk $D$ is constant on the boundary. Denoting it by $u$, it therefore holds that $\partial_\tau u =0$ on~$\partial\Omega$. As a consequence, the boundary condition for $u$ writes 
    \begin{equation}
        0 = \partial_\nu u + i\partial_\tau u + au = \partial_\nu u + au \quad  \text{on } \partial \Omega,
    \end{equation}
    and hence coincides with the Robin boundary condition. In an arbitrary domain $\Omega$, the first eigenfunction of the Robin Laplacian $v$ may not satisfy $\partial_\tau v = 0$ on $\partial\Omega$, and hence it may not satisfy the boundary condition of the $\overline\partial$-Robin Laplacian,
    \begin{equation}
        \partial_\nu v + i\partial_\tau v + av = 0 \quad  \text{on } \partial \Omega.
    \end{equation}
    Actually, since $H^1(\Omega, \R)\subsetneq E(\Omega)$, it holds that the first eigenvalue of the $\overline\partial$-Robin Laplacian is smaller or equal than the first eigenvalue of the Robin Laplacian, and in general one should expect a strict inequality.
\end{remark}

\begin{remark} \label{rmk:concavity}
    Although, as asserted in \cite[Section 4.3.2]{Henrot2017Chapter4}, all the ordered eigenvalues of the Robin Laplacian are locally concave almost everywhere in $(0,+\infty)$ ---and the same seems to hold for the $\overline\partial$-Robin Laplacian in a disk in view of \Cref{fig:EigenRodzinBall}---, we have numerical evidence that this is not in general the case for the rest of eigenvalues $\mu_{k}(a)$, $k\geq 2$, of the $\overline\partial$-Robin Laplacian. Actually, solving numerically the eigenvalue equations derived analogously as in \Cref{sec:EigenBall} when the underlying domain $\Omega$ is an annulus, we obtain the plot shown in \Cref{fig:EigenRodzinAnnulus}, from where we see that some of the eigenvalue curves are convex in some regions.
\end{remark}
\vspace{-0.55cm}
\begin{figure}[H]
    \centering
    \includegraphics[width=0.6\linewidth]{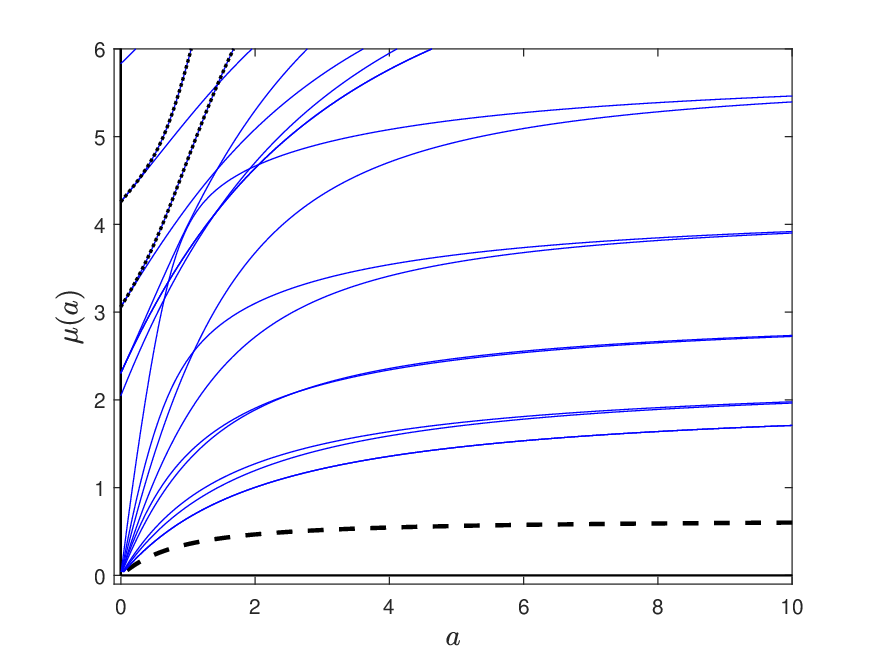}
    \vspace{-0.4cm}
    \caption{Some eigenvalue curves $a\mapsto \mu(a)$ on an annulus of inner radius $1$ and outer radius $\sqrt{10}$. The black pointed curves are not locally concave. The black dashed curve corresponds to the first eigenvalue curve on the disk of same area.}
    \label{fig:EigenRodzinAnnulus}
\end{figure}
\vspace{-0.35cm}
\begin{remark}
    In this work we do not address the question whether a Faber-Krahn type inequality for $\mu_{1}(a)$ holds for every $a>0$ or not. This should be more difficult than proving the Faber-Krahn inequality for the first eigenvalue of the Robin Laplacian, namely, the so-called \emph{Bossel-Daners-Kennedy inequality} \cite[Theorem 1.1]{Daners2007} ---first established in two dimensions by Bossel~\cite{Bossel1986}, and then in any dimension by Daners and Kennedy~\cite{Daners2006,Daners2007}. Indeed, if a Faber-Krahn type inequality for $\mu_{1}(a)$ was true for every $a>0$, by \Cref{cor:MuDiskIsMuRobin} and \Cref{rmk:simplicity} the Bossel-Daners-Kennedy inequality would follow from it. In \Cref{sec:RelationRodzinDirac}, we will briefly comment how in the paper \cite{DuranMasSanzPerela2025} we relate such a Faber-Krahn type inequality for $\mu_{1}(a)$ with the analogous question for the first positive eigenvalue of quantum dot Dirac operators. Further shape optimization problems for the eigenvalues of the $\overline\partial$-Robin Laplacian in the spirit of the ones conjectured in \cite{Antunes2024,Briet2022,Laugesen2019} would also be interesting to be addressed in the future.
\end{remark}

Our last main results concern the convergence of the operator $\Rodzin_a$ as the boundary parameter moves in $(0,+\infty)$. First, we address the limit as $a\to+\infty$. From the boundary condition $2\bar \nu \partial_{\bar z}u + au = 0$ in $H^{1/2}(\partial \Omega)$ one could heuristically expect that the trace of every $u\in \Dom(\Rodzin_a)$ tends to zero in $H^{1/2}(\partial\Omega)$ as $a\to+\infty$. In such case, the operator formally obtained taking this limit in the definition of $\Rodzin_a$ would be the Dirichlet Laplacian $-\Delta_D$, which we recall to be
\begin{equation} \label{eq:OpDirichletLaplacian}
    \begin{split}
        \Dom(-\Delta_D) & := H^2(\Omega)\cap H^1_0(\Omega), \\
        -\Delta_D u & := -\Delta u \quad \text{for all } u \in \Dom(-\Delta_D).
    \end{split}
\end{equation}
The following result, a proof of which shall be given in \Cref{sec:NRCDirichlet}, formalizes these heuristics.

\begin{theorem} \label{thm:NRCDirichletLaplacian}
    The $\overline\partial$-Robin Laplacian $\Rodzin_a$ converges to the Dirichlet Laplacian $-\Delta_D$ in the norm resolvent sense, as $a\to+\infty$. As a consequence, $\lim_{a\to+\infty} \sigma(\Rodzin_a) = \sigma(-\Delta_D)$, in the sense that
    \begin{enumerate}[label=$(\roman*)$]
        \item any $\Lambda \in \sigma(-\Delta_D)$ is the limit, as $a\to+\infty$, of some $\mu(a) \in \sigma(\Rodzin_a)$, and 
        \item if a sequence of $\mu(a) \in \sigma(\Rodzin_a)$ converges to some value $\Lambda$ as $a\to+\infty$, then $\Lambda \in \sigma(-\Delta_D)$.
    \end{enumerate}
    In particular, the $k-$th eigenvalue of $\Rodzin_a$ converges to the $k-$th eigenvalue of $-\Delta_D$, as $a\to +\infty$.
\end{theorem}

The proof of \Cref{thm:NRCDirichletLaplacian} can be adapted without difficulties to show that $\Rodzin_a$ converges to $\Rodzin_{a_0}$ in the norm resolvent sense as $a\to a_0$, for every $a_0>0$. This is done in \Cref{prop:NRCFiniteValues}.

Finally, we address the limit as $a\to 0^+$. From the boundary condition $2\bar \nu \partial_{\bar z}u + au = 0$ in $H^{1/2}(\partial \Omega)$ one could heuristically expect that $\partial_{\bar z}u$ tends to vanish in $H^{1/2}(\partial\Omega)$ for every $u\in \Dom(\Rodzin_a)$, as $a\to 0^+$. In such case, the operator formally obtained taking this limit in the definition of $\Rodzin_a$ would be
\begin{equation} \label{eq:Rodzin0}
    \begin{split}
        \Dom(\Rodzin_0) & := \left\{ u\in L^2(\Omega): \, \partial_{\bar z} u \in H^1_0(\Omega) \right\}, \\
        \Rodzin_0 u & := -\Delta u \quad \text{for all } u \in \Dom(\Rodzin_0).
    \end{split}
\end{equation}
\Cref{thm:SRCRodzin0} below will formalize these heuristics. We give the operator $\Rodzin_0$ a name: the operator $\Rodzin_0$ defined in \eqref{eq:Rodzin0} is called the \textbf{$\overline\partial$-Neumann Laplacian}. 

The next result, a proof of which can be found in \Cref{sec:NRCDirichlet}, gives some properties of its spectrum that we shall need afterward. The Bergman space $A^2(\Omega)$ appearing in the statement is studied in \Cref{sec:BergmanSpace}.

\begin{proposition} \label{prop:SpectrumRodzin0}
    The operator $\Rodzin_0$ defined in \eqref{eq:Rodzin0} is self-adjoint in $L^2(\Omega)$. Its spectrum consists of eigenvalues, hence $\sigma(\Rodzin_0) = \sigma_{\mathrm{p}}(\Rodzin_0)$. More specifically,  $\sigma(\Rodzin_0) = \{0\} \cup \sigma(-\Delta_D)$, and the following spectral decomposition holds:
    \begin{enumerate}[label=$(\roman*)$]
        \item The eigenvalue $0$ is of infinite multiplicity. Actually, $\sigma_{\mathrm{ess}}(\Rodzin_0) = \{0\}$ and the set of eigenfunctions of $\Rodzin_0$ of eigenvalue $0$ is the Bergman space $A^2(\Omega)$, that is,
        \begin{equation} \label{eq:BergmanSpace}
            \ker(\Rodzin_0) = A^2(\Omega) := \left\{u\in L^2(\Omega): \, \partial_{\bar z} u = 0 \text{ in } \Omega \right\}.
        \end{equation}
        \item Every eigenvalue of $\Rodzin_0$ different from $0$ is of finite multiplicity. Actually, $\sigma_{\mathrm{d}}(\Rodzin_0) = \sigma(-\Delta_D)$ and the multiplicity of every $\Lambda \in \sigma(\Rodzin_0)\setminus\{0\}$ as an eigenvalue of $\Rodzin_0$ coincides with its multiplicity as an eigenvalue of $-\Delta_D$. In addition, the set of eigenfunctions of $\Rodzin_0$ with eigenvalue $\Lambda\in \sigma(\Rodzin_0)\setminus\{0\}$ of multiplicity $m$ is
        \begin{equation}
            \mathrm{span} \{\partial_z v_j\}_{j=1}^m, 
        \end{equation}
        where $\{v_j\}_{j=1}^m$ are $m$ orthogonal eigenfunctions of $-\Delta_D$ with eigenvalue $\Lambda$.
    \end{enumerate}
\end{proposition}

Notice that the existence of the essential eigenvalue $0\in \sigma_{\mathrm{ess}}(\Rodzin_0)$ prevents the operators $\Rodzin_a$ from converging in the norm resolvent sense towards $\Rodzin_0$, as $a\to0^+$. Indeed, if there was norm resolvent convergence, we would get that $\lim_{a\to0^+} \sigma_{\mathrm{ess}}(\Rodzin_a) = \sigma_{\mathrm{ess}}(\Rodzin_0)$ by \cite[Satz 9.24]{Weidmann2000}, but this is impossible since $\sigma_{\mathrm{ess}}(\Rodzin_0) = \{0\}$ and $\sigma_{\mathrm{ess}}(\Rodzin_a) = \varnothing$ for every $a>0$.

Once norm resolvent convergence is discarded, it is natural to ask whether the convergence in the strong resolvent sense holds or not. If it does, then one could wonder if the existence of the essential eigenvalue $0$ is the only obstruction to having norm resolvent convergence. The following result answers these questions in the affirmative.  

\begin{theorem} \label{thm:SRCRodzin0}
    The $\overline\partial$-Robin Laplacian $\Rodzin_a$ converges to the $\overline\partial$-Neumann Laplacian $\Rodzin_0$ in the strong resolvent sense, as $a\to0^+$. As a consequence, for every $\Lambda \in \{0\} \cup \sigma(-\Delta_D)$ there exist $\mu(a) \in \sigma(\Rodzin_a)$ such that $\mu(a) \to \Lambda$ as $a\to0^+$. 
    
    Moreover, if the resolvents of $\Rodzin_a$ and $\Rodzin_0$ are projected onto the orthogonal of the Bergman space $A^2(\Omega)$, then convergence in the operator norm holds. That is, 
    \begin{equation}
        \underset{a\to 0^+}{\lim} \, \left \| P^\perp \left( (\Rodzin_0-\lambda)^{-1} - (\Rodzin_a - \lambda)^{-1} \right) \right\|_{L^2(\Omega)\to L^2(\Omega)} = 0 \quad \text{for all } \lambda\in \C\setminus \R,
    \end{equation}
    where $P^\perp: L^2(\Omega) \rightarrow A^2(\Omega)^\perp \subset L^2(\Omega)$ is the orthogonal projection onto
    \begin{equation}
        A^2(\Omega)^\perp := \left \{ u\in L^2(\Omega): \, \langle u, v \rangle_{L^2(\Omega)} = 0 \text{ for all } v \in A^2(\Omega) \right\}.
    \end{equation}
\end{theorem}

The proof of this theorem is given in \Cref{sec:NRCDirichlet}, following the same ideas as in the proof of \cite[Theorem 1.4]{DuranMas2024}.

\subsection{Relation between the $\overline\partial$-Robin Laplacian and quantum dot Dirac operators} \label{sec:RelationRodzinDirac}

As mentioned at the beginning of the introduction, one of our interests of studying the operators $\{\Rodzin_a\}_{a\in(0,+\infty)}$ is that their eigenvalues characterize the positive eigenvalues of quantum dot Dirac operators. 

These are operators with confining boundary conditions, used in relativistic quantum mechanics to describe excitations in graphene quantum dots and nano-ribbons. Their energies are modeled by the eigenvalues $\lambda= \lambda(\theta,m)$ of the problem
\begin{equation} \label{eq:EigenDirac}
    \begin{cases}
        -2i \partial_z v = (\lambda-m) u & \text{in } \Omega, \\
        -2i \partial_{\bar z} u = (\lambda+m)v & \text{in } \Omega, \\
        v = i \frac{1-\sin \theta}{\cos \theta} \nu u & \text{on } \partial \Omega.
    \end{cases}
\end{equation}
Although one can make sense of \eqref{eq:EigenDirac} for all $\theta\in\R$ and $m\in\R$, which are the parameters of the model, here we restrict ourselves to $\theta \in (-\pi/2, \pi/2)$ and $m\geq 0$; see the introduction of~\cite{DuranMasSanzPerela2025} for more details. The case $m=0$ and $\theta=0$ is of special interest in physics, and refers to the Dirac operator with infinite mass boundary conditions. A hot open problem in spectral geometry is to prove ---or disprove--- that among all planar domains of given area, the disk is the unique minimizer for the smallest positive eigenvalue $\lambda_{1}(0,0)$ of \eqref{eq:EigenDirac} \cite[Problem 5.1]{Webpage2019}. 

This problem was tackled in \cite{Antunes2021}, and the argumentation used there strongly inspires this work. The idea in \cite{Antunes2021} is to translate \cite[Problem 5.1]{Webpage2019} into a Faber-Krahn type inequality, for every $a>0$, for the first min-max level of the operator $\Rodzin_a-a^2$; see \cite[Section 7]{Antunes2021}. The quadratic term $-a^2$ appears naturally in \cite{Antunes2021} when performing this translation, but turns out to conceal some properties that we show in the present work. Motivated by the results of this work, we rewrite \cite[Conjecture 40]{Antunes2021} in our setting. In order to stress the dependence of $\mu_1(a)$ on the domain $\Omega$, let us denote it for the moment by $\mu_{1,\Omega}(a)$.

\begin{conjecture} \label{Conj:Rodzin}
    Let $\Omega\subset\R^2$ be a bounded domain with $\mathcal C^2$ boundary. Let $D\subset\R^2$ be a disk with the same area as $\Omega$. If $\Omega$ is not a disk, then $\mu_{1,D}(a) < \mu_{1,\Omega}(a)$ for all $a>0$.
\end{conjecture}

In the paper \cite{DuranMasSanzPerela2025} it is shown how the positive eigenvalues $\lambda(\theta,m)$ of \eqref{eq:EigenDirac} are characterized, for every $m\geq0$, by the eigenvalues of $\Rodzin_a$, and vice versa. More interestingly, in the paper~\cite{DuranMasSanzPerela2025} it is shown that if $\Omega$ is a bounded $\mathcal C^2$ domain for which \Cref{Conj:Rodzin} holds true at some fixed $a>0$, then $\lambda_{1}(\theta,m)$ is lower for a disk with the same area as $\Omega$, for some~$\theta\in (-\pi/2, \pi/2)$ depending on $a$, $m$, and $\Omega$; and vice versa. As a consequence, if \Cref{Conj:Rodzin} holds true, then a Faber-Krahn type inequality holds for $\lambda_{1}(\theta,m)$ for every $\theta\in (-\pi/2, \pi/2)$ and every $m\geq0$.

A positive result supporting \Cref{Conj:Rodzin} follows from the convergence to the Dirichlet Laplacian; see \Cref{thm:PropertiesMukOmegaGlobal}. Indeed, as a consequence of the fact that $\mu_{1,\Omega}(a)$ increasingly converges to the first eigenvalue of the Dirichlet Laplacian as $a\to+\infty$ together with the Faber-Krahn inequality \cite{Faber1923,Krahn1925}, if $\Omega$ is different form a disk $D$ of the same area, then there exists a large enough $a_\Omega>0$ ---depending on $\Omega$--- such that $\mu_{1,\Omega}(a) > \Lambda_{1,D} > \mu_{1,D}(a)$ for all $a>a_\Omega$, where $\Lambda_{1,D}$ is the first eigenvalue of the Dirichlet Laplacian in $D$. In the paper \cite{DuranMasSanzPerela2025} we show an analogous shape optimization result in the regime of small enough $a>0$, assuming in addition that the underlying domain $\Omega$ is simply connected.

\subsection{Organization of the paper} \label{sec:Organization}

In \Cref{sec:PreliminariesComplex} we present some preliminary results in complex analysis that will be used throughout the work. 

With these tools at hand, in \Cref{sec:TheRodzinLaplacian} we study the operator associated to the problem \eqref{eq:BvPRodzinLaplacian}, for every fixed $a>0$. First, in \Cref{sec:HSandSFsetting} we study the Hilbert space $E(\Omega)$ defined in \eqref{eq:AmbientHilbertSpace} and a suitable sesquilinear form, which provide the adequate weak formulation of \eqref{eq:BvPRodzinLaplacian}. Next, in \Cref{sec:AssociatedOperator} we establish the first part of \Cref{thm:IntroRodzinLaplacian}, proving that the $\overline\partial$-Robin Laplacian $\Rodzin_a$ is the unique self-adjoint operator associated to the problem \eqref{eq:BvPRodzinLaplacian}. Next, in \Cref{sec:LocalSpectralProperties} we establish the second part of \Cref{thm:IntroRodzinLaplacian}, regarding the discreteness of the spectrum of $\Rodzin_a$, and \Cref{thm:PropertiesMukOmega}, characterizing the eigenvalues of $\Rodzin_a$ as min-max levels of a functional linear in $a>0$. 

After this study for fixed $a>0$, in \Cref{sec:Family} we study the operators $\Rodzin_a$ as the boundary parameter $a$ moves in $(0,+\infty)$. First, in \Cref{sec:EigenvalueCurves} we verify that the family of operators $\{\Rodzin_a\}_{a>0}$ fits in Kato's perturbation theory developed in \cite{Kato1995}, in order to deduce analytic properties of the eigenvalues of $\Rodzin_a$. Next, in \Cref{sec:GlobalSpectralProperties} we establish \Cref{thm:PropertiesMukOmegaGlobal}, exhibiting the properties that such eigenvalues satisfy as functions of the boundary parameter $a>0$. \Cref{sec:Family} concludes with the study of the convergence of the operators $\Rodzin_a$ in a resolvent sense, giving in \Cref{sec:NRCDirichlet} the proofs of \Cref{thm:NRCDirichletLaplacian,thm:SRCRodzin0}, and of \Cref{prop:SpectrumRodzin0}.

This article contains one appendix, namely \Cref{sec:EigenBall}, devoted to the spectral analysis of $\Rodzin_a$ when the underlying domain is a disk.

We conclude this introductory section justifying the assumption that $\Omega$ has a $\mathcal C^2$ boundary. In order to relax this regularity to Lipschitz, the following obstructions ---which are beyond the scope of this work--- should be addressed.

\begin{remark}
    The $\mathcal C^2$ regularity of $\Omega$ is used in this paper in \Cref{lemma:AntunesRegularity,lemma:SingleLayer,lemma:GainRegularityOrthogonal,lemma:RegularityEstimatesRodzin}, and \Cref{rmk:CauchyPompeiu,rmk:enoughForGainRegularity}. However, it is only crucially needed in \Cref{lemma:AntunesRegularity,lemma:RegularityEstimatesRodzin} (which yield~\Cref{thm:IntroRodzinLaplacian}). On the one hand, the proof of \Cref{lemma:RegularityEstimatesRodzin} needs the boundary of $\Omega$ to have a well-defined and uniformly bounded mean curvature. On the other hand, the proof of \Cref{lemma:AntunesRegularity} is based on the results obtained throughout~\cite{Benguria2017Self}, where the~$\mathcal C^2$ regularity of~$\partial\Omega$ plays a key role. Indeed, the paper~\cite{Benguria2017Self} concerns the self-adjointness of quantum dot Dirac operators ---recall \Cref{sec:RelationRodzinDirac}---, which is known to be sensitive to the regularity of the boundary; see~\cite{BenhellalPank2025,Pizzichillo2021} and the references therein.
\end{remark}

\section{Preliminary results from complex analysis} \label{sec:PreliminariesComplex}

Since the problem \eqref{eq:BvPRodzinLaplacian} is related to the decomposition of the Laplacian as $\Delta= 4\partial_z\partial_{\bar z}$, it is convenient to study some of the properties that the ``Cauchy-Riemann operators" $\partial_z$ and $\partial_{\bar z}$ satisfy.

\subsection{Maximal Wirtinger operators} \label{sec:Wirtinger}

In this work, the maximal Wirtinger operators 
\begin{equation}
    \begin{split}
        \Dom(\partial_\mathrm{h}) := \left\{ u\in L^2(\Omega): \, \partial_{\bar z} u \in L^2(\Omega) \right\}, & \quad \partial_\mathrm{h} u := \partial_{\bar z} u \quad \text{for every } u\in \Dom(\partial_\mathrm{h}), \\
        \Dom(\partial_{\mathrm{ah}}) := \left\{ u\in L^2(\Omega): \, \partial_z u \in L^2(\Omega) \right\}, & \quad \partial_{\mathrm{ah}} u := \partial_z u \quad \text{for every } u\in \Dom(\partial_{\mathrm{ah}}),
    \end{split}
\end{equation}
will play a crucial role. The following properties are proven in \cite[Section 3]{Antunes2021}.

\begin{lemma} \label{lemma:propertiesWirtinger}
    The following properties hold:
    \begin{enumerate}[label=$(\roman*)$]
        \item \cite[Section 3]{Antunes2021} Let $\# \in \{\mathrm{h, ah}\}$. $\Dom(\partial_\#)$ is a Hilbert space when endowed with the scalar product defined for $u,v \in \Dom(\partial_\#)$ by
        \begin{equation}
            \langle u, v \rangle_\# = \langle u, v \rangle_{L^2(\Omega)} + \langle \partial_\# u, \partial_\# v \rangle_{L^2(\Omega)}.
        \end{equation}
        \item \cite[Lemma 15]{Antunes2021} Let $\# \in \{\mathrm{h, ah}\}$. The trace operator $t_{\partial\Omega} : H^1(\Omega) \rightarrow H^{1/2}(\partial\Omega)$ extends into a linear bounded operator from $\Dom(\partial_\#)$ to $H^{-1/2}(\partial\Omega)$, which we call again $t_{\partial\Omega}$. In particular, for every $u\in \Dom(\partial_\#)$
        \begin{equation}
            \| t_{\partial\Omega} u \|_{H^{-1/2}(\partial\Omega)} \leq C_\Omega \left( \|u\|_{L^2(\Omega)} + \|\partial_\# u\|_{L^2(\Omega)} \right),
        \end{equation}
        for some constant $C_\Omega>0$ depending only on $\Omega$.
        \item \cite[Remark 17]{Antunes2021} For every $u\in \Dom(\partial_{\mathrm{ah}})$, $v\in \Dom(\partial_{\mathrm{h}})$, and $w\in H^1(\Omega)$, the following Green formulas hold:
        \begin{equation} \label{eq:GreenFormulas}
            \begin{split}
                \langle \partial_z u, w \rangle_{L^2(\Omega)} & = - \langle u, \partial_{\bar z}w \rangle_{L^2(\Omega)} + \dfrac{1}{2} \overline{ \langle t_{\partial\Omega} u, \nu t_{\partial\Omega} w \rangle}_{H^{-1/2}(\partial\Omega), H^{1/2}(\partial\Omega)}, \\
                \langle \partial_{\bar z} v, w \rangle_{L^2(\Omega)} & = - \langle v, \partial_z w \rangle_{L^2(\Omega)} + \dfrac{1}{2} \overline{ \langle t_{\partial\Omega} v, \bar \nu t_{\partial\Omega} w \rangle}_{H^{-1/2}(\partial\Omega), H^{1/2}(\partial\Omega)}.
            \end{split}
        \end{equation}
        \item \cite[Lemma 18]{Antunes2021} Let $\# \in \{\mathrm{h, ah}\}$ and $u\in \Dom(\partial_\#)$. If $t_{\partial\Omega} u \in H^{1/2}(\partial\Omega)$, then $u\in H^1(\Omega)$.
    \end{enumerate}
\end{lemma}

For the sake of simplifying the notation, from now on we will denote $u = t_{\partial\Omega} u$ when no confusion arises.

\begin{remark}
    On the one hand, it is remarkable to say that throughout \cite{Antunes2021} the domain $\Omega$ is assumed to be bounded and $\mathcal C^\infty$; however, the proofs in \cite{Antunes2021} of the properties of \Cref{lemma:propertiesWirtinger} do not use the $\mathcal C^\infty$ regularity of $\partial\Omega$; they are valid for $\mathcal C^2$ domains, the setting of this work. On the other hand, the Green formulas appear in \Cref{lemma:propertiesWirtinger} with complex conjugate on the boundary terms, because we define $\langle \cdot, \cdot \rangle_{L^2(\Omega)}$ to be linear with respect to the first entry ---recall \Cref{sec:Notation}.
\end{remark}

Notice from \Cref{lemma:propertiesWirtinger} $(ii)$ that a function in $\Dom(\partial_\mathrm{h})$ has a trace a priori only in $H^{-1/2}(\partial\Omega)$. The following lemma gives a criterion to ensure that the trace is actually in $H^{1/2}(\partial\Omega)$.

\begin{lemma} \label{lemma:AntunesRegularity}
    Let $u\in L^2(\Omega)$ be such that $\partial_{\bar z} u \in L^2(\Omega)$ and $\Delta u\in L^2(\Omega)$. If there exists a $\mathcal C^1(\partial \Omega, \R)$ function $\vartheta$ vanishing nowhere in $\partial \Omega$ such that $\partial_{\bar z} u = \vartheta \nu u$ in $H^{-1/2}(\partial\Omega)$, then $u \in H^1(\Omega)$ and $\partial_{\bar z}u \in H^1(\Omega)$. In particular, the equality $\partial_{\bar z} u = \vartheta \nu u$ holds in $H^{1/2}(\partial\Omega)$.
\end{lemma}

A simplified version of this lemma is shown in the proof of \cite[Proposition 32]{Antunes2021}, assuming in addition that $\Omega$ is $\mathcal C^\infty$. The need for this regularity assumption arises from the strategy of showing regularity through pseudo-differential operator techniques. An alternative (but long) proof could be done, replacing the use of pseudo-differential operators by \cite[Theorem~4.3.2]{Nedelec2001} and explicit regularity estimates of the kernels of the singular integral operators appearing throughout \cite{Antunes2021}, for which $\Omega$ may simply be $\mathcal C^2$. Instead, we next give a simple proof of \Cref{lemma:AntunesRegularity} for $\mathcal C^2$ domains, which already unveils the relation of $\Rodzin_a$ with quantum dot Dirac operators ---recall \Cref{sec:RelationRodzinDirac}. Our proof of \Cref{lemma:AntunesRegularity} will be a consequence of the following regularity result, which is a restatement of \cite[Theorem 1.1, Remark 1, Lemma 2.2, and Lemma 3.3]{Benguria2017Self}.

\begin{lemma}[\cite{Benguria2017Self}] \label{lemma:BenguriaRegularity}
    Let $\Omega$ be a bounded $\mathcal C^2$ domain in $\R^2$. If $\varphi = (u, v)^\intercal \in L^2(\Omega,\C^2)$ is such that
    \begin{equation}
        \begin{split}
            & \begin{pmatrix}
            0 & \partial_z \\
            \partial_{\bar z} & 0
        \end{pmatrix} \varphi \in L^2(\Omega, \C^2), \text{ and } \\
        &   v = i\nu \dfrac{1-\sin \theta}{\cos \theta} u \ \text{ in } H^{-1/2}(\partial\Omega) \text{ for some } \theta \text{ such that } \cos\theta \neq 0,
        \end{split}
    \end{equation}
    then $\varphi\in H^1(\Omega,\C^2)$, and the boundary equation holds in $H^{1/2}(\partial\Omega)$.
\end{lemma}

Indeed, our proof of \Cref{lemma:AntunesRegularity} relies on building a function $\varphi$ from $u$ and $\partial_{\bar z} u$ which fits in the hypothesis of \Cref{lemma:BenguriaRegularity}.

\begin{proof}[Proof of \Cref{lemma:AntunesRegularity}]
    Since $\vartheta\in \mathcal C^1(\partial \Omega, \R)$ and does not vanish, it is either positive or negative on $\partial\Omega$. Without loss of generality, assume that $\vartheta$ is positive (the other case follows by analogous arguments). We extend $\vartheta$ in $\overline \Omega$ without vanishing.

    For $\epsilon>0$, set $\Omega_\epsilon := \{x\in \Omega: \mathrm{dist}(x,\partial\Omega) < \epsilon\}$. Since $\Omega$ is $\mathcal C^2$, for small enough $\epsilon>0$ the map
    \begin{equation}
        \begin{matrix}
            \partial\Omega \times (0,\epsilon) & \longrightarrow & \Omega_\epsilon \\
            (x,t) & \longmapsto & x-t\nu(x)
        \end{matrix}
    \end{equation}
    is a $\mathcal C^1$-diffeomorphism. Thus, setting
    \begin{equation}
        \begin{matrix}
            \gamma :& \Omega_\epsilon & \longrightarrow & \R \\
            & x-t\nu(x) & \longmapsto & \vartheta(x) 
        \end{matrix}
    \end{equation}
    for $x\in \partial\Omega$, $t\in(0,\epsilon)$, and taking $\eta := (1-\xi)\gamma +\xi$, where $\xi \in \mathcal C^\infty_c(\Omega\setminus\Omega_{\epsilon/2})$ is a cutoff function such that $\chi_{\Omega\setminus\Omega_\epsilon}\leq \xi\leq \chi_\Omega$, we have that $\eta \in \mathcal C^1(\overline \Omega)$ ---here, $\chi_A$ denotes the characteristic function of $A$. Moreover, $\eta$ vanishes nowhere in $\overline \Omega$, and $\eta|_{\partial\Omega} = \vartheta$. 

    Set $v:=\frac{i}{\eta} \partial_{\bar z} u$, which is well defined because $\eta$ vanishes nowhere in $\overline \Omega$. Then set $\varphi := (u, v)^\intercal$, and notice that $\varphi \in L^2(\Omega, \C^2)$ by the hypothesis in the statement and the fact that $\frac{1}{|\eta|}$ is bounded in $\overline \Omega$. Moreover,
    \begin{equation}
        \begin{pmatrix}
            0 & \partial_z \\
            \partial_{\bar z} & 0
        \end{pmatrix} \varphi = \begin{pmatrix}
            0 & \partial_z \\
            \partial_{\bar z} & 0
        \end{pmatrix} \begin{pmatrix}
            u \\
            \frac{i}{\eta }\partial_{\bar z} u
        \end{pmatrix} = \begin{pmatrix}
            \frac{i}{4\eta} \Delta u +i \partial_{\bar z} u \partial_z \frac{1}{\eta} \\
            \partial_{\bar z} u
        \end{pmatrix}
    \end{equation}
    is also in $L^2(\Omega,\C^2)$, by the hypothesis in the statement and the fact that $|\partial_z \frac{1}{\eta}|$ is bounded in $\overline \Omega$. In addition, since $\partial_{\bar z} u = \vartheta \nu u$ in $H^{-1/2}(\partial\Omega)$ by hypothesis, on the boundary there holds
    \begin{equation}
        v = \frac{i}{\eta} \partial_{\bar z} u = \frac{i}{\vartheta} \vartheta \nu u = i \nu u \quad \text{in } H^{-1/2}(\partial\Omega).
    \end{equation}
    Hence, by \Cref{lemma:BenguriaRegularity} with $\theta=0$, $\varphi\in H^1(\Omega, \C^2)$ and the boundary equation holds in $H^{1/2}(\partial\Omega, \C^2)$. That is, $u\in H^1(\Omega)$, $\partial_{\bar z} u = -i\eta v \in H^1(\Omega)$ ---since $v\in H^1(\Omega)$ and $\eta\in \mathcal C^1(\overline\Omega)$---, and $\partial_{\bar z} u = -i\eta v = -i \vartheta i \nu u = \vartheta\nu  u $ in $H^{1/2}(\partial\Omega)$.
\end{proof}

We conclude our review on the maximal Wirtinger operators with the Cauchy integrals defined by their fundamental solutions. Since the fundamental solution of $-\Delta$ is
\begin{equation} \label{eq:FundSolLaplacian}
    G(x) :=\dfrac{-1}{2\pi} \ln |x| \quad \text{for all } x=(x_1,x_2)\in\R^2\setminus\{0\},
\end{equation}
it is clear that the fundamental solution of $\partial_{\bar z}$ is
\begin{equation} \label{eq:FundSolDbar}
    K(x) := -4\partial_{z} G(x) = \dfrac{1}{\pi} \dfrac{1}{x_1+ix_2} \equiv \dfrac{1}{\pi} \dfrac{1}{z} \quad \text{for all } x=(x_1,x_2)\equiv x_1+ix_2=z,
\end{equation}
and that the fundamental solution of $\partial_z$ is $\overline{K(x)}$. Recall that the complex measure $d\xi$ is related with the surface measure $d\upsigma(\xi)$ by $d\xi = i\nu(\xi) d\upsigma(\xi)$, for every $\xi=\xi_1+i\xi_2 \equiv(\xi_1,\xi_2)\in \partial\Omega$.

\begin{lemma} \label{lemma:SingleLayer}
    Let $K$ be the fundamental solution defined in \eqref{eq:FundSolDbar}, that is, $K(z) = \frac{1}{\pi z}$ for $z\in \C\setminus\{0\}$. If $\Omega$ is $\mathcal C^2$ and $s\in(-1,1/2]$, then the Cauchy integral operator
    \begin{equation}
        \begin{array}{cccl}
            \Phi_\mathrm{h} : & H^s(\partial\Omega) & \rightarrow & H^{s+1/2}(\Omega) \\
            & f & \mapsto & \Phi_\mathrm{h}f(z)= \displaystyle{-\frac{1}{2} \int_{\partial\Omega} K(z-\xi)f(\xi) \nu(\xi) \, d\upsigma(\xi) = \frac{1}{2\pi i} \int_{\partial\Omega} \dfrac{f(\xi)}{\xi-z} \, d\xi}
        \end{array}
    \end{equation}
    is well defined and bounded. The same is true for $\Phi_\mathrm{ah}f:= \overline{\Phi_\mathrm{h}\overline f}$.
\end{lemma}

This result is proven for $s\in \{-1/2, 0, 1/2\}$ in \cite[Proposition 22]{Antunes2021}, assuming in addition that $\Omega$ is $\mathcal C^\infty$. Again, the need for this regularity assumption is due to the use of pseudo-differential opearators. We now provide an alternative proof, deriving the boundedness of the Cauchy integral $\Phi_{\mathrm{h}}$ from the boundedness of the single layer given by the fundamental solution of the Laplacian, for which $\Omega$ may simply be $\mathcal C^2$.

\begin{proof}[Proof of \Cref{lemma:SingleLayer}]
    Recall that the fundamental solution $K$ of $\partial_{\bar z}$ is given in terms of the fundamental solution $G$ of $-\Delta$ by $K = -4\partial_z G$. Since $\Omega$ is $\mathcal C^2$, by \cite[Corollary 6.14]{McLean2000}, the single layer operator
    \begin{equation}
        \begin{array}{cccl}
            \mathrm{SL} : & H^s(\partial\Omega) & \rightarrow & H^{s+3/2}(\Omega) \\
            & f & \mapsto & \int_{\partial\Omega} G(\cdot-y)f(y) d\upsigma(y)
        \end{array}
    \end{equation}
    is bounded for every $s\in(-1,1/2]$. Hence, since $\nu\in\mathcal C^1(\partial\Omega)$, composing $\mathrm{SL}(-\frac{1}{2}f\nu)$ with $-4\partial_z: H^{s+3/2}(\Omega) \rightarrow H^{s+1/2}(\Omega)$, we readily obtain the result for $\Phi_{\mathrm{h}}$.
\end{proof}

\begin{remark} \label{rmk:CauchyPompeiu}
    In \cite[Theorem 21]{Antunes2021} it is proven that for every holomorphic function $u\in L^2(\Omega)$ whose trace ---a priori in $H^{-1/2}(\partial\Omega)$, by \Cref{lemma:propertiesWirtinger}--- is actually in $L^2(\partial\Omega)$, it holds $u=\Phi_{\mathrm{h}}u$ and $t_{\partial\Omega} \Phi_{\mathrm{h}}u = t_{\partial\Omega} u$. Namely, for such functions 
    \begin{equation} \label{eq:CauchyPompeiu}
        u(z) = \Phi_{\mathrm{h}}u(z) = \dfrac{1}{2\pi i} \int_{\partial\Omega} \dfrac{u(\xi)}{\xi-z} \, d\xi \quad \text{for every } z\in \Omega, \quad \text{and} \quad \Phi_{\mathrm{h}}u = u \, \text{ in } L^2(\partial\Omega).
    \end{equation}
    The proof of \cite[Theorem 21]{Antunes2021} assumes in addition that $\Omega$ is $\mathcal C^\infty$, but it only relies on the boundedness of $\Phi_{\mathrm{h}}$ as an operator from $H^s(\partial\Omega)$ to $H^{s+1/2}(\Omega)$, for $s\in \{-1/2, 0, 1/2\}$. Hence, by \Cref{lemma:SingleLayer}, \eqref{eq:CauchyPompeiu} also holds for $\mathcal C^2$ domains.
\end{remark}

\subsection{The Bergman space} \label{sec:BergmanSpace}

The kernel of the maximal Wirtinger operator $\partial_{\mathrm{h}}$ turns out to be the so-called Bergman space, which we already introduced in \Cref{prop:SpectrumRodzin0},
\begin{equation} \label{eq:BergmanSpace}
    A^2(\Omega) = \left\{u\in L^2(\Omega): \, \partial_{\bar z} u = 0 \text{ in } \Omega \right\}.
\end{equation}
It is well known that $A^2(\Omega)$ is a closed, proper, subspace of $L^2(\Omega)$ ---actually, it is a Hilbert space \cite[Proposition 1.13]{Conway2007}. We recall that we denote the orthogonal of the Bergman space as
\begin{equation}
    A^2(\Omega)^\perp = \left \{ u\in L^2(\Omega): \, \langle u, v \rangle_{L^2(\Omega)} = 0 \text{ for all } v \in A^2(\Omega) \right\},
\end{equation}
and that since $A^2(\Omega)$ is a closed, proper, subspace of $L^2(\Omega)$, the orthogonal projections
\begin{equation}
    P^\perp: L^2(\Omega) \rightarrow A^2(\Omega)^\perp \subset L^2(\Omega) \quad \text{and} \quad P = Id-P^\perp : L^2(\Omega) \rightarrow A^2(\Omega) \subset L^2(\Omega)
\end{equation}
are well defined, bounded, self-adjoint operators in $L^2(\Omega)$, with $\|P^\perp\|_{L^2(\Omega)\to L^2(\Omega)} = 1$ and $\|P\|_{L^2(\Omega)\to L^2(\Omega)} = 1$.

In this work, many proofs involving $L^2(\Omega)$ functions $u$ with some extra regularity assumptions will be based on decomposing $u=Pu+P^\perp u$ and studying separately the projections $Pu$ and $P^\perp u$. The following lemmas, which account for regularity estimates, will therefore be useful.

\begin{lemma} \label{lemma:BergmanRegularity}
    If the trace of $u_\mathrm{h}\in A^2(\Omega)$ ---a priori in $H^{-1/2}(\partial\Omega)$, by \Cref{lemma:propertiesWirtinger}--- is actually in $L^2(\partial\Omega)$, then $u_\mathrm{h}\in H^{1/2}(\Omega)$ and
    \begin{equation}
        \|u_\mathrm{h}\|_{H^{1/2}(\Omega)} \leq C_\Omega\|u_\mathrm{h}\|_{L^2(\partial\Omega)},
    \end{equation}
    for some constant $C_\Omega>0$ depending only on $\Omega$.
\end{lemma}

\begin{proof}
    For the sake of notation, write $v=u_\mathrm{h}$. On the one hand, by \Cref{rmk:CauchyPompeiu} it holds $v= \Phi_{\mathrm{h}} v$. On the other hand, by \Cref{lemma:SingleLayer} with $s=0$, there exists a constant $C_\Omega$ depending only on $\Omega$ such that $\|\Phi_{\mathrm{h}} v\|_{H^{1/2}(\Omega)} \leq C_\Omega \|v\|_{L^2(\partial \Omega)}$.
\end{proof}

\begin{lemma} \label{lemma:GainRegularityOrthogonal}
    If $u_\perp\in A^2(\Omega)^\perp \cap \Dom(\partial_{\mathrm{h}})$, then $u_\perp\in H^1(\Omega)$ and 
    \begin{equation}
        \|u_\perp\|_{H^1(\Omega)} \leq C  \|\partial_{\bar z} u_\perp\|_{L^2(\Omega)},
    \end{equation}
    for some constant $C>0$ depending only on $\Omega$.
\end{lemma}

\begin{proof}
    For the sake of notation, write $v=u_\perp$. We will follow the ideas from the proof of \cite[Lemma 3.2]{DuranMas2024}. More specifically, the idea is to consider
    \begin{equation} \label{eq:vk}
        v_K(x) := \int_\Omega \overline{K(x-y)} v(y) \, dy \quad \text{for a.e. } x\in\Omega,
    \end{equation}
    where $K$ is the fundamental solution of $\partial_{\bar z}$ as in \eqref{eq:FundSolDbar}, show that it solves 
    \begin{equation}
        \begin{cases}
            -\Delta v_K = -4\partial_{\bar z} v & \text{in } \Omega, \\
            v_K = 0 & \text{on } \partial\Omega,
        \end{cases}
    \end{equation}
    and then use elliptic estimates. The orthogonality assumption $v\in A^2(\Omega)^\perp$ will be used crucially to show that $v_K\in H^1_0(\Omega)$. The details go as follows.

    First, we show that $v_K\in H^1(\Omega)$ with
    \begin{equation}\label{eq:vk_H1_v_L2}
        \|v_K\|_{H^1(\Omega)} \leq C \|v\|_{L^2(\Omega)},
    \end{equation}
    for some $C>0$ depending only on $\Omega$. For this, let $G$ be the fundamental solution of $-\Delta$ as in \eqref{eq:FundSolLaplacian}, and define
    \begin{equation} \label{def:Newt_pot}
        v_G(x):=\int_\Omega G(x-y)v(y)\,dy \quad \text{for a.e. } x\in\Omega.
    \end{equation}
    On the one hand, since $G \in L^1(\Omega)$ ---because $\Omega$ is bounded---, and $v_G$ is the convolution of $G$ and $v$, by Young's convolution inequality \cite[Section 4.2]{Lieb2001} we get that $v_G\in L^2(\Omega)$ with $\|v_G\|_{L^2(\Omega)} \leq \|G\|_{L^1(\Omega)} \|v\|_{L^2(\Omega)}$. On the other hand, since $v_G$ is the Newtonian potential of $v \in L^2(\Omega)$, by \cite[Theorem 9.9]{Gilbarg-Trudinger} we get that $v_G\in H^2(\Omega)$ with $\|D^2 v_G\|_{L^2(\Omega)} \leq C \|v\|_{L^2(\Omega)}$ for some $C>0$ depending only on $\Omega$. Using now the Gagliardo-Nirenberg interpolation inequality in bounded domains \cite[Theorem 1]{Nirenberg1966}, we have that $\|v_G\|_{H^2(\Omega)} \leq C \|v\|_{L^2(\Omega)}$ for some $C>0$ depending only on $\Omega$. Since $\overline K=-4 \partial_{\bar z} G$, it holds that $v_K=-4 \partial_{\bar z} v_G$ and, thus, $v_K\in H^1(\Omega)$ and \eqref{eq:vk_H1_v_L2} holds true.

    To show that $v_K\in H^1_0(\Omega)$, we will use, given $g\in \mathcal C^1(\partial\Omega)$, the function $w_g$ defined by
    \begin{equation}
        w_g(x):= \Phi_{\mathrm{h}}(-2g\overline\nu) (x) = \int_{\partial\Omega} K(x-y) g(y)\,d\upsigma(y) \quad \text{for all } x\in\Omega.
    \end{equation}
    That $w_g\in L^2(\Omega)$ follows from \Cref{lemma:SingleLayer} with $s=0$. However, let us give some details of the proof of this fact, which will be useful afterward in \eqref{eq:dom_conv_estimate}. By the Cauchy-Schwarz inequality,
    \begin{equation}
        |w_g(x)|^2\leq\int_{\partial\Omega} |K(x-y)|^{1/2}\,d\upsigma(y)\int_{\partial\Omega} |K(x-t)|^{3/2}|g(t)|^{2}\,d\upsigma(t) \quad\text{for all } x\in\Omega.
    \end{equation}
    Since $|K(x)|^{1/2}=(\pi|x|)^{-1/2}$ for all $x\in\R^2\setminus\{0\}$, using that $-1/2>-1$, and that $\partial\Omega$ is bounded and $\mathcal C^2$, it easily follows that the first integral is finite. Therefore, since $|K(x)|^{3/2}=(\pi|x|)^{-3/2}$ for all $x\in\R^2\setminus\{0\}$, using Fubini's theorem, that $-3/2>-2$, and that $\Omega$ is bounded, we obtain
    \begin{equation}
        \|w_g\|_{L^2(\Omega)}^2\leq C\int_{\partial\Omega}|g(t)|^2\int_\Omega |K(x-t)|^{3/2}\,dx\,d\upsigma(t) \leq C\|g\|_{L^2(\partial\Omega)}^2
    \end{equation}
    for some $C>0$ depending only on $\Omega$. This proves that $w_g\in L^2(\Omega)$. Moreover, $w_g$ is continuously differentiable in $\Omega$. In particular, using that $\partial_{\bar z} K(x)=0$ for all $x\in\R^2\setminus\{0\}$, we get that $\partial_{\bar z} w_g(x)=0$ for all $x\in\Omega$, that is, $w_g\in A^2(\Omega)$.

    Having $w_g$ defined, we claim that 
    \begin{equation} \label{eq:auxPairingsvK}
        \langle v, w_g\rangle_{L^2(\Omega)} = - \langle v_K, g\rangle_{L^2(\partial\Omega)},
    \end{equation}
    and thus since $v\in A^2(\Omega)^\perp$ and $w_g\in A^2(\Omega)$ we have $\langle v_K, g\rangle_{L^2(\partial\Omega)}=0$. Since this holds for all $g\in \mathcal C^1(\partial\Omega)$, we deduce that the trace of $v_K$ on $\partial\Omega$ vanishes and, therefore, $v_K\in H^1_0(\Omega)$. Let us show \eqref{eq:auxPairingsvK}, which will follow from
    \begin{equation}\label{eq:norm_conv_zero_trace}
        \begin{split}
            \langle v, w_g\rangle_{L^2(\Omega)}
            &=\int_\Omega \int_{\partial\Omega} v(x)\,\overline{K(x-y)g(y)}\,d\upsigma(y)\,dx\\
            & = -\int_{\partial\Omega} \int_\Omega \overline{K(y-x)}v(x)\,\overline{g(y)}\,dx\,d\upsigma(y) = - \langle v_K, g\rangle_{L^2(\partial\Omega)}
        \end{split}
    \end{equation}
    once we justify the equalities. For the second equality we used Fubini's theorem and the estimate
    \begin{equation}\label{eq:dom_conv_estimate}
        \begin{split}
            & \Big(\int_\Omega \int_{\partial\Omega} |v(x)||K(x-y)||g(y)|\,d\upsigma(y)\,dx\Big)^{2}\\
            &\quad \leq\Big(\int_\Omega |v(x)|^2 \int_{\partial\Omega} |K(x-y)|^{1/2}\,d\upsigma(y)\,dx\Big) \Big(\int_{\partial\Omega}|g(\tilde y)|^2\int_\Omega  |K(\tilde x-\tilde y)|^{3/2}\,d\tilde x\,d\upsigma(\tilde y)\Big)\\
            &\quad <+\infty,
        \end{split}
    \end{equation}
    which follows similarly to what we argued to prove that $w_g\in L^2(\Omega)$. Next, we justify the third equality in \eqref{eq:norm_conv_zero_trace}, in which the term $v_K$ on the right hand side denotes the trace on $\partial\Omega$ of $v_K\in H^1(\Omega)$. Given $\epsilon>0$ set $\Omega_\epsilon := \{x\in \Omega: \operatorname{dist}(x, \partial \Omega)>\epsilon \}$ and $v_\epsilon:= \chi_{\Omega_\epsilon} v$, where $\chi_{\Omega_\epsilon}$ is the characteristic function of $\Omega_\epsilon$. Let $(v_\epsilon)_K$ be defined as in $\eqref{def:Newt_pot}$ replacing $v$ by $v_\epsilon$. On the one hand, note that $(v_\epsilon)_K$ is continuous in a neighborhood of $\partial\Omega$ since $v_\epsilon$ vanishes outside $\Omega_\epsilon$. Hence, the trace of $(v_\epsilon)_K$ on $\partial\Omega$ is given by the formula
    \begin{equation}\label{eq:trace_approx_vK_vKeps2}
        (v_\epsilon)_K(y) = \int_{\Omega} \overline{K(y-x)}v_\epsilon(x) \,dx = \int_{\Omega_\epsilon} \overline{K(y-x)}v(x) \,dx \quad \text{for all } y\in \partial \Omega.
    \end{equation}
    On the other hand, by the trace theorem and \eqref{eq:vk_H1_v_L2} applied to $v_\epsilon-v$, we have
    \begin{equation}\label{eq:trace_approx_vK_vKeps}
        \|(v_\epsilon)_K - v_K\|_{L^2(\partial\Omega)} \leq C \|(v_\epsilon)_K - v_K\|_{H^1(\Omega)} \leq C \|v_\epsilon - v\|_{L^2(\Omega)},
    \end{equation}
    for some $C>0$ depending only on $\Omega$. Now, applying \eqref{eq:trace_approx_vK_vKeps} and \eqref{eq:trace_approx_vK_vKeps2}, and then using dominated convergence thanks to \eqref{eq:dom_conv_estimate}, we conclude that
    \begin{equation}
        \begin{split}
            \langle v_K, g\rangle_{L^2(\partial\Omega)} & =         \lim_{\epsilon \downarrow 0}  \,\langle (v_\epsilon)_K, g\rangle_{L^2(\partial\Omega)} = \lim_{\epsilon \downarrow 0}  \int_{\partial \Omega} \int_{\Omega_\epsilon} \overline{K(y-x)}v(x) \, \overline{g(y)} \,dx\,d\upsigma(y) \\
            & = \int_{\partial \Omega} \int_{\Omega} \overline{K(y-x)}v(x) \, \overline{g(y)} \,dx\,d\upsigma(y).
        \end{split}
    \end{equation}

    Having now $v_k\in H^1_0(\Omega)$, recall that $\partial_{\bar z} K=\delta_0$, which leads to $\partial_z v_K=v$ by the definition of $v_K$ in terms of $v$. Therefore, in the sense of distributions in $\Omega$, it holds that
    \begin{equation}
    -\Delta v_K=-4\partial_{\bar z} (\partial_z v_K) =-4\partial_{\bar z} v \in L^2(\Omega).
    \end{equation}
    Using now that $\Omega$ is bounded, that $\partial\Omega$ is of class $\mathcal C^2$, and the global $H^2$-regularity theorem from \cite[Theorem 4 in Section 6.3.2 and Remark in pg. 335]{Evans2010}, we deduce that $v_K\in H^2(\Omega)$ and that there exists $C>0$ depending only on $\Omega$ such that
    \begin{equation}\label{eq:norm_conv_reg_est_H2}
    \|v_K\|_{H^2(\Omega)} \leq C\|\Delta v_K\|_{L^2(\Omega)}.
    \end{equation}

    The proof of the lemma finishes as follows. On the one hand, since $v_K\in H^2(\Omega)$ and $\partial_z v_K=v$, we get $v\in H^1(\Omega)$ and $\|v\|_{H^1(\Omega)}\leq C\|v_K\|_{H^2(\Omega)}$. On the other hand, \eqref{eq:vk_H1_v_L2} leads to $\|v_K\|_{L^2(\Omega)}\leq C\|v\|_{L^2(\Omega)}$. With these estimates in hand, and recalling that $-\Delta v_K= -4\partial_{\bar z} v$, from \eqref{eq:norm_conv_reg_est_H2} we conclude that
    \begin{equation}
        \|v\|_{H^1(\Omega)}\leq C\|v_K\|_{H^2(\Omega)} \leq C\|\Delta v_K\|_{L^2(\Omega)} \leq C\|\partial_{\bar z} v\|_{L^2(\Omega)}
    \end{equation}
    for some $C>0$ depending only on $\Omega$, as desired.
\end{proof}

\begin{remark} \label{rmk:enoughForGainRegularity}
    In view of the proof of \Cref{lemma:BergmanOrthRegularity}, in order to conclude that a function $u_\perp \in \mathrm{Dom}(\partial_{\mathrm{h}})$ is actually in $H^1(\Omega)$ it is enough to show that $u_\perp$ is orthogonal in $L^2(\Omega)$ to the set $\{\Phi_{\mathrm{h}}(-2g\overline\nu): g\in \mathcal C^1(\partial\Omega) \}$, which, on the one hand, it is a subspace of $A^2(\Omega)\cap H^1(\Omega)$ ---by \Cref{lemma:SingleLayer} and the proof of \Cref{lemma:BergmanOrthRegularity}---, and, on the other hand, it coincides with $\{\Phi_{\mathrm{h}}(g): g\in \mathcal C^1(\partial\Omega) \}$ ---by the arbitrariness of $g\in \mathcal C^1(\partial\Omega)$.
\end{remark}

As explained at the beginning of this section, we will be interested in decomposing a function $u$ as $u=Pu+P^\perp u$. As the following lemma shows, when $u\in \mathrm{Dom}(\partial_{\mathrm{h}})$ we can explicitly describe who these projections are.

\begin{lemma} \label{lemma:BergmanCharacterization}
    Given $u\in \mathrm{Dom}(\partial_{\mathrm{h}})$, let $w_u \in H^2(\Omega)\cap H^1_0(\Omega)$ be the unique solution to the Dirichlet problem
    \begin{equation} \label{lemma:BergmanDirichlet}
        \begin{cases}
            \Delta w_u = 4\partial_{\bar z} u & \text{in } \Omega, \\
            w_u = 0 & \text{on } \partial\Omega.
        \end{cases}
    \end{equation}
    Then it holds that $u=u_\mathrm{h}+u_\perp$ with $u_\mathrm{h} := Pu = u - \partial_z w_u \in A^2(\Omega)$ and $u_\perp := P^\perp u = \partial_z w_u \in A^2(\Omega)^\perp \cap H^1(\Omega)$. If, in addition, $u\in L^2(\partial\Omega)$, then $u_\mathrm{h} \in L^2(\partial\Omega)$, hence $u_\mathrm{h} \in A^2(\Omega) \cap H^{1/2}(\Omega)$. 
\end{lemma}

\begin{proof}
    Write $u=Pu+P^\perp u$, where $P$ and $P^\perp$ are the orthogonal projections onto $A^2(\Omega)$ and $A^2(\Omega)^\perp$, respectively. Since $\partial_{\bar z} P^\perp u =\partial_{\bar z} u \in L^2(\Omega)$, we have that $P^\perp u \in A^2(\Omega)^\perp \cap \mathrm{Dom}(\partial_{\mathrm{h}})$. In particular, by \Cref{lemma:GainRegularityOrthogonal} we have that $P^\perp u \in H^1(\Omega)$. Actually, defining $(P^\perp u)_K$ as in \eqref{eq:vk}, namely, 
    \begin{equation}
        (P^\perp u)_K(x) := \int_\Omega \overline{K(x-y)} P^\perp u(y) \, dy \quad \text{for a.e. } x\in\Omega,
    \end{equation}
    where $K$ is the fundamental solution of $\partial_{\bar z}$ as in \eqref{eq:FundSolDbar}, in the proof of \Cref{lemma:GainRegularityOrthogonal} we have shown that $(P^\perp u)_K$ solves the Dirichlet problem
    \begin{equation}
        \begin{cases}
            \Delta (P^\perp u)_K = 4\partial_{\bar z} P^\perp u = 4\partial_{\bar z} u & \text{in } \Omega, \\
            (P^\perp u)_K = 0 & \text{on } \partial\Omega.
        \end{cases}
    \end{equation}
    By uniqueness ---see, for example, \cite[Section 6]{Evans2010}---, we conclude that $(P^\perp u)_K=w_u$. 

    We now characterize the projections. On the one hand, since $K$ is the fundamental solution of $\partial_{\bar z}$, it is clear that $\partial_z (P^\perp u)_K = P^\perp u$. On the other hand, we have already proved that $(P^\perp u)_K=w_u$. Hence, $P^\perp u = \partial_z w_u$, and $Pu = u - P^\perp u = u-\partial_z w_u$.

    To conclude, if $u$ is in addition in $L^2(\partial\Omega)$, then $Pu = u-\partial_z w_u \in L^2(\partial\Omega)$, because $w_u \in H^2(\Omega)$ ---see, for example, \cite[Section 6.3.2]{Evans2010}. By \Cref{lemma:BergmanRegularity} we then conclude that $Pu \in H^{1/2}(\Omega)$.
\end{proof}

We recall from \Cref{lemma:GainRegularityOrthogonal} that for a function $u_\perp \in A^2(\Omega)^\perp \cap \Dom(\partial_{\mathrm{h}})$, the $L^2(\Omega)-$norm of $\partial_{\bar z} u_\perp$ controls the $H^1(\Omega)-$norm of $u_\perp$. As an application of \Cref{lemma:BergmanCharacterization}, we can give a sharp constant for the weaker control of the $L^2(\Omega)-$norm of $u_\perp$ by the $L^2(\Omega)-$norm of $\partial_z u_\perp$.

\begin{lemma} \label{lemma:BergmanOrthRegularity}
    Let $u_\perp\in A^2(\Omega)^\perp\cap \Dom(\partial_{\mathrm{h}})$. Then 
    \begin{equation}
        \|u_\perp\|_{L^2(\Omega)} \leq \frac{2}{\sqrt{\Lambda_1}} \| \partial_{\bar z} u_\perp \|_{L^2(\Omega)},
    \end{equation}
    where $\Lambda_{1}$ is the first eigenvalue of the Dirichlet Laplacian in $\Omega$. Moreover, there is equality if and only if $u_\perp = c \partial_z U_{1}$ for some $c\in \C$, where $U_{1}$ is the first eigenfunction of the Dirichlet Laplacian in $\Omega$, that is, the solution (up to a multiplying constant) to
    \begin{equation} \label{eq:DirichletLaplacian}
        \begin{cases}
            -\Delta U_{1} = \Lambda_{1} U_{1} & \text{in } \Omega, \\
            U_{1} = 0 & \text{on } \partial \Omega.
        \end{cases}
    \end{equation}
\end{lemma}

\begin{proof}
    For the sake of notation, write $u=u_\perp$. By \Cref{lemma:BergmanCharacterization}, $u=\partial_z w_u$, where $w_u\in H^2(\Omega)\cap H^1_0(\Omega)$ is the unique solution to \eqref{lemma:BergmanDirichlet}.
    Multiplying the equation $\Delta w_u = 4\partial_{\bar z}u$ by $\overline{w_u}$ and integrating in $\Omega$, by the divergence theorem we have
    \begin{equation}
        \int_\Omega |\nabla w_u|^2 = -4 \int_\Omega \partial_{\bar z} u \, \overline{w_u}.
    \end{equation} 
    Using the triangle inequality, Cauchy-Schwarz and the sharp Poincar\'e inequality for $w_u\in H^1_0(\Omega)$, we get
    \begin{equation} \label{eq:SharpInequalities}
        \|\nabla w_u \|_{L^2(\Omega)}^2 \leq 4\int_\Omega |\partial_{\bar z} u| \, |w_u| \leq 4 \|\partial_{\bar z} u \|_{L^2(\Omega)} \|w_u\|_{L^2(\Omega)} \leq \frac{4}{\sqrt{ \Lambda_{1}}} \|\partial_{\bar z} u \|_{L^2(\Omega)} \|\nabla w_u\|_{L^2(\Omega)}.
    \end{equation}
    Applying the divergence theorem twice to $w_u\in H^1_0(\Omega)$, and using \eqref{eq:SharpInequalities}, we get
    \begin{equation}
        \|u\|_{L^2(\Omega)}^2 = \int_\Omega |\partial_z w_u|^2 = -\frac{1}{4} \int_\Omega \Delta w_u \, \overline{w_u} = \frac{1}{4} \int_\Omega |\nabla w_u|^2 \leq \frac{4}{\Lambda_{1}}  \|\partial_{\bar z} u \|_{L^2(\Omega)}^2,
    \end{equation}
    as desired. Notice that there is equality if and only if all inequalities in \eqref{eq:SharpInequalities} are equalities, that is, if and only if
    \begin{enumerate}[label=$(\roman*)$]
        \item \label{SharpTI} Sharp triangle inequality: $-4 \partial_{\bar z} u \overline{w_u} \geq 0$;
        \item \label{SharpCS} Sharp Cauchy-Schwarz: either $w_u=0$ ---which yields $u=\partial_z w_u = 0$--- or $4|\partial_{\bar z} u| = k |w_u|$ for some constant $k\geq 0$;
        \item \label{SharpP} Sharp Poincar\'e inequality in $H^1_0(\Omega)$: if $w_u\neq 0$, $w_u$ is the first eigenfunction of the Dirichlet Laplacian in $\Omega$.
    \end{enumerate}
    Assuming that $u\neq 0$, we verify that these three conditions are satisfied by the functions of the statement. Let $U_{1}$ be the first eigenfunction of the Dirichlet Laplacian in $\Omega$, namely, the solution (up to a multiplying constant) of \eqref{eq:DirichletLaplacian}, and take $u=\partial_z U_{1}$. Notice that $\Delta U_{1} = 4 \partial_{\bar z} u$ in $\Omega$, and $U_{1} = 0$ on $\partial \Omega$. Therefore, by \Cref{lemma:BergmanCharacterization}, $w_u = U_{1}$. As a consequence, $u=P^\perp u \in A^2(\Omega)^\perp \cap \mathrm{Dom}(\partial_{\mathrm{h}})$, and properties \ref{SharpTI}, \ref{SharpCS}, and \ref{SharpP} are satisfied. 
    
    In order to conclude the proof, it only remains to show that there can not exist a function $\tilde u \in A^2(\Omega)^\perp\cap \Dom(\partial_{\mathrm{h}})$ different from $\partial_z U_{1}$ with $w_{\tilde u} = U_{1}$ ---this last equality holds in view of \ref{SharpP}. Indeed, by \Cref{lemma:BergmanCharacterization} we have that $\tilde u = P\tilde u + P^\perp \tilde u = 0+\partial_z w_{\tilde u} = \partial_z U_1$.
\end{proof}

As a consequence of the previous estimates, we have the following Poincar\'e type inequality.

\begin{lemma} \label{lemma:PoincareTypeIneq}
    For every $u\in \Dom(\partial_{\mathrm{h}})$ whose trace ---a priori in $H^{-1/2}(\partial\Omega)$, by \Cref{lemma:propertiesWirtinger}--- is actually in $L^2(\partial\Omega)$, there holds
    \begin{equation} \label{eq:PoincareTypeIneq}
        \int_\Omega |u|^2 \leq K_\Omega \left(\int_\Omega |\partial_{\bar z} u|^2 + \int_{\partial\Omega} |u|^2 \right),
    \end{equation}
    where $K_\Omega>0$ is a constant depending only on $\Omega$.
\end{lemma}

\begin{proof}
    Decompose $u$ as $u=u_\mathrm{h}+u_\perp$ with $u_\mathrm{h} \in A^2(\Omega)$ and $u_\perp \in A^2(\Omega)^\perp$, as in \Cref{lemma:BergmanCharacterization}. Due to orthogonality, it is enough to estimate the $L^2(\Omega)-$norms of $u_\mathrm{h}$ and $u_\perp$ separately. On the one hand, using \Cref{lemma:BergmanOrthRegularity} we have
    \begin{equation} \label{eq:auxiliarOrthBound}
        \|u_\perp\|_{L^2(\Omega)}^2 \leq \frac{4}{\Lambda_{1}} \| \partial_{\bar z} u_\perp \|_{L^2(\Omega)}^2 = \frac{4}{\Lambda_{1}} \| \partial_{\bar z} u \|_{L^2(\Omega)}^2,
    \end{equation}
    where we have used that $\partial_{\bar z} u = \partial_{\bar z} u_\perp$ because $u_\mathrm{h}\in A^2(\Omega)$. On the other hand, $u_\mathrm{h} = u - \partial_z w_u\in L^2(\partial\Omega)$  ---where $w_u$ is the unique solution to \eqref{lemma:BergmanDirichlet}---, and $\|u_\mathrm{h}\|_{L^2(\Omega)} \leq C \|u_\mathrm{h}\|_{L^2(\partial \Omega)}$, by \Cref{lemma:BergmanRegularity}. Using this, the triangle inequality, the trace theorem, and the bound $\|w_u\|_{H^2(\Omega)} \leq C' \|\partial_{\bar z} u\|_{L^2(\Omega)}$, for some constant $C'>0$ depending only on $\Omega$, which follows by standard regularity theory \cite[Theorem 4 in Section 6.3.2 and Remark in pg. 335]{Evans2010}, we get
    \begin{equation} \label{eq:BergmanAuxReg}
        \begin{split}
            \|u_\mathrm{h}\|_{L^2(\Omega)} & \leq C \|u-\partial_z w_u\|_{L^2(\partial\Omega)} \leq C \left( \|u\|_{L^2(\partial\Omega)} + \|\partial_z w_u \|_{L^2(\partial\Omega)} \right) \\
        & \leq C \left( \|u\|_{L^2(\partial\Omega)} + \|\nabla w_u \|_{L^2(\partial\Omega)} \right) \leq C \left( \|u\|_{L^2(\partial\Omega)} + C_\Omega\|\nabla w_u \|_{H^1(\Omega)} \right) \\
        & \leq C \left( \|u\|_{L^2(\partial\Omega)} + C_\Omega\|w_u \|_{H^2(\Omega)} \right) \leq C \left( \|u\|_{L^2(\partial\Omega)} + C_\Omega C'\|\partial_{\bar z} u\|_{L^2(\Omega)} \right).
        \end{split}
    \end{equation}
    This concludes the proof of the lemma.
\end{proof}

\subsection{A Hilbert subspace of the Bergman space} \label{sec:HilbertSubspace}

We recall that functions in the Bergman space $A^2(\Omega)$ have trace in $H^{-1/2}(\partial\Omega)$, by \Cref{lemma:propertiesWirtinger}. We abuse notation and denote by $A^2(\Omega)\cap L^2(\partial\Omega)$ the set of those functions in $A^2(\Omega)$ whose trace is actually in $L^2(\partial\Omega)$ ---in the literature, this space is sometimes called \emph{Hardy space}. As the following lemma shows, this is a Hilbert space when endowed with the natural scalar product.

\begin{lemma} \label{prop:HSBergmanL2Bdry}
    The space $A^2(\Omega)\cap L^2(\partial\Omega)$ is a Hilbert space when endowed with the scalar product $\langle \cdot, \cdot\rangle_{L^2(\partial\Omega)}$. 
\end{lemma}

\begin{proof}
    Let $\{u_j\}_j \subset A^2(\Omega)\cap L^2(\partial\Omega)$ be a Cauchy sequence with respect to $\|\cdot\|_{L^2(\partial\Omega)}$. By \Cref{lemma:BergmanRegularity}, $\{u_j\}_j$ is also a Cauchy sequence with respect to $\|\cdot\|_{L^2(\Omega)}$.
    \begin{enumerate}[label=$(\roman*)$]
        \item On the one hand, since $(L^2(\partial\Omega), \langle \cdot, \cdot\rangle_{L^2(\partial\Omega)})$ is a Hilbert space, there exists $u_{\partial\Omega}\in L^2(\partial\Omega)$ such that $\|u_{\partial\Omega} - u_j\|_{L^2(\partial\Omega)} \to 0$ as $j \to +\infty$.
        \item On the other hand, since $(A^2(\Omega), \langle \cdot, \cdot\rangle_{L^2(\Omega)})$ is a Hilbert space, there exists $u_{\Omega}\in A^2(\Omega)$ such that $\|u_{\Omega} - u_j\|_{L^2(\Omega)} \to 0$ as $j \to +\infty$.
    \end{enumerate}
    Therefore, for every $g\in H^{1/2}(\partial\Omega)$, denoting also by $g$ an extension in $\Omega$, by the Green's formulas \eqref{eq:GreenFormulas} we have
    \begin{equation}
        \begin{split}
            \dfrac{1}{2} \langle u_{\partial\Omega}, \overline \nu g \rangle_{L^2(\partial\Omega)} & = \underset{j\to +\infty}{\lim} \, \dfrac{1}{2} \langle u_j, \overline \nu g \rangle_{L^2(\partial\Omega)} = \underset{j\to +\infty}{\lim} \, \int_\Omega \partial_{\bar z} (u_j \, \overline g) = \underset{j\to +\infty}{\lim} \, \int_\Omega  u_j \, \partial_{\bar z} \overline g = \int_\Omega  u_\Omega \, \partial_{\bar z} \overline g \\
            & = \int_\Omega \partial_{\bar z} (u_\Omega \, \overline g) = \dfrac{1}{2} \overline{\langle t_{\partial\Omega} u_\Omega, \overline \nu g \rangle}_{H^{-1/2}(\partial\Omega), H^{1/2}(\partial\Omega)}.
        \end{split}
    \end{equation}
    In view of the property of the pairing \eqref{eq:Brezis}, this yields $t_{\partial\Omega} u_\Omega = u_{\partial\Omega}$ first in $H^{-1/2}(\partial\Omega)$, and then in $L^2(\partial\Omega)$. As a consequence, $u_j \to u_\Omega$ in $L^2(\partial\Omega)$ with $u_\Omega\in A^2(\Omega)\cap L^2(\partial\Omega)$, as desired.
\end{proof}

Our last preliminary result from complex analysis shall be needed in \Cref{sec:GlobalSpectralProperties}.

\begin{proposition} \label{prop:MinMaxSplope}
    For every $f\in A^2(\Omega)$, there exists a unique $u_f\in A^2(\Omega)\cap L^2(\partial\Omega)$ such that
    \begin{equation} \label{eq:RFSlope}
        \int_{\partial\Omega} u_f \, \overline v = \int_\Omega f \, \overline v \quad \text{for every } v\in A^2(\Omega)\cap L^2(\partial\Omega).
    \end{equation}
    Moreover, the operator 
    \begin{equation}
        \begin{array}{rcll}
            \mathcal S: & A^2(\Omega) & \longrightarrow & A^2(\Omega) \\
            & f & \longmapsto & \, \, \, u_f \quad \text{(the unique solution to \eqref{eq:RFSlope})}
        \end{array}
    \end{equation}
    is compact and self-adjoint in $A^2(\Omega)$. As a consequence, there exists a countably infinite orthonormal basis $\{w_k\}_k$ of $A^2(\Omega)$ consisting of eigenvectors of $\mathcal S$, with corresponding eigenvalues $(S_{k})^{-1}$, where
    \begin{equation}
        S_{k} = \underset{\substack{ F\subset A^2(\Omega)\cap L^2(\partial\Omega) \\ \mathrm{dim}(F)=k }}{\inf} \, \, \underset{u\in F\setminus\{0\}}{\sup} \, \dfrac{ \int_{\partial\Omega} |u|^2}{\int_\Omega |u|^2} \in (0,+\infty).
    \end{equation}
\end{proposition}

\begin{proof}
    Given $f\in L^2(\Omega)$, we define the linear functional 
    \begin{equation}
        \begin{matrix}
            T_f: & A^2(\Omega)\cap L^2(\partial\Omega) & \longrightarrow & \C \\
            & v & \longmapsto & \int_\Omega \overline f \, v
        \end{matrix}
    \end{equation}
    and claim that it is continuous with respect to the norm $\|\cdot\|_{L^2(\partial\Omega)}$. Indeed, by Cauchy-Schwarz and \Cref{lemma:BergmanRegularity}
    \begin{equation}
        |T_f(v)| \leq \int_\Omega |f| \, |v| \leq \|f \|_{L^2(\Omega)} \|v \|_{L^2(\Omega)} \leq C \|f \|_{L^2(\Omega)} \|v \|_{L^2(\partial\Omega)}.
    \end{equation}
    Hence, $T_f$ is in the dual space of $A^2(\Omega)\cap L^2(\partial\Omega)$. Then the first part of the statement follows by \Cref{prop:HSBergmanL2Bdry} and Riesz-Fr\'echet theorem \cite[Theorem 5.5 in Section 5.2]{Brezis2011}. As a consequence, the operator $\mathcal S$ is linear and continuous. Indeed, plugging $v=u_f$ in \eqref{eq:RFSlope} we get
    \begin{equation} \label{eq:auxiliarCompact}
        \|u_f\|_{L^2(\partial \Omega)}^2 =|T_f(u_f)| \leq  C \|f \|_{L^2(\Omega)} \|u_f \|_{L^2(\partial\Omega)},
    \end{equation}
    and combining this with \Cref{lemma:BergmanRegularity} we obtain
    \begin{equation}
        \|\mathcal S f\|_{L^2(\Omega)} = \|u_f\|_{L^2(\Omega)} \leq \|u_f\|_{H^{1/2}(\Omega)} \leq C \|u_f\|_{L^2(\partial \Omega)} \leq C \|f \|_{L^2(\Omega)}.
    \end{equation}
    
    Let us now prove that the operator $\mathcal S$ is self-adjoint. Since $\mathcal S$ is bounded, it is enough to show that it is symmetric. Given $f$ and $g$ in $L^2(\Omega)$, denote $\mathcal Sf = u_f$ and $\mathcal Sg = u_g$. Then
    \begin{equation}
        \langle \mathcal Sf, g \rangle_{L^2(\Omega)} = \overline{\int_\Omega  g \, \overline{u_f}} = \overline{\int_{\partial\Omega} u_g \, \overline{u_f} } = \int_{\partial\Omega} u_f \, \overline{u_g} = \int_\Omega f \, \overline{u_g} = \langle f, u_g \rangle_{L^2(\Omega)} = \langle f, \mathcal Sg \rangle_{L^2(\Omega)},
    \end{equation}
    where in the second equality we have used that $u_g$ is the solution to \eqref{eq:RFSlope} for $g$, and in the fourth equality we have used that $u_f$ is the solution to \eqref{eq:RFSlope} for $f$. This shows that $\mathcal S$ is self-adjoint. 
    
    Let us now show that $\mathcal S$ is compact. If $\{f_j\}_j \subset A^2(\Omega)$ is a bounded sequence, then by \eqref{eq:auxiliarCompact} $\{\mathcal Sf_j\}_j$ is a bounded sequence in $L^2(\partial\Omega)$. Moreover, since $\mathcal Sf_j \in A^2(\Omega)\cap L^2(\partial\Omega)$, by \Cref{rmk:CauchyPompeiu} it holds $\mathcal Sf_j = \Phi_{\mathrm{h}} \mathcal Sf_j$, and hence by \Cref{lemma:SingleLayer} $\{\mathcal Sf_j\}_j$ is a bounded sequence in $H^{1/2}(\Omega)$. By the compact embedding of $H^{1/2}(\Omega)$ in $L^2(\Omega)$ ---because $\Omega$ is bounded, see \cite[Lemma 6.10]{Palatucci2012}---, we conclude that there exists a subsequence, which we call again $\{\mathcal S f_j\}_j \subset A^2(\Omega)\cap L^2(\partial\Omega)$, converging to some $u\in L^2(\Omega)$ strongly in $L^2(\Omega)$. In particular, $\partial_{\bar z} u =0$ in the weak sense. Indeed, for every $v\in \mathcal C^\infty_c(\Omega)$ we have
    \begin{equation}
        \int_\Omega u \, \overline{\partial_z v} = \underset{j\to+\infty}{\lim} \, \int_\Omega \mathcal S f_j \, \overline{\partial_z v} = \underset{j\to+\infty}{\lim} \, - \int_\Omega \partial_{\bar z} \mathcal S f_j \, \overline v = 0.
    \end{equation}    
    We therefore conclude that the limit $u$ is in $A^2(\Omega)$, which shows that $\mathcal S$ is compact.

    Since $\mathcal S$ is a compact self-adjoint operator, by the spectral theorem \cite[Theorem 6.11 in Section 6.4]{Brezis2011} there exist $\{s_k\}_k \subset \R$ and $\{w_k\}_k$ orthonormal basis in $L^2(\Omega)$ such that $\mathcal S w_k = s_k w_k$, for every $k\in \N$. We now justify that zero is not an eigenvalue of $\mathcal S$. By way of contradiction, if $f\in A^2(\Omega)\setminus\{0\}$ was such that $\mathcal Sf =0$, then by \eqref{eq:RFSlope} we would have that
    \begin{equation} \label{eq:auxWoC0eig}
        \int_\Omega f \, \overline v =0 \quad \text{for every } v\in A^2(\Omega)\cap L^2(\partial\Omega).
    \end{equation}
    In particular, $f\in A^2(\Omega)\subset \mathrm{Dom}(\partial_{\mathrm{h}})$ would be orthogonal in $L^2(\Omega)$ to $\Phi_{\mathrm{h}}g \in A^2(\Omega)\cap L^2(\partial\Omega)$, for every $g\in \mathcal C^1(\partial\Omega)$. As a consequence, \Cref{rmk:enoughForGainRegularity} would yield that $f\in A^2(\Omega)\cap H^1(\Omega) \subset A^2(\Omega)\cap L^2(\partial\Omega)$, hence taking $v=f$ in \eqref{eq:auxWoC0eig} we would conclude that $f\equiv 0$, reaching a contradiction.

    Using that $\mathcal S w_k = s_k w_k$ and that $s_k\neq 0$, we now can show that $s_k>0$ for every $k\in \N$. Indeed, on the one hand, we have that $w_k = s_k^{-1} \mathcal S w_k \in A^2(\Omega)\cap L^2(\partial\Omega)$. On the other hand, by definition, for every $v\in A^2(\Omega)\cap L^2(\partial\Omega)$ there holds
    \begin{equation} \label{eq:auxiliarEL1}
        \int_\Omega w_k \, \overline v = \int_{\partial\Omega} \mathcal S w_k \, \overline v = \int_{\partial\Omega} s_k w_k \, \overline v = s_k \int_{\partial\Omega} w_k \, \overline v.
    \end{equation}
    In particular, taking $v=w_k$ we conclude that $s_k>0$. Notice that since $\{w_k\}_k$ is an orthonormal basis in $L^2(\Omega)$, \eqref{eq:auxiliarEL1} yields that $\{w_k\}_k$ is orthogonal in $L^2(\partial\Omega)$. Actually, $\{w_k\}_k$ is an orthogonal basis in $A^2(\Omega)\cap L^2(\partial\Omega)$. Indeed, if there existed $u\in A^2(\Omega)\cap L^2(\partial\Omega)$ orthogonal to every $w_k$ in $L^2(\partial\Omega)$, \eqref{eq:auxiliarEL1} would yield that $u$ is orthogonal to every $w_k$ in $L^2(\Omega)$, and since $\{w_k\}_k$ is an orthonormal basis in $L^2(\Omega)$ this would lead to $u=0$.

    As a consequence of the previous facts, every $u\in A^2(\Omega)\cap L^2(\partial\Omega)$ can be written in the form
    \begin{equation}
        u = \underset{k\in \N}{\sum} \, c_k w_k \quad \text{in } A^2(\Omega)\cap L^2(\partial\Omega), \text{ for some } c_k\in \C, \text{ with} \quad \dfrac{ \int_{\partial\Omega} |u|^2}{\int_\Omega |u|^2} = \dfrac{\underset{k\in \N}{\sum} \, |c_k|^2/s_k}{\underset{k\in \N}{\sum} \, |c_k|^2}.
    \end{equation}
    Therefore,
    \begin{equation} \label{eq:auxiliarEL2}
        S_k:= \dfrac{1}{s_k} = \underset{\substack{ F\subset A^2(\Omega)\cap L^2(\partial\Omega) \\ \mathrm{dim}(F)=k }}{\inf} \, \, \underset{u\in F\setminus\{0\}}{\sup} \, \dfrac{ \int_{\partial\Omega} |u|^2}{\int_\Omega |u|^2},
    \end{equation}
    which concludes the proof.
\end{proof} 

In this work, we do not investigate to which problem \eqref{eq:RFSlope} is its weak formulation. However, in the paper \cite{DuranMasSanzPerela2025} we show that the min-max level $S_{1}$ of \Cref{prop:MinMaxSplope} is related with the slope of the eigenvalue curve $\mu_{1}(a)$ of $\Rodzin_a$ ---see \Cref{thm:PropertiesMukOmega,thm:PropertiesMukOmegaGlobal}--- when departing from $a=0$.

We conclude this section justifying that the sequence of eigenvalues described in \Cref{thm:PropertiesMukOmega} is not bounded from below in the regime $a<0$.

\begin{remark} \label{rmk:Negativea}
    As a consequence of \Cref{prop:MinMaxSplope}, we have that
    \begin{equation}
        \inf_{u\in E(\Omega)\setminus\{0\}}\dfrac{4\int_\Omega |\partial_{\bar z} u|^2 + a\int_{\partial\Omega} |u|^2}{\int_\Omega |u|^2} = -\infty \quad \text{for every } a<0.
    \end{equation}
    Indeed, choosing the subset of competitors $\{w_k\}_{k\in\N} \subset E(\Omega)\setminus\{0\}$ ---which are eigenfunctions of a compact operator with eigenvalues $0<S_k^{-1} \to 0^+$ as $k\to+\infty$---, we have
    \begin{equation}
        \inf_{u\in E(\Omega)\setminus\{0\}}\dfrac{4\int_\Omega |\partial_{\bar z} u|^2 + a\int_{\partial\Omega} |u|^2}{\int_\Omega |u|^2} \leq \inf_{k\in \N}\dfrac{4\int_\Omega |\partial_{\bar z} w_k|^2 + a\int_{\partial\Omega} |w_k|^2}{\int_\Omega |w_k|^2} = \inf_{k\in \N}\ a S_k = -\infty,
    \end{equation}
    because $a<0$.
\end{remark}

\section{The $\overline\partial$-Robin Laplacian $\Rodzin_a$} \label{sec:TheRodzinLaplacian}

In this section, we characterize the unique self-adjoint operator associated to the problem~\eqref{eq:BvPRodzinLaplacian} and study some of its spectral properties, for each fixed $a>0$. To this end, we first explore the natural Hilbert space and sesquilinear form associated to the weak formulation of~\eqref{eq:BvPRodzinLaplacian}, which may be unveiled from \eqref{eq:DivergenceToRodzin}, namely, 
\begin{equation}
4\int_\Omega \partial_{\bar z}u \, \overline{\partial_{\bar z} v}
+a\int_{\partial\Omega}u \, \overline v = \int_\Omega f \, \overline v \quad \text{for all } v\in \mathcal C^\infty(\overline \Omega).
\end{equation}

\subsection{Hilbert space and sesquilinear form setting} \label{sec:HSandSFsetting}

Throughout \Cref{sec:PreliminariesComplex}, we have studied those functions $u$ in $L^2(\Omega)$ with weak derivative $\partial_{\bar z} u$ in $L^2(\Omega)$ and trace in $L^2(\partial\Omega)$. The space of such functions has been formally introduced in \eqref{eq:AmbientHilbertSpace},
\begin{equation}
    E(\Omega) = \left\{ u \in L^2(\Omega): \, \partial_{\bar z}u \in L^2(\Omega), \, u\in L^2(\partial\Omega) \right\}.
\end{equation}
As the following lemma shows, the space $E(\Omega)$ is a Hilbert space when endowed with the natural scalar product, which turns out to be compactly embedded in $L^2(\Omega)$.

\begin{lemma} \label{lemma:AmbientHilbertSpace}
    The space $E(\Omega)$ is a Hilbert space when endowed with the scalar product
    \begin{equation} \label{eq:ScalarProduct}
        \langle u, v \rangle_{E(\Omega)} := \int_\Omega u \, \overline v + \int_\Omega \partial_{\bar z} u \, \overline{\partial_{\bar z} v} + \int_{\partial \Omega} u \, \overline v \quad \text{for } u,v \in E(\Omega)
    \end{equation}
    and the associated norm $\|u\|_{E(\Omega)} := \sqrt{\langle u, u \rangle_{E(\Omega)}}$ for $u\in E(\Omega)$, on which $\mathcal C^\infty(\overline \Omega)$ is dense. Moreover, the space $E(\Omega)$ is compactly embedded in $L^2(\Omega)$.
\end{lemma}

\begin{proof}
    It is straightforward to verify that \eqref{eq:ScalarProduct} is a scalar product. To see that $E(\Omega)$ is a Hilbert space, take a Cauchy sequence $\{u_k\}_k \subset E(\Omega)$ with respect to $\|\cdot\|_{E(\Omega)}$. In particular,
    \begin{enumerate}[label=$(\roman*)$]
        \item $\{u_k\}_k$ is a Cauchy sequence in $L^2(\Omega)$. Hence, there exists $u_\Omega\in L^2(\Omega)$ such that $u_k \to u_\Omega$ in $L^2(\Omega)$, as $k\to +\infty$;
        \item $\{\partial_{\bar z} u_k\}_k$ is a Cauchy sequence in $L^2(\Omega)$. Hence, there exists $v\in L^2(\Omega)$ such that $\partial_{\bar z} u_k \to v$ in $L^2(\Omega)$, as $k\to +\infty$;
        \item $\{u_k\}_k$ is a Cauchy sequence in $L^2(\partial\Omega)$. Hence, there exists $u_{\partial\Omega}\in L^2(\partial\Omega)$ such that $u_k \to u_{\partial\Omega}$ in $L^2(\partial\Omega)$, as $k\to +\infty$.
    \end{enumerate}
    Then, on the one hand, for every $\varphi\in \mathcal C^\infty_c(\Omega)$,
    \begin{equation}
        \int_\Omega u_\Omega \, \overline{\partial_z \varphi} = \underset{k\to+\infty}{\lim} \, \int_\Omega u_k \, \overline{\partial_z \varphi} = \underset{k\to+\infty}{\lim} \, -\int_\Omega \partial_{\bar z} u_k \, \overline \varphi = -\int_\Omega v \, \overline \varphi.
    \end{equation}
    As a consequence, $\partial_{\bar z}u_\Omega = v$ first in the distributional sense, and then in $L^2(\Omega)$. In particular, $u_\Omega\in \mathrm{Dom}(\partial_{\mathrm{h}})$, where we recall that $\partial_{\mathrm{h}}$ is defined in \Cref{sec:Wirtinger}. Knowing this, on the other hand, for every $\varphi \in H^{1/2}(\partial\Omega)$ with extension in $\overline \Omega$ denoted again by $\varphi$, by the Green's formulas \eqref{eq:GreenFormulas} and the property of the pairing \eqref{eq:Brezis}, we have
    \begin{equation}
        \begin{split}
            \frac{1}{2} \overline{ \langle u_{\partial\Omega}, \overline \nu \varphi \rangle}_{H^{-1/2}(\partial\Omega),H^{1/2}(\partial\Omega)} & = \frac{1}{2} \int_{\partial\Omega} \nu u_{\partial\Omega} \, \overline \varphi = \underset{k\to+\infty}{\lim} \, \frac{1}{2} \int_{\partial\Omega} \nu u_k \, \overline \varphi = \underset{k\to+\infty}{\lim} \, \int_\Omega \partial_{\bar z} (u_k \, \overline \varphi) \\
            & = \underset{k\to+\infty}{\lim} \, \int_\Omega \partial_{\bar z} u_k \, \overline \varphi + \int_\Omega u_k \, \overline{\partial_z \varphi} = \int_\Omega \partial_{\bar z} u_\Omega \, \overline \varphi + \int_\Omega u_\Omega \, \overline{ \partial_z \varphi} \\
            & = \frac{1}{2} \overline{ \langle t_{\partial\Omega}u_\Omega, \overline \nu \varphi \rangle}_{H^{-1/2}(\partial\Omega),H^{1/2}(\partial\Omega)}
        \end{split}
    \end{equation}
    hence $t_{\partial\Omega} u_\Omega = u_{\partial\Omega}$ first in $H^{-1/2}(\partial\Omega)$, and then in $L^2(\partial\Omega)$. In conclusion, $u_\Omega\in E(\Omega)$ and $u\to u_\Omega$ in $E(\Omega)$ as $k\to +\infty$, as desired. 
    
    Next, we show that $\mathcal C^\infty(\overline\Omega)$ is dense in $E(\Omega)$. Assume that there exists $u\in E(\Omega)$ orthogonal to every $\varphi\in \mathcal C^\infty(\overline \Omega)$, that is, such that
    \begin{equation} \label{eq:auxOrth}
        0 = \langle u, \varphi \rangle_{E(\Omega)} = \langle u, \varphi \rangle_{L^2(\Omega)} + \langle \partial_{\bar z}u, \partial_{\bar z}\varphi \rangle_{L^2(\Omega)} + \langle u, \varphi \rangle_{L^2(\partial \Omega)}, \quad \text{for every } \varphi\in \mathcal C^\infty(\overline \Omega).
    \end{equation}
    We want to see that $u=0$. On the one hand, from \eqref{eq:auxOrth}, for every $\varphi\in \mathcal C^\infty_c(\Omega)$ it holds
    \begin{equation}
        \int_\Omega \partial_{\bar z} u \, \overline{\partial_{\bar z} \varphi} = -\int_\Omega u \, \overline \varphi,
    \end{equation}
    hence $\partial_z\partial_{\bar z} u = u$ first in the distributional sense, and then in $L^2(\Omega)$. In particular, $\partial_{\bar z}u \in \mathrm{Dom}(\partial_{\mathrm{ah}})$. Knowing this, for every $\varphi \in \mathcal C^\infty(\overline\Omega)$, by the Green's formulas \eqref{eq:GreenFormulas} and the property of the pairing \eqref{eq:Brezis}, we have
    \begin{equation}
        \begin{split}
            \frac{1}{2} \overline{ \langle t_{\partial\Omega} \partial_{\bar z} u, \nu \varphi \rangle}_{H^{-1/2}(\partial\Omega),H^{1/2}(\partial\Omega)} & = \int_\Omega \partial_z \partial_{\bar z} u \, \overline \varphi + \int_\Omega \partial_{\bar z} u \, \overline{\partial_{\bar z} \varphi} = \int_\Omega u \, \overline \varphi + \int_\Omega \partial_{\bar z} u \, \overline{\partial_{\bar z} \varphi} = - \int_{\partial\Omega} u \, \overline \varphi \\
            & = - \overline{ \langle t_{\partial\Omega} u, \varphi \rangle}_{H^{-1/2}(\partial\Omega),H^{1/2}(\partial\Omega)},
        \end{split}
    \end{equation}
    where in the second equality we have used that $\partial_z\partial_{\bar z} u = u$, and in the third equality we have used \eqref{eq:auxOrth}. Hence, $\frac{1}{2} t_{\partial\Omega}( \partial_{\bar z}u) = - \nu t_{\partial\Omega}u$ first in $H^{-1/2}(\partial\Omega)$, and then in $L^2(\partial\Omega)$. That is, $u$ solves the eigenvalue problem
    \begin{equation}
        \begin{cases}
            \Delta u = 4u & \text{in } L^2(\Omega), \\
            2\bar \nu \partial_{\bar z}u + 4u = 0 & \text{in } L^2(\partial\Omega).
        \end{cases}
    \end{equation}
    However, multiplying the PDE by $\bar u$ and integrating in $\Omega$ we get, using the divergence theorem,
    \begin{equation}
        0 \leq 4 \int_\Omega |u|^2 = \int_\Omega \Delta u \, \overline u = -4 \int_\Omega |\partial_{\bar z} u|^2 + 2\int_{\partial\Omega} \bar \nu \partial_{\bar z}u \, \overline u = -4 \int_\Omega |\partial_{\bar z} u|^2 - 4 \int_{\partial\Omega} |u|^2 \leq 0,
    \end{equation}
    from where we conclude that $u=0$, as desired. In conclusion, $\mathcal C^\infty(\overline \Omega)$ is dense in $E(\Omega)$ with respect to the norm of $E(\Omega)$. 

    We conclude showing that $E(\Omega)$ is compactly embedded in $L^2(\Omega)$. To this end, let $\{u_j\}_j \subset E(\Omega)$ be a uniformly bounded sequence. Our goal is to find a strongly convergent subsequence in $L^2(\Omega)$. To this end, we decompose $u_j$ as $u_j = (u_j)_\mathrm{h} + (u_j)_\perp$ with $(u_j)_\mathrm{h}\in A^2(\Omega)$ and $(u_j)_\perp \in A^2(\Omega)^\perp$, as in \Cref{lemma:BergmanCharacterization}. By uniform boundedness and orthogonality in $L^2(\Omega)$, there exists $C>0$ such that
    \begin{equation} \label{eq:BoundedSequence}
       \int_\Omega |(u_j)_\mathrm{h}|^2 + \int_\Omega |(u_j)_\perp|^2 + \int_\Omega |\partial_{\bar z} (u_j)_\perp |^2 + \int_{\partial\Omega} |(u_j)_\perp + (u_j)_\mathrm{h}|^2 = \|u_j\|_{E(\Omega)}^2 \leq C \quad \text{for all } j.
    \end{equation}
    We will show that both $\{(u_j)_\mathrm{h}\}_j$ and $\{(u_j)_\perp\}_j$ have strongly convergent subsequences in $L^2(\Omega)$.

    First, by \Cref{lemma:GainRegularityOrthogonal} and \eqref{eq:BoundedSequence}, we get
    \begin{equation} \label{eq:BergamnH1Uniform}
        \|(u_j)_\perp\|_{H^1(\Omega)}^2 \leq C_\Omega  \|\partial_{\bar z} (u_j)_\perp\|_{L^2(\Omega)}^2 \leq C_\Omega C
    \end{equation}
    for some constant $C_\Omega>0$ depending only on $\Omega$. Since $H^1(\Omega)$ is compactly embedded in $L^2(\Omega)$ ---because $\Omega$ is bounded, see \cite[Theorem 9.16]{Brezis2011}---, there exist $u^\perp_\star\in L^2(\Omega)$ and a subsequence, which we call again $\{(u_j)_\perp\}_j$, such that $\|(u_j)_\perp-u^\perp_\star\|_{L^2(\Omega)} \to 0$ as $j\to + \infty$.

    Next, by the triangle inequality, \eqref{eq:BoundedSequence}, the trace theorem, and \eqref{eq:BergamnH1Uniform}, for this subsequence of indices $j$ we get
    \begin{equation} \label{eq:BergmanBdryUnif}
        \int_{\partial\Omega} |(u_j)_\mathrm{h}|^2 \leq 2 \int_{\partial\Omega} |(u_j)_\mathrm{h} + (u_j)_\perp|^2 + 2 \int_{\partial\Omega} |(u_j)_\perp|^2 \leq 2C + 2C_\Omega' C,
    \end{equation}
    for some constant $C_\Omega'>0$ depending only on $\Omega$. Since $(u_j)_\mathrm{h} = \Phi_{\mathrm{h}} (u_j)_\mathrm{h}$ by \Cref{rmk:CauchyPompeiu}, combining \Cref{lemma:SingleLayer} for $s=0$ and \eqref{eq:BergmanBdryUnif} we get
    \begin{equation}
        \|(u_j)_\mathrm{h}\|_{H^{1/2}(\Omega)} \leq K,
    \end{equation}
    for some constant $K>0$ depending only on $\Omega$ and the bound in \eqref{eq:BergmanBdryUnif}. Again, since $H^{1/2}(\Omega)$ is compactly embedded in $L^2(\Omega)$ ---because $\Omega$ is bounded, see \cite[Lemma 6.10]{Palatucci2012}---, there exist $u^\mathrm{h}_\star\in L^2(\Omega)$ and a subsequence, which we call again $\{(u_j)_\mathrm{h}\}_j$, such that $\|(u_j)_\mathrm{h}-u^\mathrm{h}_\star\|_{L^2(\Omega)} \to 0$ as $j\to+ \infty$.

    In conclusion, for this subsequence of indices $j$, the subsequence $\{u_j\}_j$ converges strongly to $u:= u^\perp_\star+u^\mathrm{h}_\star$ in $L^2(\Omega)$. This concludes the proof.
\end{proof}

In view of the equality \eqref{eq:DivergenceToRodzin} resulting from the problem \eqref{eq:BvPRodzinLaplacian}, it is natural to consider, for each fixed $a>0$, the sesquilinear form $B_a(\cdot,\cdot): E(\Omega) \times E(\Omega) \rightarrow \C$ defined by 
\begin{equation} \label{eq:SesquilinearForm}
    B_a(u,v) := 4 \int_\Omega \partial_{\bar z} u \, \overline{\partial_{\bar z} v} + a \int_{\partial \Omega} u \, \overline v.
\end{equation}

This form is the linear in the parameter counterpart of the quadratic form $q_E^\Omega$ appearing in~\cite{Antunes2021}. As a consequence of \Cref{lemma:AmbientHilbertSpace}, the form $B_a$ is a sesquilinear form defined in a Hilbert space. The following lemma summarizes some of its properties.

\begin{lemma} \label{lemma:PropertiesSesquilinearForm}
    For every $a>0$, the sesquilinear form $B_a$ defined in \eqref{eq:SesquilinearForm} is 
    \begin{enumerate}[label=$(\roman*)$]
        \item densely defined in $L^2(\Omega)$ independently of $a$,
        \item symmetric (in the sense of \cite[pg. 309]{Kato1995}),
        \item sectorial (in the sense of \cite[pg. 310]{Kato1995}), and 
        \item closed (in the sense of \cite[pg. 313]{Kato1995}).
    \end{enumerate}
    Moreover, $\mathbb C \ni a \mapsto B_a(u,v)$ is holomorphic in $a$ for each fixed $u,v\in E(\Omega)$.
\end{lemma}

\begin{proof}
    Since $\Dom(B_a) = E(\Omega) \supset \mathcal C^\infty_c(\Omega)$, it is clear that $B_a$ is densely defined in $L^2(\Omega)$ independently of $a$. It is also clear that $B_a(v,u) = \overline{B_a(u,v)}$ for every $u,v\in E(\Omega)$, hence that $B_a$ is symmetric. That $B_a$ is sectorial (in fact, accretive) is also clear from the fact that
    \begin{equation}
        B_a(u,u) = 4\int_\Omega |\partial_{\bar z} u|^2 + a\int_{\partial\Omega} |u|^2 \in [0, +\infty) \quad \text{for every } u\in E(\Omega).
    \end{equation}
    Finally, it is also clear that $B_a(u,v)$ is holomorphic in $a$ for each fixed $u,v\in E(\Omega)$, since $B_a(u,v)$ is linear in $a$.

    It only remains to show that $B_a$ is closed. To this end, let $\{u_k\}_k \subset E(\Omega)$ be such that $u_k \to u$ in $L^2(\Omega)$ as $k\to+\infty$, and such that $B_a(u_k-u_j, u_k-u_j) \to 0$ as $j,k\to +\infty$. We want to show that $u\in E(\Omega)$ and $B_a(u_k-u, u_k-u) \to 0$ as $k\to+\infty$.

    To see that $u\in E(\Omega)$, notice that $\{u_k\}_k$ is a Cauchy sequence in $L^2(\Omega)$, and since $B_a(u_k-u_j, u_k-u_j) \to 0$ as $j,k \to +\infty$, $\{\partial_{\bar z} u_k \}_k$  and $\{u_k\}_k$ are Cauchy sequences in $L^2(\Omega)$ and $L^2(\partial\Omega)$, respectively. Therefore, $\{u_k\}_k$ is a Cauchy sequence in the Hilbert space $E(\Omega)$, hence there exists $u_\star \in E(\Omega)$ such that $u_k$ converges to $u_\star$ in $E(\Omega)$. It is clear that $u=u_\star$ in $L^2(\Omega)$. Actually, proceeding similarly as in the proof of \Cref{lemma:AmbientHilbertSpace}, it holds that $\partial_{\bar z} u = \partial_{\bar z} u_\star$ in the distributional sense, and $t_{\partial\Omega} u = t_{\partial\Omega} u_\star$ in $H^{-1/2}(\partial\Omega)$. That is, $u$ is in $E(\Omega)$.
    
    Finally, since $u_k$ converges to $u$ in $E(\Omega)$, it is clear that
    \begin{equation}
        B_a(u_k-u, u_k -u) = 4\|\partial_{\bar z} u_k - \partial_{\bar z} u \|_{L^2(\Omega)}^2 + a \|u_k-u\|_{L^2(\partial\Omega)}^2 \to 0 \text{ as } k\to+\infty,
    \end{equation}
    which concludes the proof that $B_a$ is closed.
\end{proof}

\subsection{The associated self-adjoint operator} \label{sec:AssociatedOperator}

As a consequence of \Cref{lemma:PropertiesSesquilinearForm}, by Kato's first representation theorem \cite[Theorem 2.1 in Section VI]{Kato1995}, for every $a>0$ there exists a unique self-adjoint operator $R_a$ acting in $L^2(\Omega)$ associated to $B_a$. More specifically, the operator $R_a$ is uniquely determined by the condition that 
\begin{equation} \label{eq:DeterminationOfRodzin}
    \begin{split}
        & \Dom(R_a)\subseteq \Dom(B_a) = E(\Omega), \, \text{ and} \\
        & B_a(u,v) = \langle R_a u, v \rangle_{L^2(\Omega)}, \, \text{ for every } u\in \Dom(R_a) \text{ and } v\in E(\Omega).
    \end{split}
\end{equation}
Moreover, by \cite[Theorem 2.1 iii) in Section VI]{Kato1995}, $R_a$ has the property that
\begin{equation} \label{eq:CharacterizationOfRodzin}
    \begin{split}
        & \text{if } u\in E(\Omega), w\in L^2(\Omega) \text{ and } B_a(u,v) = \langle w, v \rangle_{L^2(\Omega)} \text{ holds for every } v\in E(\Omega), \\
        & \text{then } u\in \Dom(R_a) \text{ and } R_a u =w.
    \end{split}
\end{equation}
Notice that \eqref{eq:DeterminationOfRodzin} exactly corresponds to the weak form \eqref{eq:DivergenceToRodzin} as long as $R_a u = f$ in $L^2(\Omega)$. This is proven in the following result, which ensures that $R_a$ coincides with the $\overline\partial$-Robin Laplacian 
\begin{equation}
    \begin{split}
        \Dom(\Rodzin_a) & = \left\{u\in H^1(\Omega): \, \partial_{\bar z} u \in H^1(\Omega), \, 2\bar \nu \partial_{\bar z}u + au = 0 \text{ in } H^{1/2}(\partial \Omega) \right\}, \\
        \Rodzin_a u & = -\Delta u \quad \text{for all } u \in \Dom(\Rodzin_a),
    \end{split}
\end{equation}
introduced in \eqref{eq:RodzinLaplacian}.

\begin{proposition} \label{prop:RodzinLaplacian}
    For every $a>0$, there holds $R_a=\Rodzin_a$. That is, the operator $\Rodzin_a$ is the unique self-adjoint operator associated, through the condition \eqref{eq:DeterminationOfRodzin}, to the sesquilinear form $B_a$ defined in \eqref{eq:SesquilinearForm}. 
\end{proposition}

This proposition is a precise restatement of the first part of \Cref{thm:IntroRodzinLaplacian}. The second part of \Cref{thm:IntroRodzinLaplacian} is proven in \Cref{prop:BoundedResolvent} below.

\begin{proof}[Proof of \Cref{prop:RodzinLaplacian}]
    The goal is to show that $R_a = \Rodzin_a$. We start proving that $\Dom(R_a) \subseteq \Dom(\Rodzin_a)$ and $R_a u = -\Delta u$ for every $u\in \Dom(R_a)$. To this end, fix $u\in \Dom(R_a) \subset E(\Omega)$. First, for every $\varphi \in \mathcal C^\infty_c(\Omega)$, by \eqref{eq:DeterminationOfRodzin} 
    \begin{equation}
        \langle R_a u, \varphi \rangle_{L^2(\Omega)} = B_a(u, \varphi) = 4 \int_\Omega \partial_{\bar z} u \, \overline{\partial_{\bar z} \varphi},
    \end{equation}
    hence $\partial_z (4\partial_{\bar z} u) = -R_a u$ first in the distributional sense and, since $R_a u \in L^2(\Omega)$, then in $L^2(\Omega)$. In particular, $R_a u = -\Delta u \in L^2(\Omega)$. As a consequence, $\partial_{\bar z} u \in \mathrm{Dom}(\partial_{\mathrm{ah}})$. Knowing this, for every $\varphi\in H^{1/2}(\partial\Omega)$ with extension in $\overline \Omega$ denoted again by $\varphi$, by the Green's formulas \eqref{eq:GreenFormulas} and the property of the pairing \eqref{eq:Brezis}, we have
    \begin{equation}
        \begin{split}
            \langle R_a u, \varphi \rangle_{L^2(\Omega)} & = \langle -\Delta u, \varphi \rangle_{L^2(\Omega)} = 4\int_\Omega \partial_{\bar z} u \, \overline{\partial_{\bar z} \varphi} - 2 \overline{ \langle  \partial_{\bar z} u, \nu \varphi \rangle}_{H^{-1/2}(\partial\Omega), H^{1/2}(\partial\Omega)} \\
            & = B_a(u, \varphi) - a\int_{\partial\Omega} u\, \overline{\varphi} - \overline{ \langle 2\bar\nu \partial_{\bar z} u , \varphi \rangle}_{H^{-1/2}(\partial\Omega), H^{1/2}(\partial\Omega)} \\
            & = B_a(u, \varphi) - \overline{ \langle 2\bar\nu \partial_{\bar z} u + au, \varphi \rangle}_{H^{-1/2}(\partial\Omega), H^{1/2}(\partial\Omega)}.
        \end{split}
    \end{equation}
    Hence, using that $\langle R_a u, \varphi \rangle_{L^2(\Omega)} = B_a(u, \varphi)$ by \eqref{eq:DeterminationOfRodzin}, we get $2\bar \nu \partial_{\bar z} u + a u = 0$ in $H^{-1/2}(\partial\Omega)$. In summary, $u\in \Dom(R_a) \subseteq E(\Omega)$ is such that $u, \, \partial_{\bar z} u, \, \Delta u \in L^2(\Omega)$ and $2\bar \nu \partial_{\bar z} u + a u = 0$ in $H^{-1/2}(\partial\Omega)$, with $a>0$. Thus, by \Cref{lemma:AntunesRegularity} we have that $u\in H^1(\Omega)$, that $\partial_{\bar z} u\in H^1(\Omega)$, and that $2\bar \nu \partial_{\bar z} u + a u = 0$ holds in $H^{1/2}(\partial\Omega)$. In conclusion, $u\in \Dom(\Rodzin_a)$ and $R_a u = -\Delta u = \Rodzin_a u$.

    We finally show that $\Dom(\Rodzin_a) \subseteq \Dom(R_a)$. To this end, let $u\in \Dom(\Rodzin_a) \subset E(\Omega)$. Clearly $\Delta u = 4 \partial_z(\partial_{\bar z} u) \in L^2(\Omega)$. Then, for every $v\in E(\Omega)$, by the divergence theorem and using the boundary condition for $u$, it holds
    \begin{equation}
        B_a(u,v) = 4 \int_\Omega \partial_{\bar z} u \, \overline{\partial_{\bar z} v} + a \int_{\partial\Omega} u \, \overline v = \int_\Omega -\Delta u \, \overline v + 2\int_{\partial\Omega} \bar \nu \partial_{\bar z} u \, \overline v + a \int_{\partial\Omega} u \, \overline v = \langle -\Delta u, v\rangle_{L^2(\Omega)}.
    \end{equation}
    As a consequence of \eqref{eq:CharacterizationOfRodzin}, $u\in \Dom(R_a)$ and $R_a u = -\Delta u = \Rodzin_a u$.
\end{proof}

The functions in $\Dom(\Rodzin_a)$ are $H^1(\Omega)$ regular. As the following lemma shows, the $H^1(\Omega)-$ norm of a function $u\in \Dom(\Rodzin_a)$ is controlled by the norms $\|u\|_{L^2(\Omega)}$, $\|\partial_{\bar z} u\|_{L^2(\Omega)}$ and $\|\Delta u\|_{L^2(\Omega)}$, and similarly for the $H^1(\Omega)-$norm of $\partial_{\bar z} u$. The following quantitative estimates will be crucial for bounding the resolvent of $\Rodzin_a$, and hence proving discreteness of the spectrum of $\Rodzin_a$.

\begin{lemma} \label{lemma:RegularityEstimatesRodzin}
    Let $a>0$. For every $u\in \Dom(\Rodzin_a)$, 
    \begin{equation}
        \begin{split}
            \|\nabla u\|_{L^2(\Omega)} & \leq C_\Omega \left( \|u\|_{L^2(\Omega)} + \|\partial_{\bar z}u\|_{L^2(\Omega)} + \frac{1}{a} \|\Delta u\|_{L^2(\Omega)} \right), \quad  \text{and} \\[2pt]
            \|\nabla \partial_{\bar z} u\|_{L^2(\Omega)} & \leq C_\Omega \left( a \|u\|_{L^2(\Omega)} + a \|\partial_{\bar z}u\|_{L^2(\Omega)} + \|\Delta u\|_{L^2(\Omega)} \right),
        \end{split}
    \end{equation}
    for some constant $C_\Omega>0$ depending only on $\Omega$.
\end{lemma}

In order to prove these quantitative estimates, we shall use the relations of the following lemma. In its statement, it appears the boundary term $\overline{\langle \partial_\tau v, v \rangle}_{H^{-1/2}(\partial\Omega), H^{1/2}(\partial\Omega)}$, which, in view of the property of the pairing \eqref{eq:Brezis}, is well defined for $v\in \mathcal C^\infty(\overline \Omega)$. Recalling from \Cref{sec:Introduction} that $\partial_\tau = \nu_1\partial_2 - \nu_2\partial_1$ is the tangential derivative in $\partial\Omega$, by the divergence theorem we have
    \begin{equation}
        \overline{\langle \partial_\tau v, v \rangle}_{H^{-1/2}(\partial\Omega), H^{1/2}(\partial\Omega)} = \int_\Omega \mathrm{div} \left( \left[ \begin{pmatrix}
                0 & 1 \\
                -1 & 0
            \end{pmatrix} \nabla v \right ] \, \overline{v} \right) = \int_\Omega \left[ \begin{pmatrix}
                0 & 1 \\
                -1 & 0
            \end{pmatrix} \nabla v \right ] \cdot \overline{\nabla v},
    \end{equation}
which gives meaning to the pairing for $v\in H^1(\Omega)$.

\begin{lemma} \label{lemma:RelationGradientDz}
    For every $v\in H^1(\Omega)$, there holds
    \begin{equation}
        \begin{split}
            \int_\Omega |\nabla v|^2  & = 4 \int_\Omega |\partial_z v|^2 + i \overline{ \langle \partial_\tau v, v\rangle}_{H^{-1/2}(\partial\Omega), H^{1/2}(\partial\Omega)}, \quad \text{and} \\[2pt]
            \int_\Omega |\nabla v|^2 & = 4 \int_\Omega |\partial_{\bar z} v|^2 - i \overline{ \langle \partial_\tau v, v\rangle}_{H^{-1/2}(\partial\Omega), H^{1/2}(\partial\Omega)}.
        \end{split}
    \end{equation}
\end{lemma}

\begin{proof}
    We only prove the first identity, because the proof for the second identity is analogous. To this end, we first show the result for smooth functions. If $v \in \mathcal C^\infty(\overline \Omega)$, using that $2\nu \partial_z v = \partial_\nu v -i \partial_\tau v$ on $\partial\Omega$, by the Green's formulas \eqref{eq:GreenFormulas} and the property of the pairing \eqref{eq:Brezis}, we have
    \begin{equation} \label{eq:densityDivergenceThm}
        \begin{split}
            \int_\Omega |\nabla v|^2 & = -\int_\Omega \Delta v \, \overline v + \int_{\partial\Omega} \partial_\nu v \, \overline v = 4 \int_\Omega |\partial_z v|^2 - 2\int_{\partial\Omega} \nu \partial_z v \, \overline v + \int_{\partial\Omega} \partial_\nu v \, \overline v \\
            & = 4 \int_\Omega |\partial_z v|^2 + i\int_{\partial\Omega} \partial_\tau v \, \overline v = 4 \int_\Omega |\partial_z v|^2 + i \overline{\langle \partial_\tau v, v\rangle}_{H^{-1/2}(\partial\Omega), H^{1/2}(\partial\Omega)}.
        \end{split}
    \end{equation}
    The result holds for $v\in H^1(\Omega)$ by density.
\end{proof}

\begin{proof}[Proof of \Cref{lemma:RegularityEstimatesRodzin}]
    Let $u\in \Dom(\Rodzin_a)$. On the one hand, since $u \in H^1(\Omega)$, by \Cref{lemma:RelationGradientDz} we have
    \begin{equation} \label{eq:L2GradientTangential1}
        \int_\Omega |\nabla u|^2 = 4 \int_\Omega |\partial_{\bar z} u|^2 - i \overline{ \langle \partial_\tau u, u\rangle}_{H^{-1/2}(\partial\Omega), H^{1/2}(\partial\Omega)}.
    \end{equation}
    On the other hand, since $\partial_{\bar z} u \in H^1(\Omega)$, by \Cref{lemma:RelationGradientDz} we have 
    \begin{equation} \label{eq:L2GradientTangential2}
        \int_\Omega |\nabla \partial_{\bar z} u|^2 = \dfrac{1}{4} \int_\Omega |\Delta u|^2 +i \overline{ \langle \partial_\tau \partial_{\bar z} u, \partial_{\bar z} u \rangle}_{H^{-1/2}(\partial\Omega), H^{1/2}(\partial\Omega)}.
    \end{equation}
    Using that $2\bar\nu \partial_{\bar z} u +au=0$ in $H^{1/2}(\partial\Omega)$, differentiating tangentially we obtain
    \begin{equation}
        0 = 2(\partial_\tau \bar \nu) \partial_{\bar z}u + 2\bar\nu\partial_\tau (\partial_{\bar z}u) + a \partial_\tau u \quad \text{in } H^{-1/2}(\partial\Omega),
    \end{equation}
    and hence
    \begin{equation} \label{eq:TangentialDomRodzin}
        \partial_\tau u = \dfrac{-1}{a} \left( 2(\partial_\tau \bar \nu) \partial_{\bar z}u + 2\bar\nu\partial_\tau (\partial_{\bar z}u) \right) \quad \text{in } H^{-1/2}(\partial\Omega).
    \end{equation}
    Substituting this in \eqref{eq:L2GradientTangential1}, we get
    \begin{equation}
        \begin{split}
            \int_\Omega |\nabla u |^2 & = 4\int_\Omega |\partial_{\bar z} u|^2 + \dfrac{2i}{a} \overline{ \langle (\partial_\tau \bar \nu) \partial_{\bar z}u , u \rangle}_{H^{-1/2}(\partial\Omega), H^{1/2}(\partial\Omega)} + \dfrac{2i}{a} \overline{ \langle \bar\nu\partial_\tau (\partial_{\bar z}u) , u \rangle}_{H^{-1/2}(\partial\Omega), H^{1/2}(\partial\Omega)} \\
            & = 4\int_\Omega |\partial_{\bar z} u|^2 + \dfrac{2i}{a} \overline{ \langle (\partial_\tau \bar \nu) \partial_{\bar z}u , u \rangle}_{H^{-1/2}(\partial\Omega), H^{1/2}(\partial\Omega)} - \dfrac{4}{a^2}i \overline{ \langle \partial_\tau (\partial_{\bar z}u) , \partial_{\bar z} u \rangle}_{H^{-1/2}(\partial\Omega), H^{1/2}(\partial\Omega)},
        \end{split}
    \end{equation}
    where in the last equality we have used the boundary condition $u = -\frac{2}{a} \bar \nu \partial_{\bar z} u$ in $H^{1/2}(\partial\Omega)$. Replacing the last pairing by the expression obtained in \eqref{eq:L2GradientTangential2}, we obtain
    \begin{equation}
        \int_\Omega |\nabla u |^2 = 4\int_\Omega |\partial_{\bar z} u|^2 + \dfrac{1}{a^2} \int_\Omega |\Delta u|^2 + \dfrac{2i}{a} \overline{ \langle (\partial_\tau \bar \nu) \partial_{\bar z}u , u \rangle}_{H^{-1/2}(\partial\Omega), H^{1/2}(\partial\Omega)} - \dfrac{4}{a^2} \int_\Omega |\nabla \partial_{\bar z} u|^2.
    \end{equation}
    Using again the boundary condition $2\partial_{\bar z} u = -\nu a u$ in $H^{1/2}(\partial\Omega)$, we get
    \begin{equation}
        \int_\Omega |\nabla u |^2 + \dfrac{4}{a^2} \int_\Omega |\nabla \partial_{\bar z} u|^2 = 4\int_\Omega |\partial_{\bar z} u|^2 + \dfrac{1}{a^2} \int_\Omega |\Delta u|^2 - i \int_{\partial\Omega} (\partial_\tau \bar \nu) \nu |u|^2,
    \end{equation}
    where we have used the property of the pairing \eqref{eq:Brezis}. The boundary integral is related with the mean curvature $\kappa$ of $\partial\Omega$. Indeed, using that $\partial_\tau \nu = - \kappa \tau$, the identifications $\nu=(\nu_1,\nu_2)\equiv\nu_1+i\nu_2$ and $\tau=(-\nu_2,\nu_1)\equiv -\nu_2+i\nu_1=i\nu$, and that $0=\partial_\tau (\nu \overline \nu)$, one can straightforwardly verify that $(\partial_\tau \bar \nu) \nu = i\kappa$. Hence, we obtain that
    \begin{equation} \label{eq:NormsCurvature}
        \int_\Omega |\nabla u |^2 + \dfrac{4}{a^2} \int_\Omega |\nabla \partial_{\bar z} u|^2 = 4\int_\Omega |\partial_{\bar z} u|^2 + \dfrac{1}{a^2} \int_\Omega |\Delta u|^2 + \int_{\partial\Omega} \kappa |u|^2.
    \end{equation}

    Since $\Omega$ is $\mathcal C^2$ and bounded, there exists $C_{\partial\Omega}>0$ such that $\|\kappa \|_{L^\infty(\partial\Omega)} \leq C_{\partial\Omega}$. Using that $\Omega$ is a bounded domain with $\mathcal C^2$ boundary, and the trace theorem \cite[Theorem 5.5.1]{Evans2010} combined with Young's inequality with $\epsilon>0$, from \eqref{eq:NormsCurvature} we see that
    \begin{equation}
        \begin{split}
            \int_\Omega |\nabla u |^2 + \dfrac{4}{a^2} \int_\Omega |\nabla \partial_{\bar z} u|^2 & \leq 4\int_\Omega |\partial_{\bar z} u|^2 + \dfrac{1}{a^2} \int_\Omega |\Delta u|^2 + C_{\partial\Omega} \|u \|_{L^2(\partial\Omega)}^2 \\
            & \leq 4\int_\Omega |\partial_{\bar z} u|^2 + \dfrac{1}{a^2} \int_\Omega |\Delta u|^2 + C_{\partial\Omega} \left( \epsilon \| \nabla u \|_{L^2(\Omega)}^2 + C_\epsilon \| u\|_{L^2(\Omega)}^2 \right),
        \end{split}
    \end{equation}
    for some $C_\epsilon>0$ depending only on $\epsilon$ and $\Omega$. Taking $\epsilon>0$ such that $C_{\partial\Omega} \epsilon \leq 1/2$, we obtain
    \begin{equation}
        \dfrac{1}{2} \|\nabla u \|_{L^2(\Omega)}^2 + \dfrac{4}{a^2} \|\nabla \partial_{\bar z} u\|_{L^2(\Omega)}^2 \leq 4\|\partial_{\bar z} u\|_{L^2(\Omega)}^2 + \dfrac{1}{a^2} \|\Delta u\|_{L^2(\Omega)}^2 + C_\Omega' \|u\|_{L^2(\Omega)}^2,
    \end{equation}
    for some $C_\Omega'>0$ depending only on $\Omega$, and the lemma follows.
\end{proof}

\subsection{Spectral properties of $\Rodzin_a$} \label{sec:LocalSpectralProperties}

As aforementioned, the quantitative estimates of \Cref{lemma:RegularityEstimatesRodzin} lead to the discreteness of the spectrum of the $\overline\partial$-Robin Laplacian $\Rodzin_a$.

\begin{proposition} \label{prop:BoundedResolvent}
    For every $a>0$ and $\lambda \in \C\setminus\R$, there exists a constant $C_{\Omega,\lambda}>0$ depending only on $\Omega$ and $\lambda$ such that 
    \begin{equation}
        \| (\Rodzin_a-\lambda)^{-1} f \|_{H^1(\Omega)} \leq C_{\Omega,\lambda}(1+1/a) \| f \|_{L^2(\Omega)} \quad \text{for all } f\in L^2(\Omega).
    \end{equation}
    As a consequence, $(\Rodzin_a-\lambda)^{-1}$ is a compact operator from $L^2(\Omega)$ to $L^2(\Omega)$, and $\sigma(\Rodzin_a)$ is purely discrete.
\end{proposition}

\begin{proof}
    Fix $f\in L^2(\Omega)$ and set $\varphi_a= (\Rodzin_a - \lambda)^{-1}f \in \Dom(\Rodzin_a)$, which is well defined because $\Rodzin_a$ is self-adjoint. In view of \Cref{lemma:RegularityEstimatesRodzin}, if we can bound the $L^2(\Omega)-$norms of $\varphi_a$, $\partial_{\bar z} \varphi_a$, and $\Delta \varphi_a$ in terms of the $L^2(\Omega)-$norm of $f$, we will be done.

    First, by self-adjointness of $\Rodzin_a$ ---see, for example, \cite[equation V.3.13]{Kato1995}---,
    \begin{equation} \label{eq:BoundedResolventL2L2}
        \|\varphi_a \|_{L^2(\Omega)} = \|(\Rodzin_a - \lambda)^{-1}f \|_{L^2(\Omega)} \leq \dfrac{1}{|\mathrm{Im}(\lambda)|} \|f\|_{L^2(\Omega)}.
    \end{equation}
    Next, notice that $\varphi_a$ solves the problem
    \begin{equation}
        \begin{cases}
            -\Delta \varphi_a - \lambda \varphi_a = f & \text{in } L^2(\Omega), \\
            2\bar \nu \partial_{\bar z} \varphi_a + a\varphi_a = 0 & \text{in } H^{1/2}(\partial \Omega).
        \end{cases}
    \end{equation}
    As a first consequence, by the triangle inequality and \eqref{eq:BoundedResolventL2L2},
    \begin{equation} \label{eq:L2boundLaplacian}
        \| \Delta \varphi_a\|_{L^2(\Omega)} = \| \lambda \varphi_a + f \|_{L^2(\Omega)} \leq |\lambda| \| \varphi_a\|_{L^2(\Omega)} + \| f\|_{L^2(\Omega)} \leq \left( 1 + \dfrac{|\lambda|}{|\mathrm{Im}(\lambda)|} \right) \|f\|_{L^2(\Omega)}.
    \end{equation}
    As a second consequence, multiplying the equality $-\Delta \varphi_a = f + \lambda \varphi_a$ by $\overline \varphi_a$ and integrating in $\Omega$, the divergence theorem and the boundary condition lead to
    \begin{equation}
        4 \int_\Omega |\partial_{\bar z} \varphi_a|^2 + a\int_{\partial\Omega} |\varphi_a|^2 = \int_\Omega f \, \overline \varphi_a + \lambda \int_\Omega |\varphi_a|^2.
    \end{equation}
    Then, since $a\geq 0$, by the triangle inequality and \eqref{eq:BoundedResolventL2L2} we can bound
    \begin{equation} \label{eq:L2boundDz}
        \| \partial_{\bar z} \varphi_a \|_{L^2(\Omega)}^2 \leq \left| \int_\Omega f \, \overline \varphi_a + \lambda \int_\Omega |\varphi_a|^2 \right| \leq \left(\dfrac{1}{|\mathrm{Im}(\lambda)|}+\dfrac{|\lambda|}{|\mathrm{Im}(\lambda)|^2}\right) \|f\|_{L^2(\Omega)}^2.
    \end{equation}
    Combining \Cref{lemma:RegularityEstimatesRodzin} with the estimates \eqref{eq:BoundedResolventL2L2}, \eqref{eq:L2boundLaplacian}, and \eqref{eq:L2boundDz}, we finally get
    \begin{equation}
        \|\varphi_a\|_{H^1(\Omega)} \leq C_\Omega(1+1/a) \left( \dfrac{1}{|\mathrm{Im}(\lambda)|} + \sqrt{\dfrac{1}{|\mathrm{Im}(\lambda)|}+\dfrac{|\lambda|}{|\mathrm{Im}(\lambda)|^2}} + 1 + \dfrac{|\lambda|}{|\mathrm{Im}(\lambda)|} \right) \|f\|_{L^2(\Omega)},
    \end{equation}
    proving the first part of the statement. To finish, since $H^1(\Omega)$ is compactly embedded in $L^2(\Omega)$ ---because $\Omega$ is bounded, see \cite[Theorem 9.16]{Brezis2011}---, from the previous inequality we deduce that $(\Rodzin_a -\lambda)^{-1}$ is a compact operator from $L^2(\Omega)$ to $L^2(\Omega)$. This, together with \cite[Proposition 8.8 in Appendix A, and the paragraph below it]{Taylor2011}, yields the discreteness of $\sigma(\Rodzin_a)$.
\end{proof}

\begin{proof}[Proof of \Cref{thm:IntroRodzinLaplacian}]
    The first part is the content of \Cref{prop:RodzinLaplacian}, and the second part is the content of \Cref{prop:BoundedResolvent}.
\end{proof}

We are now ready to prove the characterization of the ordered eigenvalues of the $\overline\partial$-Robin Laplacian in terms of min-max levels.

\begin{proof}[Proof of \Cref{thm:PropertiesMukOmega}]
    Let $a>0$ and $k\in\N$. By the min-max Theorem \cite[Theorem 4.14]{Teschl2014} and \Cref{prop:RodzinLaplacian}, the $k-$th eigenvalue of $\Rodzin_a$ is
    \begin{equation}
        \mu_k = \underset{\substack{ F\subset \mathrm{Dom}(\Rodzin_a) \\ \mathrm{dim}(F)=k }}{\inf} \, \, \underset{u\in F\setminus\{0\}}{\sup} \, \dfrac{4\int_\Omega |\partial_{\bar z} u|^2 + a \int_{\partial\Omega} |u|^2}{\int_\Omega |u|^2},
    \end{equation}
    where the infimum is taken among subspaces $F$ of $\mathrm{Dom}(\Rodzin_a)$, and for the sake of notation we omit the dependence on $a$. We want to show that $\mu_k$ coincides with the min-max level
    \begin{equation}
        M_k := \underset{\substack{ F\subset E(\Omega) \\ \mathrm{dim}(F)=k }}{\inf} \, \, \underset{u\in F\setminus\{0\}}{\sup} \, \dfrac{4\int_\Omega |\partial_{\bar z} u|^2 + a \int_{\partial\Omega} |u|^2}{\int_\Omega |u|^2},
    \end{equation}
    where the infimum is taken among subspaces $F$ of $E(\Omega) \supsetneq \mathrm{Dom}(\Rodzin_a)$. It is therefore clear that $0\leq M_k \leq \mu_k$. In order to show the reverse inequality, it shall be useful to rewrite $M_k$ as follows.
    
    By \Cref{lemma:PoincareTypeIneq}, for every $u\in E(\Omega)$ we have the Poincar\'e type inequality
    \begin{equation}
        \|u\|_{L^2(\Omega)}^2 \leq K_\Omega \left(\|\partial_{\bar z} u\|_{L^2(\Omega)}^2 + \|u\|_{L^2(\partial\Omega)}^2 \right).
    \end{equation}
    Hence, the norms $\|\cdot\|_{E(\Omega)}$ and $\|\cdot\|_a$ are equivalent on $E(\Omega)$, where we have set
    \begin{equation} \label{def:HilbertSpaceEa}
        \begin{split}
            \langle u, v \rangle_a & := B_a(u,v) = \displaystyle{ 4 \int_\Omega \partial_{\bar z} u \, \overline{\partial_{\bar z}v} + a \int_{\partial\Omega} u \, \overline v } \quad \text{and} \\
            \|u\|_a & := \sqrt{\langle u, u \rangle_a} = \left( 4\| \partial_{\bar z} u \|_{L^2(\Omega)}^2 + a \|u\|_{L^2(\partial\Omega)}^2 \right)^{1/2}.
        \end{split}
    \end{equation}
    The advantage of this new norm is that $M_k$ writes as
    \begin{equation} \label{eq:muOmegaEa}
        M_k = \underset{\substack{ F\subset E(\Omega) \\ \mathrm{dim}(F)=k }}{\inf} \, \, \underset{u\in F\setminus\{0\}}{\sup} \, Q_a(u), \quad \text{where} \quad  Q_a(u) :=  \dfrac{\|u\|_a^2}{\|u\|_{L^2(\Omega)}^2}.
    \end{equation}

    We want to show that $M_k$ is attained by some $u_k \in E(\Omega)$, which is actually an eigenfunction of $\Rodzin_a$ of eigenvalue $M_k$ orthogonal to $u_\ell$, for $\ell=1,\dots, k-1$ ---the previous eigenfunctions of $\Rodzin_a$ of eigenvalue $M_\ell$. In order to see this, we auxiliary show, by induction on $k\in\N$, the following claim. 

    \textbf{\underline{Claim:}} the min-max levels $M_k$ are positive and coincide with
    \begin{equation}
        M_k^\perp := \underset{u \in E(\Omega)\cap \mathrm{span}\{u_{1}, \dots, u_{k-1}\}^\perp\setminus\{0\}}{\inf} \ Q_a(u),
    \end{equation}
    and this infimum is attained by an eigenfunction $u_k\in\Dom(\Rodzin_a)$ of $\Rodzin_a$ of eigenvalue $M_k^\perp=M_k$. 

    \noindent \underline{Base case (BC):} We prove the result for $M_1^\perp$. Since $Q_a$ is homogeneous, it is clear that $M_1^\perp=M_1\geq 0$. In particular, $M_1^\perp$ is finite. It is also strictly positive: indeed, by way of contradiction, if $M_1$ was zero then there would exist a minimizing sequence $\{u^j\}_j \subset E(\Omega)$ with 
    \begin{equation}
        \|u^j\|_{L^2(\Omega)}=1, \quad \int_\Omega |\partial_{\bar z} u^j|^2 \rightarrow 0 \quad \text{and} \quad \int_{\partial\Omega} |u^j|^2 \rightarrow 0 \quad \text{as } j \to+\infty,
    \end{equation}
    contradicting the Poincar\'e inequality in \Cref{lemma:PoincareTypeIneq}.

    \underline{Step BC1.} First, we show that $M_1^\perp$ is attained by some $u_1 \in E(\Omega)$. To this end, let $\{u^j\}_j \subset E(\Omega)\setminus \{0\}$ be a minimizing sequence. Without loss of generality, we shall assume that $\|u^j\|_{L^2(\Omega)} = 1$ for all $j\in\N$, so that
    \begin{equation}
        M_1^\perp = \underset{j \to \infty}{\lim} \ \|u^j\|_a^2.
    \end{equation}
    Then, $\{u^j\}_j$ is uniformly bounded in $(E(\Omega), \langle\cdot,\cdot\rangle_a)$, because $M_1^\perp$ is finite. Since $\|\cdot\|_a$ and $\|\cdot\|_{E(\Omega)}$ are comparable on $E(\Omega)$, by \Cref{lemma:AmbientHilbertSpace}, there exists a strongly convergent subsequence in $L^2(\Omega)$, which we call again $\{u^j\}_j$. Also, since $E(\Omega)$ is a Hilbert space, by weak* compactness we can conclude that there exists $u_1 \in E(\Omega)$ such that
    \begin{equation}
        u^j \to u_1 \text{ strongly in } L^2(\Omega) \text{ and weakly in } (E(\Omega), \langle\cdot,\cdot\rangle_a), \text{ as } j \to +\infty.
    \end{equation}
    In particular, $\|u_1\|_{L^2(\Omega)}=1$ and $\|u_1\|_a\neq 0$. Using all the above,
    \begin{equation}
        \begin{split}
            M_1^\perp & = \underset{j \to \infty}{\lim} \ \|u^j\|_a^2 = \underset{j \to \infty}{\lim} \ \left( \underset{\|\varphi\|_a=1}{\sup} \ \left|\langle u^j, \varphi \rangle_a \right| \right)^2 \geq \underset{j \to \infty}{\lim} \ \left| \left\langle u^j, \dfrac{u_1}{\|u_1\|_a} \right\rangle_a \right|^2 \\
            & = \left|\left\langle u_1, \dfrac{u_1}{\|u_1\|_a} \right\rangle_a \right|^2 = \|u_1\|_a^2 \geq M_1^\perp.
        \end{split}
    \end{equation}
    As a consequence, all inequalities are equalities. Therefore, $M_1^\perp$ is attained by $u_1$. 
    
    \underline{Step BC2.} We now address the proof that $u_1$ is an eigenfunction of $\Rodzin_a$ of eigenvalue $M_1^\perp$. To prove this, it is enough to show that $u_1$ satisfies the assumption in \eqref{eq:CharacterizationOfRodzin} with $w= M_1^\perp u_1$. Indeed, if this is true, then by \eqref{eq:CharacterizationOfRodzin} and \Cref{prop:RodzinLaplacian} we will conclude that $u_1 \in \Dom(\Rodzin_a)$ and $\Rodzin_a u_1 = M_1^\perp u_1$. Hence, the proof reduces to showing that $u_1$ satisfies
    \begin{equation} \label{eq:goalIndk1}
        B_a(u_1,v) = \langle M_1^\perp u_1, v \rangle_{L^2(\Omega)} \quad \text{for all } v \in E(\Omega),
    \end{equation}
    which is nothing but the Euler-Lagrange equation for the minimizer $u_1\in E(\Omega)$. Given $v\in E(\Omega)$ and $t\in \R$ ---small enough so that $u+tv \neq 0$ in $E(\Omega)$---, set
    \begin{equation}
        f(t) := \dfrac{ \displaystyle{ 4 \int_\Omega |\partial_{\bar z}u_1|^2 + t^2|\partial_{\bar z}v|^2 + 2t \, \mathrm{Re} \left( \partial_{\bar z}u_1 \, \overline{\partial_{\bar z} v} \right)+ a\int_{\partial \Omega} |u_1|^2+t^2|v|^2 + 2t \, \mathrm{Re} (u_1 \, \overline v) } }{ \displaystyle{ \int_\Omega |u_1|^2+t^2|v|^2 + 2t\, \mathrm{Re} (u_1 \, \overline v) } }.
    \end{equation}
    Since $u_1$ is a minimizer of $M_1$, we deduce that $f(0) \leq f(t)$ for all $|t|$ small enough, and thus $\frac{d}{dt}f(0) = 0$, which yields
    \begin{equation}
        \mathrm{Re} \left( 4 \int_\Omega \partial_{\bar z}u_1 \, \overline{\partial_{\bar z} v} + a\int_{\partial \Omega} u_1 \, \overline v \right) = M_1^\perp \mathrm{Re} \left( \int_\Omega u_1 \, \overline v \right) \quad \text{for all } v\in E(\Omega).
    \end{equation}
    Notice that the same is true replacing $v$ by $-iv$. Hence, $u\in E(\Omega)$ solves
    \begin{equation}
        \displaystyle{4\int_{\Omega} \partial_{\bar z} u_1 \,  \overline{\partial_{\bar z} v} + a\int_{\partial\Omega} u_1 \, \overline v = M_1^\perp \int_\Omega u_1 \, \overline v} \quad \text{for all } v \in E(\Omega),
    \end{equation}
    which shows \eqref{eq:goalIndk1}.

    \noindent \underline{Inductive case (IC):} Assume, for $\ell = 1,\dots,k-1$, that $M_\ell^\perp$ are attained by (pairwise orthogonal) eigenfunctions $u_\ell$ of $\Dom(\Rodzin_a)$ of eigenvalue $M_\ell^\perp$, and that $M_\ell^\perp=M_\ell>0$. We want to show that $M_k^\perp$ is attained by an eigenfunction $u_k$ of $\Rodzin_a$ of eigenvalue $M_k^\perp$, and that $M_k^\perp=M_k>0$. Actually, by the induction hypothesis we already have $M_k\geq M_{k-1}>0$ and $M_k^\perp \geq M_{k-1}^\perp >0$.

    \underline{Step IC1.} We show that $M_k^\perp\leq M_k$ ---and as a consequence, $M_k^\perp$ is finite. Let $F$ be an arbitrary $k-$dimensional subspace of $E(\Omega)$, and take $w\in F\cap \mathrm{span}\{ u_1, \dots, u_{k-1}\}^\perp\setminus\{0\}$. We justify that such $w$ exists: since $F$ is a $k-$dimensional subspace of $E(\Omega)$, it is generated by $k$ linearly independent functions $w_1,\dots,w_k\in E(\Omega)$, and hence every $w\in F$ can be written as a linear combination of the form
    \begin{equation}
        w = \underset{j=1}{\overset{k}{\sum}} \, c_jw_j, \quad \text{for some } c_j\in \C;
    \end{equation}
    in particular, the existence of a non-trivial $w\in F$ orthogonal to $u_1, \dots, u_{k-1}$ is given by the existence of solutions $(c_1,\dots,c_k)\in\C^k\setminus\{0\}$ of the system of $k-1$ linear equations
    \begin{equation}
        0 = \langle w, u_\ell\rangle_{L^2(\Omega)} = \underset{j=1}{\overset{k}{\sum}} \, c_j \langle w_j, u_\ell \rangle_{L^2(\Omega)}, \quad \text{for } \ell = 1,\dots,k-1.
    \end{equation}
    Since $w$ is a concrete function in $F\cap\mathrm{span}\{ u_1, \dots, u_{k-1}\}^\perp\setminus\{0\}$, which is a subset of both $F\setminus\{0\}$ and $E(\Omega)\cap\mathrm{span}\{ u_1, \dots, u_{k-1}\}^\perp\setminus\{0\}$, it is clear that 
    \begin{equation}
        M_k^\perp = \underset{u \in E(\Omega)\cap \mathrm{span}\{u_{1}, \dots, u_{k-1}\}^\perp\setminus\{0\}}{\inf} \ Q_a(u) \leq Q_a(w) \leq \underset{u\in F\setminus\{0\}}{\sup} \, Q_a(u).
    \end{equation}
    Since this holds true for every $k-$dimensional subspace $F$ of $E(\Omega)$, we conclude that
    \begin{equation}
        M_k^\perp \leq \underset{\substack{ F\subset E(\Omega) \\ \mathrm{dim}(F)=k }}{\inf} \, \, \underset{u\in F\setminus\{0\}}{\sup} \, Q_a(u) = M_k.
    \end{equation}
    
    \underline{Step IC2.} We show that $M_k^\perp$ is attained by some $u_k\in E(\Omega)\cap \mathrm{span}\{u_{1}, \dots, u_{k-1}\}^\perp\setminus\{0\}$. The proof follows by the same compactness argument as in Step BC1 ---simply replacing $1$ by $k$, and taking the (normalized) minimizing sequence $\{u^j\}_j$ in $E(\Omega)\cap \mathrm{span}\{u_{1}, \dots, u_{k-1}\}^\perp\setminus\{0\}$. The only delicate step is the last inequality $\|u_k\|_a^2 \geq M_k^\perp$, which we justify now. Since $u^j$ are orthogonal in $L^2(\Omega)$ to $u_\ell$, for $\ell=1,\dots,k-1$, and $u^j$ converge to $u_k\in E(\Omega)$ strongly in $L^2(\Omega)$, also $u_k\in E(\Omega)\cap \mathrm{span}\{u_{1}, \dots, u_{k-1}\}^\perp\setminus\{0\}$. Therefore, $M_k^\perp$ is smaller or equal than $Q_a(u_k)=\|u_k\|_a^2$, by normalization.

    \underline{Step IC3.} We show that the minimizer $u_k\in E(\Omega)\cap \mathrm{span}\{u_{1}, \dots, u_{k-1}\}^\perp\setminus\{0\}$ obtained in Step~IC2 is actually an eigenfunction of $\Rodzin_a$ of eigenvalue $M_k^\perp$. The proof follows by an analogous variational formulation as in Step BC2. The difference now is that, for the Euler-Lagrange equation, one must take $v\in E(\Omega)\cap\mathrm{span}\{u_1, \dots, u_{k-1}\}^\perp\setminus\{0\}$ and consider $u_k + tv$ for $|t|$ small enough. This is needed to ensure that $u_k + tv$ is an admissible competitor, which guarantees that $f(0) \leq f(t)$ for $|t|$ small enough, where now 
    \begin{equation}
        f(t) := \dfrac{ \displaystyle{ 4 \int_\Omega |\partial_{\bar z}u_k|^2 + t^2|\partial_{\bar z}v|^2 + 2t \, \mathrm{Re} \left( \partial_{\bar z}u_k \, \overline{\partial_{\bar z} v} \right)+ a\int_{\partial \Omega} |u_k|^2+t^2|v|^2 + 2t \, \mathrm{Re} (u_k \, \overline v) } }{ \displaystyle{ \int_\Omega |u_k|^2+t^2|v|^2 + 2t\, \mathrm{Re} (u_k \, \overline v) } }.
    \end{equation}
    As a result, the same computation as in Step BC2 (simply replacing $1$ by $k$) shows that 
    \begin{equation}
        \displaystyle{4\int_{\Omega} \partial_{\bar z} u_k \,  \overline{\partial_{\bar z} v} + a\int_{\partial\Omega} u_k \, \overline v = M_k^\perp \int_\Omega u_k \, \overline v} \quad \text{for all } v \in E(\Omega)\cap\mathrm{span}\{u_1, \dots, u_{k-1}\}^\perp.
    \end{equation}
    In order to extend this to every $v\in E(\Omega)$, and thus get the analogous goal to \eqref{eq:goalIndk1}, notice that by inductive hypothesis $u_\ell$ are eigenfunctions of $\Rodzin_a$ of eigenvalue $M_\ell^\perp$ orthogonal to $u_k$, for $\ell=1,\dots,k-1$. Hence, by \Cref{prop:RodzinLaplacian} and \eqref{eq:DeterminationOfRodzin} they satisfy
    \begin{equation}
        \displaystyle{4\int_{\Omega} \partial_{\bar z} u_\ell \,  \overline{\partial_{\bar z} u_k} + a\int_{\partial\Omega} u_\ell \, \overline{u_k} = M_\ell^\perp \int_\Omega u_\ell \, \overline{u_k}} = 0 \quad \text{for } \ell = 1,\dots,k-1.
    \end{equation}
    In particular, there holds
    \begin{equation}
        \displaystyle{4\int_{\Omega} \partial_{\bar z} u_k \,  \overline{\partial_{\bar z} v} + a\int_{\partial\Omega} u_k \, \overline v = 0 = M_k^\perp \int_\Omega u_k \, \overline v} \quad \text{for all } v \in \mathrm{span}\{u_1, \dots, u_{k-1}\}.
    \end{equation}
    In summary, $u_k\in E(\Omega)$ satisfies $B_a(u_k,v) = \langle M_k^\perp u_k, v \rangle_{L^2(\Omega)}$ for all $v \in E(\Omega)$, and hence by \Cref{prop:RodzinLaplacian} and \eqref{eq:CharacterizationOfRodzin} it follows that $u_k\in \Dom(\Rodzin_a)$ with $\Rodzin_a u_k = M_k^\perp u_k$.

    \underline{Step IC4.} We show that $M_k^\perp\geq M_k$, which combined with Step IC1 leads to $M_k^\perp= M_k$. To this end, we first prove that 
    \begin{equation} \label{eq:Goal6}
        M_k^\perp = \underset{u\in \mathrm{span}\{ u_1, \dots, u_k\}\setminus\{0\}}{\sup} \, Q_a(u).
    \end{equation}
    Let $w\in \mathrm{span}\{ u_1, \dots, u_k\}\setminus\{0\}$ and write it as $w=c_1 u_1 +\cdots + c_k u_k$, for some $c_1,\dots,c_k\in \C$ ---not all zero. Since, by Step IC3 and the induction hypothesis, for all $\ell=1,\dots,k$, $u_\ell$ satisfies
    \begin{equation}
        \langle u_\ell, v\rangle_a = B_a(u_\ell, v) = M_\ell^\perp \langle u_\ell, v\rangle_{L^2(\Omega)} \quad \text{for every } v\in E(\Omega),
    \end{equation}
    the functions $u_\ell$ are also pairwise orthogonal in $\left( E(\Omega), \langle \cdot, \cdot \rangle_a \right)$. Using this, we can compute
    \begin{equation} 
        Q_a(w) = Q_a\left( \underset{\ell=1}{\overset{k}{\sum}} \, c_\ell u_\ell \right) = \dfrac{ \underset{\ell=1}{\overset{k}{\sum}} \, M_\ell^\perp |c_\ell|^2\|u_\ell\|_{L^2(\Omega)}^2 }{\underset{\ell=1}{\overset{k}{\sum}} \, |c_\ell|^2 \|u_\ell\|_{L^2(\Omega)}^2} \leq M_k^\perp = Q_a(u_k),
    \end{equation}
    where we have used that $M_\ell^\perp\leq M_k^\perp$ for every $\ell=1,\dots,k$, and the fact that $M_k^\perp$ is attained by $u_k$, by Step IC2. Since this holds for every $w\in \mathrm{span}\{ u_1, \dots, u_k\}\setminus\{0\}$, and there is equality when $w=u_k$, \eqref{eq:Goal6} follows. Since $F:=\mathrm{span}\{ u_1, \dots, u_k\}$ is a particular $k-$dimensional subspace of $E(\Omega)$, in view of \eqref{eq:Goal6} it is then clear that $M_k\leq M_k^\perp$. As a consequence, the claim is proved.

    \underline{The min-max levels $M_k$ are attained by pairwise orthogonal eigenfunctions.} Using that $M_k$ coincides with $M_k^\perp$, \eqref{eq:Goal6}, the min-max characterization of $\mu_k$, and the fact that $\mu_k\geq M_k$ ---because $\mathrm{Dom}(\Rodzin_a)\subset E(\Omega)$---, we have
    \begin{equation}
        M_k = M_k^\perp = \underset{u\in \mathrm{span}\{ u_1, \dots, u_k\}\setminus\{0\}}{\sup} \, Q_a(u) \geq \underset{\substack{ F\subset \mathrm{Dom}(\Rodzin_a) \\ \mathrm{dim}(F)=k }}{\inf} \, \, \underset{u\in F\setminus\{0\}}{\sup} \, Q_a(u) = \mu_k \geq M_k,
    \end{equation}
    where we have used that $F:=\mathrm{span}\{ u_1, \dots, u_k\} $ is a particular subset of $\mathrm{Dom}(\Rodzin_a)$, as proved in the claim. Therefore, all the inequalities are equalities. As a consequence, the attainers $u_k$ of $M_k^\perp$ are also attainers of both $M_k$ and $\mu_k$.

    \underline{Bounds of $\mu_k$.} We conclude the proof showing that $\mu_k\in(0,\Lambda_k)$, where $\Lambda_k$ is the $k-$th eigenvalue of the Dirichlet Laplacian in $\Omega$. Since $\mu_k=M_k$, the lower bound has already been proven in the claim. For the upper bound, using that $H^1_0(\Omega, \R) \subset E(\Omega)$, it follows that
    \begin{equation}
        \mu_k = M_k \leq \underset{\substack{ F\subset H^1_0(\Omega, \R) \\ \mathrm{dim}(F)=k }}{\inf} \, \, \underset{u\in F\setminus\{0\}}{\sup} \, \dfrac{4\int_\Omega |\partial_{\bar z} u|^2 + a \int_{\partial\Omega} |u|^2}{\int_\Omega |u|^2} = \underset{\substack{ F\subset H^1_0(\Omega, \R) \\ \mathrm{dim}(F)=k }}{\inf} \, \, \underset{u\in F\setminus\{0\}}{\sup} \, \dfrac{\int_\Omega |\nabla u|^2}{\int_\Omega |u|^2} =: \Lambda_{k}.
    \end{equation}
    The strict inequality is a consequence of the fact that attainers of $M_k$ are eigenfunctions of $\Rodzin_a$ of eigenvalue $M_k$. Indeed, if the equality $\mu_k = \Lambda_{k}$ held, then the $k-$th (real-valued) eigenfunction $U_k$ of the Dirichlet Laplacian would be an attainer of both $\Lambda_{k}$ and $M_k$; hence, by assumption $U_k$ would simultaneously solve 
    \begin{equation}
        \begin{cases}
            -\Delta U_k = \Lambda_{k} U_k & \text{in } \Omega, \\
            U_k = 0 & \text{on } \partial\Omega,
        \end{cases} \quad \text{and} \quad \begin{cases}
            -\Delta U_k = \Lambda_{k} U_k & \text{in } \Omega, \\
            2\bar\nu \partial_{\bar z} U_k + a U_k = 0 & \text{on } \partial\Omega.
        \end{cases}
    \end{equation}
    Using that $2\bar\nu \partial_{\bar z} U_k = \partial_\nu U_k + i \partial_\tau U_k$ and that $U_k$ is constantly zero on $\partial\Omega$, we would conclude that $U_k$ solves
    \begin{equation}
        \begin{cases}
            -\Delta U_k = \Lambda_{k} U_k & \text{in } \Omega, \\
            U_k = 0 & \text{on } \partial\Omega,\\
            \partial_\nu U_k = 0 & \text{on } \partial\Omega.
        \end{cases}
    \end{equation}
    But then, \cite[Lemma 1]{Gregoire1976} would yield that $U_k = 0$, reaching a contradiction.
\end{proof}

As a consequence of the proof of \Cref{thm:PropertiesMukOmega}, we have the following alternative characterization of the eigenvalues $\mu_{k}(a)$ of $\Rodzin_a$.

\begin{corollary}
    Let $a>0$ and $k\in \N$. Suppose that $u_{\ell}(a)$ are pairwise orthogonal eigenfunctions of $\Rodzin_a$ of eigenvalue $\mu_{\ell}(a)$, for $\ell=1,\dots,k-1$. Then
    \begin{equation}
        \mu_{k}(a) = \underset{u \in E(\Omega)\cap \mathrm{span}\{u_{1}(a), \dots, u_{k-1}(a)\}^\perp\setminus\{0\}}{\inf} \ \dfrac{4\int_\Omega |\partial_{\bar z} u|^2 + a \int_{\partial\Omega} |u|^2}{\int_\Omega |u|^2},
    \end{equation}
    and the infimum is attained by an eigenfunction $u_{k}(a)$ of $\Rodzin_a$ of eigenvalue $\mu_{k}(a)$.
\end{corollary}

\section{The family of operators $\{\Rodzin_a\}_{a>0}$} \label{sec:Family}

In this section, we are interested in the study of the operators $\Rodzin_a$ as the boundary parameter $a$ moves in $(0,+\infty)$.

\subsection{Properties of the eigenvalue curves} \label{sec:EigenvalueCurves}

Since the sesquilinear forms $B_a$, for $a>0$, are densely defined independently of $a$, sectorial, closed, and $B_a(u,u)$ is holomorphic in $a$ for each fixed $u\in E(\Omega)$ ---by \Cref{lemma:PropertiesSesquilinearForm}---, by definition $\{B_a\}_{a>0}$ is a holomorphic family of forms of type (a) in the sense of \cite[pg. 395]{Kato1995}. Since $\{\Rodzin_a\}_{a>0}$ are the associated operators ---by \Cref{prop:RodzinLaplacian}---, it follows by \cite[Theorem VII.4.2]{Kato1995} that $\{\Rodzin_a\}_{a>0}$ is a holomorphic family of operators; more specifically, it is a holomorphic family of type (B) in the sense of \cite[pg. 395]{Kato1995}. 

Actually, by \Cref{prop:BoundedResolvent}, the family of operators $\{\Rodzin_a\}_{a>0}$ is a self-adjoint holomorphic family of type (B) with compact resolvents. A direct consequence of this is the following result, which is a restatement of \cite[Theorem VII.3.9 and Remark VII.4.22]{Kato1995}.

\begin{proposition} \label{prop:AnalyticityKato}
    There is a sequence of scalar valued functions $\mu^n(a)$ and a sequence of vector valued functions $u^n(a) \in L^2(\Omega)$, $n\in\N$, all analytic when seen as functions of $a\in (0,+\infty)$, such that the $\mu^n(a)$ represent all the eigenvalues of $\Rodzin_a$ (repeated according to their multiplicity), and the $u^n(a)$ form a complete orthonormal family of the associated eigenfunctions of $\Rodzin_a$. 
\end{proposition}

Notice that \Cref{prop:AnalyticityKato} asserts that the eigenvalues of $\Rodzin_a$ can be parametrized by analytic functions in $a>0$, but that the eigenvalue $\mu^n(a)$ need not coincide with the $n$-th eigenvalue $\mu_{n}(a)$. Similarly, the theorem asserts that the eigenfunctions of $\Rodzin_a$ can be parametrized by analytic functions in $a>0$, but the eigenfunction $u^n(a)$ need not coincide with the eigenfunctions $u_{n}(a)$ of \Cref{thm:PropertiesMukOmega}. To distinguish the ordered eigenvalues from the eigenvalues of \Cref{prop:AnalyticityKato}, let us give the latter a name: the analytic functions $a\mapsto\mu^n(a)$ of \Cref{prop:AnalyticityKato} are called \textbf{eigenvalue curves}.

\begin{remark} \label{rmk:SenseOfAnalyticity}
    It is clear what analytic means for the scalar valued functions $\mu^n(a)$; for the vector valued functions $u^n(a)$, analyticity is in the sense of \cite[Section VII.1.1]{Kato1995}, namely, $u^n(a)$ is analytic in $a$ if and only if each $a>0$ has a neighborhood in which $\|u^n(a)\|_{L^2(\Omega)}$ is bounded, and the scalar valued function $a \mapsto \langle u^n(a), v \rangle_{L^2(\Omega)}$ is analytic for every $v\in L^2(\Omega)$. In view of \cite[Theorem III.1.37 and Remark III.1.38]{Kato1995}, if $u^n(a)$ is analytic in $a$, then it is strongly continuous in $a$, namely, $\|u^n(a)-u^n(b)\|_{L^2(\Omega)} \to 0$ as $b\to a$.
\end{remark}

A consequence of \Cref{prop:AnalyticityKato} is the piecewise analyticity of the ordered eigenvalues $\mu_{k}(a)$ and associated eigenfunctions $u_{k}(a)$ of $\Rodzin_a$, as the following remark summarizes.

\begin{remark} \label{rmk:AnalyticityKato}
    With the notation of \Cref{prop:AnalyticityKato}, we have that $\mu_{k}(a)$ piecewise coincides with $\mu^n(a)$, for some fixed $n\in\N$. As a consequence, $\mu_{k}(a)$ is piecewise analytic in $a>0$, and in particular it is continuous in $a>0$. The points where there might not be analyticity correspond to crossings of the eigenvalue curves (which occur a finite number of times in every compact set). Crossings of the eigenvalue curves are points where different eigenvalue curves intersect ---see \Cref{fig:EigenRodzinBall}---, leading to a discontinuity of the multiplicity of a certain $\mu_{k}(a)$. In addition, in the intervals $I\subseteq (0,+\infty)$ where $\mu_{k}(a)$ is analytic, by \Cref{prop:AnalyticityKato} we can choose an associated eigenfunction $u^k(a)$ analytic in $a$. Thus, we can choose an attainer $u_{k}(a)$ of $\mu_{k}(a)$ as in \Cref{thm:PropertiesMukOmega} to be piecewise analytic in $a\in(0,+\infty)$ ---which need not be continuous along the crossings of the eigenvalue curves. In the intervals $I\subseteq (0,+\infty)$ where $u_{k}(a)$ is analytic in $a$, in particular it is strongly continuous, and hence $\|u_{k}(a)-u_{k}(b)\|_{L^2(\Omega)} \to 0$ as $b\to a$. 
\end{remark}

The following lemma shows continuity with respect to $a$ in a finer topology than $L^2(\Omega)$ of the eigenfunction $u_{k}(a)$ associated to the eigenvalue $\mu_{k}(a)$ in the intervals $I\subseteq (0,+\infty)$ where it is analytic.

\begin{lemma} \label{lemma:continuityEigenRodzin}
    For every $a>0$ and $k\in \N$, let $I\subseteq (0,+\infty)$ be an open interval where $\mu_{k}(a)$ is analytic, and let $u_{k}(a)$ be an associated eigenfunction analytic in $a$, as in \Cref{rmk:AnalyticityKato}. Then, for every $a\in I$,
    \begin{equation}
        \|u_{k}(b) - u_{k}(a)\|_{E(\Omega)} \to 0 \quad \text{as } b\to a.
    \end{equation}
\end{lemma}

\begin{proof}
    For the sake of notation, let us denote $\mu(a) \equiv \mu_{k}(a)$ and $u(a) \equiv u_{k}(a)$. Since by \Cref{lemma:PoincareTypeIneq} the norms $\|\cdot\|_{E(\Omega)}$ and $\|\cdot\|_a$ are equivalent on $E(\Omega)$, where 
    \begin{equation} 
        \langle u, v \rangle_a := B_a(u,v) = \displaystyle{ 4 \int_\Omega \partial_{\bar z} u \, \overline{\partial_{\bar z}v} + a \int_{\partial\Omega} u \, \overline v } \quad \text{and} \quad \|u\|_a^2 := B_a(u,u),
    \end{equation}
    in order to prove the lemma we will prove that $\|u(a) - u(b)\|_a \to 0$ as $b\to a$. Notice that by \Cref{rmk:AnalyticityKato}, both $\mu(a)$ and $\langle u(a), v \rangle_{L^2(\Omega)}$ are continuous in $a\in I$, for every $v\in L^2(\Omega)$. Hence, since by \Cref{prop:RodzinLaplacian} and \eqref{eq:DeterminationOfRodzin} there holds
    \begin{equation} \label{eq:auxContn}
        \mu(a) \langle u(a), v \rangle_{L^2(\Omega)} = B_a(u(a), v) \quad \text{for every } v\in E(\Omega),
    \end{equation}
    we conclude that $a\mapsto B_a(u(a), v)$ is continuous in $a\in I$ for every $v\in E(\Omega)$.
    
    To prove that $\|u(a) - u(b)\|_a \to 0$ as $b\to a$, using that $u(a)$ and $u(b)$ are eigenfunctions, we rewrite this norm as
    \begin{equation} \label{eq:AuxiliarContinuity1}
        \begin{split}
            \|u(a) - u(b)\|_a^2 & = \|u(a)\|_a^2 + \|u(b)\|_a^2 - 2\text{Re} \, B_a \left( u(b), u(a) \right) \\
            & = \mu(a) \|u(a)\|_{L^2(\Omega)}^2 + \mu(b) \|u(b)\|_{L^2(\Omega)}^2 \\
            & \quad + (a-b) \int_{\partial\Omega} |u(b)|^2 - 2\text{Re} \, B_a \left( u(b), u(a) \right) \\
            & = \mu(a) \|u(a)\|_{L^2(\Omega)}^2 + \mu(b) \|u(b)\|_{L^2(\Omega)}^2 - 2\text{Re} \, B_b \left( u(b), u(a) \right) \\
            & \quad + (a-b) \left( \int_{\partial\Omega} |u(b)|^2 - 2\text{Re} \int_{\partial\Omega} u(b) \, \overline{u(a)} \right),
        \end{split}
    \end{equation}
    and study the convergence of the last expression as $b\to a$.  
    
    On the one hand, by continuity of $b\mapsto \mu(b)$, $b\mapsto u(b)$ in $L^2(\Omega)$ ---recall \Cref{rmk:SenseOfAnalyticity}---, and of $b\mapsto B_b(u(b), v)$ for fixed $v\in E(\Omega)$, using \eqref{eq:auxContn} we have
    \begin{equation} \label{eq:AuxiliarContinuity2}
        \begin{split}
            & \mu(a) \|u(a)\|_{L^2(\Omega)}^2 + \mu(b) \|u(b)\|_{L^2(\Omega)}^2 - 2\text{Re} \left( B_b \left( u(b), u(a) \right) \right) \\
            & \quad \longrightarrow 2 \mu(a) \|u(a)\|_{L^2(\Omega)}^2 - 2\text{Re} \, B_a \left( u(a), u(a) \right) = 0 \quad \text{as } b \to a.
        \end{split}
    \end{equation}
    On the other hand, by continuity of $b\mapsto \mu(b)/b$ and $b\mapsto u(b)$ in $L^2(\Omega)$, using \eqref{eq:auxContn} we have
    \begin{equation} \label{eq:AuxiliarContinuity3}
        \begin{split}
            & \left| \int_{\partial\Omega} |u(b)|^2 - 2\text{Re} \int_{\partial\Omega} u(b) \, \overline{u(a)} \right| \leq \|u(b)\|_{L^2(\partial\Omega)}^2 + \int_{\partial\Omega} \left( |u(b)|^2 + |u(a)|^2 \right) \\
            & \quad \leq \|u(a)\|_{L^2(\partial\Omega)}^2 + 2 \|u(b)\|_{L^2(\partial\Omega)}^2 \leq  \frac{1}{a}B_a(u(a),u(a)) + \frac{2}{b} B_b(u(b),u(b))  \\
            & \quad = \dfrac{\mu(a)}{a} \|u(a)\|_{L^2(\Omega)}^2 + 2 \dfrac{\mu(b)}{b} \|u(b)\|_{L^2(\Omega)}^2  \to 3 \dfrac{\mu(a)}{a} \|u(a)\|_{L^2(\Omega)}^2 \quad \text{as } b \to a,
        \end{split}
    \end{equation}
    which is bounded. Taking the limit $b\to a$ in \eqref{eq:AuxiliarContinuity1}, by \eqref{eq:AuxiliarContinuity2} and \eqref{eq:AuxiliarContinuity3} we conclude that $\|u(a) - u(b)\|_a \to 0$ as $b\to a$, as desired.
\end{proof}

\subsection{Dependence of the eigenvalues of $\Rodzin_a$ on $a$} \label{sec:GlobalSpectralProperties}

In this section, we prove the properties that the eigenvalues described in \Cref{thm:PropertiesMukOmega,thm:PropertiesMukOmegaGlobal} satisfy, seen as functions of the boundary parameter $a$.

\begin{proof}[Proof of \Cref{thm:PropertiesMukOmegaGlobal}]
    \Cref{item:Analytic} is a direct combination of \Cref{thm:PropertiesMukOmega}, \Cref{rmk:AnalyticityKato} and Lemma \ref{lemma:continuityEigenRodzin}. Before showing \cref{item:ContIncBij}, let us address the proof of \cref{item:Derivative}. Since $u_{k}(a)$ is an attainer, by \Cref{thm:PropertiesMukOmega} it is an eigenfunction of $\Rodzin_a$ of eigenvalue $\mu_{k}(a)$. Then, by \Cref{prop:RodzinLaplacian} and \eqref{eq:DeterminationOfRodzin} we have
    \begin{equation}
        \mu_{k}(a) \int_\Omega u_{k}(a) \, \overline{u_{k}(b)} = 4\int_\Omega \partial_{\bar z} u_{k}(a) \, \overline{\partial_{\bar z} u_{k}(b)} + a \int_{\partial\Omega} u_{k}(a) \, \overline{u_{k}(b)}.
    \end{equation}
    Similarly, and since $\mu_k(b)\in \R$, for every $b\in (0,+\infty)\setminus X_k$ in the same connected component as $a$, we have
    \begin{equation}
        \mu_{k}(b) \int_\Omega u_{k}(a) \, \overline{u_{k}(b)} = 4\int_\Omega \partial_{\bar z} u_{k}(a) \, \overline{\partial_{\bar z} u_{k}(b)} + b \int_{\partial\Omega} u_{k}(a) \, \overline{u_{k}(b)}.
    \end{equation}
    Since $\langle u_{k}(a), u_{k}(b) \rangle_{L^2(\Omega)} \to \|u_{k}(a)\|_{L^2(\Omega)}^2 \neq 0$ as $b\to a$, for close enough $b$ to $a$ we conclude that $\langle u_{k}(a), u_{k}(b) \rangle_{L^2(\Omega)} \neq 0$. Hence, subtracting the previous equations and dividing by $\langle u_{k}(a), u_{k}(b) \rangle_{L^2(\Omega)}$, we obtain that
    \begin{equation}
        \mu_{k}(a) - \mu_{k}(b) = (a-b) \dfrac{\int_{\partial\Omega} u_{k}(a) \, \overline{u_{k}(b)}}{\int_\Omega u_{k}(a) \, \overline{u_{k}(b)}}.
    \end{equation}
    Dividing both sides by $a-b$ and taking the limit $b\to a$, by the continuity of $b\mapsto u_{k}(b)$ in the $E(\Omega)-$norm, given by \Cref{lemma:continuityEigenRodzin}, we conclude that 
    \begin{equation}
        \dfrac{d}{da} \mu_{k}(a) = \underset{b\to a}{\lim} \ \frac{\mu_{k}(a) - \mu_{k}(b)}{a-b} = \frac{\int_{\partial\Omega} |u_{k}(a)|^2}{\int_\Omega |u_{k}(a)|^2}.
    \end{equation}
    We conclude the proof of \cref{item:Derivative} justifying that the right-hand side is strictly positive for every $a\in(0,+\infty)$. Indeed, if it was zero for some $a_\star>0$, then $u_{k}(a_\star)$ would be an eigenfunction of $\Rodzin_{a_\star}$ of eigenvalue $\mu_{k}(a_\star)$ ---by \Cref{thm:PropertiesMukOmega}--- with zero trace in $L^2(\partial\Omega)$. Using that $2\bar\nu \partial_{\bar z} u_{k}(a_\star) = \partial_\nu u_{k}(a_\star) + i \partial_\tau u_{k}(a_\star)$ and that $u_{k}(a_\star)$ would be constantly zero on $\partial\Omega$, we would conclude (similarly as in the proof of \Cref{thm:PropertiesMukOmega}) that $u_{k}(a_\star)$ solves
    \begin{equation}
        \begin{cases}
            -\Delta u_{k}(a_\star) = \mu_{k}(a_\star) u_{k}(a_\star) & \text{in } \Omega, \\
            u_{k}(a_\star) = 0 & \text{on } \partial\Omega,\\
            \partial_\nu u_{k}(a_\star) = 0 & \text{on } \partial\Omega.
        \end{cases}
    \end{equation}
    But then, \cite[Lemma 1]{Gregoire1976} would yield that $u_{k}(a_\star) = 0$, reaching a contradiction.

    Once \cref{item:Derivative} is proved we can establish \cref{item:ContIncBij}. We first show that $\mu_{k}(a) \to 0$ as $a\to0^+$. Indeed, since by \cref{item:Analytic} we know that $\mu_{k}(a)$ is continuous ---and positive--- in $(0,+\infty)$, and $A^2(\Omega)\cap L^2(\partial\Omega) \subset E(\Omega)$ ---recall the Bergman space $A^2(\Omega)$ studied in \Cref{sec:BergmanSpace}---, we have
    \begin{equation}
        \begin{split}
            0 & \leq \underset{a\to 0^+}{\lim} \, \mu_{k}(a) \leq \underset{a\to 0^+}{\lim} \, \underset{\substack{ F\subset A^2(\Omega)\cap L^2(\partial\Omega) \\ \mathrm{dim}(F)=k }}{\inf} \, \, \underset{u\in F\setminus\{0\}}{\sup} \, \dfrac{4\int_\Omega |\partial_{\bar z} u|^2 + a \int_{\partial\Omega} |u|^2}{\int_\Omega |u|^2} \\
            & = \underset{a\to 0^+}{\lim} \, \underset{\substack{ F\subset A^2(\Omega)\cap L^2(\partial\Omega) \\ \mathrm{dim}(F)=k }}{\inf} \, \, \underset{u\in F\setminus\{0\}}{\sup} \, \dfrac{a \int_{\partial\Omega} |u|^2}{\int_\Omega |u|^2} = \underset{a\to 0^+}{\lim} \, a \left( \underset{\substack{ F\subset A^2(\Omega)\cap L^2(\partial\Omega) \\ \mathrm{dim}(F)=k }}{\inf} \, \, \underset{u\in F\setminus\{0\}}{\sup} \, \dfrac{ \int_{\partial\Omega} |u|^2}{\int_\Omega |u|^2} \right) = 0,
        \end{split}
    \end{equation}
    where in the last equality we have used the finiteness of the min-max level, given by \Cref{prop:MinMaxSplope}. The convergence $\mu_{k}(a) \to \Lambda_{k}$ as $a\to+\infty$ is shown in \Cref{thm:NRCDirichletLaplacian}, a proof of which shall be given in \Cref{sec:NRCDirichlet}. 

    Finally, from \cref{item:Derivative} we have that $a\mapsto \mu_{k}(a)$ is strictly increasing in $(0,+\infty)\setminus X_k$. Since $X_k$ is finite in every compact subset of $(0,+\infty)$, we conclude that $a\mapsto \mu_{k}(a)$ is a strictly increasing and continuous function from $(0,+\infty)$ to $(0,\Lambda_{k})$, with $\mu_{k}(a) \to 0$ as $a\to0^+$ and $\mu_{k}(a) \to \Lambda_{k}$ as $a\to+\infty$. As a consequence, \cref{item:ContIncBij} follows.

    It only remains to prove that for $k=1$ the function $a\mapsto \mu_{1}(a)$ is strictly concave in $(0,+\infty)$, namely, that for every $a\in(0,+\infty)$ the following inequality holds:
    \begin{equation}
        \mu_1(ta+(1-t)b) > t\mu_1(a) + (1-t)\mu_1(b) \quad \text{for every } b\in(0,+\infty)\setminus\{a\} \text{ and } t\in(0,1).
    \end{equation}
    Since the infimum of a sum is larger than the sum of infima, in view of \Cref{thm:PropertiesMukOmega} we have
    \begin{equation}
        \begin{split}
            \mu_1(ta+(1-t)b) & = \underset{u\in E(\Omega)\setminus\{0\}}{\inf} \, \frac{4\int_\Omega |\partial_{\bar z}u|^2 + (ta+(1-t)b) \int_{\partial\Omega} |u|^2}{\int_\Omega |u|^2} \\
            & = \underset{u\in E(\Omega)\setminus\{0\}}{\inf} \, \left( t\frac{4\int_\Omega |\partial_{\bar z}u|^2 + a \int_{\partial\Omega} |u|^2}{\int_\Omega |u|^2} + (1-t) \frac{4\int_\Omega |\partial_{\bar z}u|^2 + b \int_{\partial\Omega} |u|^2}{\int_\Omega |u|^2} \right) \\
            & \geq t\underset{u\in E(\Omega)\setminus\{0\}}{\inf} \, \frac{4\int_\Omega |\partial_{\bar z}u|^2 + a \int_{\partial\Omega} |u|^2}{\int_\Omega |u|^2} + (1-t) \underset{u\in E(\Omega)\setminus\{0\}}{\inf} \, \frac{4\int_\Omega |\partial_{\bar z}u|^2 + b \int_{\partial\Omega} |u|^2}{\int_\Omega |u|^2} \\
            & = t\mu_1(a) + (1-t)\mu_1(b).
        \end{split}
    \end{equation}
    We conclude the proof of \cref{item:Concavity} justifying that the inequality is strict. Since, by \Cref{thm:PropertiesMukOmega}, the infima in $\mu_1(ta+(1-t)b)$, $\mu_1(a)$ and $\mu_1(b)$ are attained, respectively, by $u_1(ta+(1-t)b)$, $u_1(a)$ and $u_1(b)$, if there was equality for some $a\in(0,+\infty)$, $b\in(0,+\infty)\setminus\{a\}$ and $t\in(0,1)$, then $u_1(ta+(1-t)b)$ would be a minimizer for both $\mu_1(a)$ and $\mu_1(b)$. Hence, by \Cref{thm:PropertiesMukOmega}, $u_1(ta+(1-t)b)$ would solve
    \begin{equation}
        \mu_1(a) u_1(ta+(1-t)b) = - \Delta u_1(ta+(1-t)b) = \mu_1(b) u_1(ta+(1-t)b) \quad \text{in } \Omega.
    \end{equation}
    Since $a\mapsto\mu_1(a)$ is strictly increasing and $b\neq a$, there holds $\mu_1(a)\neq \mu_1(b)$. Hence, the previous equation would lead to $u_1(ta+(1-t)b)=0$ in $\Omega$, which is a contradiction.
\end{proof}

As we said, to complete the proof of \Cref{thm:PropertiesMukOmegaGlobal} it only remains to show that $\mu_{k}(a) \to \Lambda_{k}$ as $a\to+\infty$. This could be proven by induction with a standard compactness argument. More directly, it is a consequence of \Cref{thm:NRCDirichletLaplacian}, which we prove in the following section.

\subsection{Convergence of $\Rodzin_a$} \label{sec:NRCDirichlet}

In this section, we study the convergence of $\Rodzin_a$ as $a$ moves in $(0,+\infty)$. We first show that $\Rodzin_a$ converges to $-\Delta_D$ ---recall the definition of $-\Delta_D$ in \eqref{eq:OpDirichletLaplacian}--- in the norm resolvent sense as $a\to+\infty$.

\begin{proof}[Proof of \Cref{thm:NRCDirichletLaplacian}]
    Since $-\Delta_D$ and $\Rodzin_a$ are self-adjoint, by \cite[Theorem VIII.19 (a)]{ReedSimon1980} it is enough to prove that the difference of resolvents at $\lambda=i$, namely
    \begin{equation}
        W_{a,D} := (\Rodzin_a-i)^{-1} - (-\Delta_D - i)^{-1},
    \end{equation}
    converges to zero in the operator norm as $a\to+\infty$. Given $f,g\in L^2(\Omega)$, set
    \begin{equation}
        \begin{split}
            \varphi_a & := (\Rodzin_a-i)^{-1} f \in \Dom(\Rodzin_a), \\
            \psi & := (-\Delta_D + i)^{-1} g \in \Dom(-\Delta_D).
        \end{split}
    \end{equation}
    Integration by parts and self-adjointness of $-\Delta_D$ lead to
    \begin{equation} \label{eq:auxNRCPairingInfty}
        \begin{split}
            \langle W_{a,D}f, g\rangle_{L^2(\Omega)} & = \langle (\Rodzin_a-i)^{-1}f - (-\Delta_D - i)^{-1}f, g\rangle_{L^2(\Omega)} \\
            & = \langle (\Rodzin_a-i)^{-1}f, g\rangle_{L^2(\Omega)} - \langle f, (-\Delta_D + i)^{-1}g \rangle_{L^2(\Omega)} \\
            & = \langle \varphi_a, (-\Delta_D + i)\psi \rangle_{L^2(\Omega)} - \langle (\Rodzin_a-i)\varphi_a, \psi \rangle_{L^2(\Omega)} \\
            & = \langle \varphi_a, -\Delta \psi \rangle_{L^2(\Omega)} - \langle -\Delta \varphi_a, \psi \rangle_{L^2(\Omega)} \\
            & = - \langle \varphi_a, 2 \overline \nu \partial_{\bar z} \psi \rangle_{L^2(\partial\Omega)} + \langle 2 \bar \nu \partial_{\bar z} \varphi_a, \psi \rangle_{L^2(\partial\Omega)} \\
            & = \dfrac{4}{a} \langle \partial_{\bar z} \varphi_a, \partial_{\bar z} \psi \rangle_{L^2(\partial\Omega)},
        \end{split}
    \end{equation}
    where in the last equality we have used the boundary conditions $2 \bar \nu \partial_{\bar z} \varphi_a + a \varphi_a=0$ and $\psi=0$ that $\varphi_a\in \Dom(\Rodzin_a)$ and $\psi\in \Dom(-\Delta_D)$ satisfy. As we shall see afterward, and in view of \eqref{eq:auxNRCPairingInfty}, in order to prove the norm resolvent convergence it is enough to bound the last pairing uniformly in $a$. This is what we do next. 
    
    Since both $\partial_{\bar z} \varphi_a$ and $\partial_{\bar z} \psi$ belong to $H^1(\Omega)$, by the property of the pairing \eqref{eq:Brezis}, the trace theorem from $H^1(\Omega)$ to $H^{1/2}(\partial\Omega)$, and \Cref{lemma:propertiesWirtinger} for $\partial_{\bar z} \varphi_a \in \mathrm{Dom}(\partial_{\mathrm{ah}})$, we deduce that
    \begin{equation} \label{eq:boundWaD}
        \left| \langle \partial_{\bar z} \varphi_a, \partial_{\bar z} \psi \rangle_{L^2(\partial\Omega)} \right| \leq \|\partial_{\bar z} \varphi_a\|_{H^{-1/2}(\partial\Omega)} \|\partial_{\bar z} \psi\|_{H^{1/2}(\partial\Omega)} \leq C \left( \|\partial_{\bar z} \varphi_a\|_{L^2(\Omega)} + \|\Delta\varphi_a \|_{L^2(\Omega)} \right) \|\psi\|_{H^2(\Omega)}
    \end{equation}
    for some constant $C>0$ depending only on $\Omega$. On the one hand, by the resolvent estimates \eqref{eq:L2boundLaplacian} and \eqref{eq:L2boundDz} for $\lambda=i$, we have that
    \begin{equation}
         \|\partial_{\bar z} \varphi_a\|_{L^2(\Omega)} + \|\Delta\varphi_a \|_{L^2(\Omega)} \leq (2+\sqrt{2}) \|f\|_{L^2(\Omega)}.
    \end{equation}
    On the other hand, since $\psi$ is the unique solution to the Laplacian problem
    \begin{equation}
        \begin{cases}
            -\Delta \psi + i\psi = g & \text{in } \Omega, \\
            \psi = 0 & \text{on } \partial\Omega,
        \end{cases}
    \end{equation}
    with $g \in L^2(\Omega)$, by standard regularity theory ---see, for example, \cite[Theorem 4 in Section 6.3.2 and Remark in pg. 335]{Evans2010}--- we have
    \begin{equation}
        \|\psi\|_{H^2(\Omega)} \leq C\|g\|_{L^2(\Omega)}.
    \end{equation}
    Combining these two estimates with \eqref{eq:boundWaD} and \eqref{eq:auxNRCPairingInfty}, we obtain
    \begin{equation}
        \left|\langle W_{a,D}f, g\rangle_{L^2(\Omega)}\right| \leq K/a \|f\|_{L^2(\Omega)} \|g\|_{L^2(\Omega)},
    \end{equation}
    for some constant $K>0$ depending only on $\Omega$. Dividing both sides of this last inequality by $\|f\|_{L^2(\Omega)} \|g\|_{L^2(\Omega)}$ and taking the supremum among all functions $f,g\in L^2(\Omega)\setminus\{0\}$, we conclude that
    \begin{equation}
        \|W_{a,D}\|_{L^2(\Omega)\to L^2(\Omega)} \leq K/a.
    \end{equation}
    Norm resolvent convergence follows taking $a\to +\infty$. Then, the convergence of the spectra follows by \cite[Theorems VIII.23 (a) and VIII.24 (a)]{ReedSimon1980}. As a consequence of \cite[Theorem 2.3.1]{Henrot2006} applied to the resolvents of $\Rodzin_a$ and $-\Delta_D$, the $k-$th eigenvalue of $\Rodzin_a$ converges to the $k-$th eigenvalue of $-\Delta_D$ as $a$ goes to $+\infty$. This concludes the proof of the theorem.
\end{proof}

The proof of \Cref{thm:NRCDirichletLaplacian} adapts without difficulties to show that $\Rodzin_a$ converges to $\Rodzin_{a_0}$ in the norm resolvent sense as $a\to a_0$, for every $a_0>0$.

\begin{proposition} \label{prop:NRCFiniteValues}
    For every $a_0>0$, the $\overline\partial$-Robin Laplacian $\Rodzin_a$ converges to $\Rodzin_{a_0}$ in the norm resolvent sense as $a\to a_0$. 
\end{proposition}

\begin{proof}
    As before, we will show that
    \begin{equation}
        W_{a,a_0} := (\Rodzin_a-i)^{-1} - (\Rodzin_{a_0} - i)^{-1},
    \end{equation}
    converges to zero in norm as $a\to a_0$. Given $f,g\in L^2(\Omega)$, set
    \begin{equation}
        \varphi_a := (\Rodzin_a-i)^{-1} f \in \Dom(\Rodzin_a) \quad \text{and} \quad \varphi_{a_0} := (\Rodzin_{a_0} + i)^{-1} g \in \Dom(\Rodzin_{a_0}).
    \end{equation}
    Similarly as in \eqref{eq:auxNRCPairingInfty}, integration by parts and self-adjointness of $\Rodzin_{a_0}$ straightforwardly lead to
    \begin{equation}
        \langle W_{a,a_0}f, g\rangle_{L^2(\Omega)} = - \langle \varphi_a, 2 \overline \nu \partial_{\bar z} \varphi_{a_0} \rangle_{L^2(\partial\Omega)} + \langle 2 \bar \nu \partial_{\bar z} \varphi_a, \varphi_{a_0} \rangle_{L^2(\partial\Omega)}.
    \end{equation}
    Using the boundary conditions $2\overline \nu \partial_{\bar z} \varphi_\alpha + \alpha \varphi_\alpha=0$ for $\alpha\in\{a,a_0\}$, we get that
    \begin{equation} \label{eq:auxNRCPairing}
        \langle W_{a,a_0}f, g\rangle_{L^2(\Omega)} =(a_0-a) \langle \varphi_a, \varphi_{a_0} \rangle_{L^2(\partial\Omega)}.
    \end{equation}
    As we did for \Cref{thm:NRCDirichletLaplacian}, it is enough to bound the pairing in the right-hand side of \eqref{eq:auxNRCPairing} uniformly in $a$. This is what we do next.
    
    First, since $\varphi_a,\varphi_{a_0}\in H^1(\Omega)$, by the property of the pairing \eqref{eq:Brezis}, the trace theorem from $H^1(\Omega)$ to $H^{1/2}(\partial\Omega)$, and \Cref{lemma:propertiesWirtinger}, we deduce that
    \begin{equation} \label{eq:boundWaa0}
        \left| \langle \varphi_a, \varphi_{a_0} \rangle_{L^2(\partial\Omega)} \right | \leq \| \varphi_a \|_{H^{-1/2}(\partial\Omega)} \| \varphi_{a_0} \|_{H^{1/2}(\partial\Omega)} \leq C \left( \| \varphi_a\|_{L^2(\Omega)} + \|\partial_{\bar z} \varphi_a \|_{L^2(\Omega)} \right) \|\varphi_{a_0}\|_{H^1(\Omega)}
    \end{equation}
    for some constant $C>0$ depending only on $\Omega$. On the one hand, by the resolvent estimates \eqref{eq:BoundedResolventL2L2} and \eqref{eq:L2boundDz} for $\lambda=i$, we have that
    \begin{equation}
         \|\varphi_a\|_{L^2(\Omega)} + \| \partial_{\bar z} \varphi_a \|_{L^2(\Omega)} \leq (1+\sqrt{2}) \|f\|_{L^2(\Omega)}.
    \end{equation}
    On the other hand, by \Cref{prop:BoundedResolvent} with $\lambda=-i$ ---and $g$ instead of $f$---, we have that
    \begin{equation}
        \|\varphi_{a_0}\|_{H^1(\Omega)} \leq C(1+1/a_0) \|g\|_{L^2(\Omega)},
    \end{equation}
    for some constant $C>0$ depending only on $\Omega$. Combining these two estimates with \eqref{eq:boundWaa0} we obtained the desired uniform bound and conclude the proof.
\end{proof}

We conclude this section with the convergence of $\Rodzin_a$ to $\Rodzin_0$ as $a\to 0^+$. To this end, we first study the $\overline\partial$-Neumann Laplacian $\Rodzin_0$ defined in \eqref{eq:Rodzin0}.

\begin{proof}[Proof of \Cref{prop:SpectrumRodzin0}]
    We start showing that the $\overline\partial$-Neumann Laplacian is self-adjoint in $L^2(\Omega)$. To this end, by \cite[Proposition 8.5 in Appendix A]{Taylor2011} it is enough to show that it is symmetric with $\mathrm{Range}(\Rodzin_0\pm i) = L^2(\Omega)$.

    We first verify that $\Rodzin_0$ is symmetric. Indeed, for any $u,v\in \Dom(\Rodzin_0)$, since $\partial_{\bar z} u, \partial_{\bar z} v \in H^1_0(\Omega)$, by the divergence theorem
    \begin{equation}
        \langle \Rodzin_0 u, v \rangle_{L^2(\Omega)} = \int_\Omega -\Delta u \, \overline v = 4\int_\Omega \partial_{\bar z} u \, \overline{\partial_{\bar z} v} = \int_\Omega u \, \overline{-\Delta v} =\langle u, \Rodzin_0 v \rangle_{L^2(\Omega)}.
    \end{equation}

    We now verify that $\mathrm{Range}(\Rodzin_0\pm i) = L^2(\Omega)$. We want to show that for every $f\in L^2(\Omega)$ there exists $u_f^\pm \in \Dom(\Rodzin_0)$ such that $(\Rodzin_0\pm i) u_f^\pm = f$. Namely, we want to show that the problem
    \begin{equation} \label{eq:BvPRodzin0SA}
        \begin{cases}
            -\Delta u \pm i u = f & \text{in } \Omega, \\
            \partial_{\bar z} u = 0 & \text{on } \partial\Omega,
        \end{cases}
    \end{equation}
    has a solution $u_f^\pm\in \mathrm{Dom}(\Rodzin_0)$ for every $f\in L^2(\Omega)$. We show this via Lax-Milgram Theorem. To this end, notice that the weak formulation of \eqref{eq:BvPRodzin0SA} is
    \begin{equation} \label{eq:auxLaxMilgram}
        4\int_\Omega \overline{\partial_{\bar z} u} \, \partial_{\bar z} v \mp i \int_\Omega \overline u \, v = \int_\Omega \overline f \, v \quad \text{for all } v\in \Dom(\partial_{\mathrm{h}}),
    \end{equation}
    and recall from \Cref{lemma:propertiesWirtinger} $(i)$ that
    \begin{equation}
        \Dom(\partial_\mathrm{h}) = \left\{ u\in L^2(\Omega): \, \partial_{\bar z} u \in L^2(\Omega) \right\}
    \end{equation}
    is a Hilbert space with norm $\|u\|_\mathrm{h}^2 := \|u\|_{L^2(\Omega)}^2+\|\partial_{\bar z} u \|_{L^2(\Omega)}^2$. It is straightforward to verify that the right-hand side of \eqref{eq:auxLaxMilgram} defines a continuous linear functional in $\Dom(\partial_\mathrm{h})$, and that the left-hand side of \eqref{eq:auxLaxMilgram} defines a continuous sesquilinear form that is coercive with respect to $\|\cdot\|_\mathrm{h}^2$. As a consequence, by the Lax-Milgram Theorem \cite[Theorem 6.3.6]{Lax2002} there exists a unique $u_f^\pm \in \Dom(\partial_{\mathrm{h}})$ solving \eqref{eq:BvPRodzin0SA} in the weak sense. In particular, $-\Delta u_f^\pm = f \mp i u_f^\pm$ first in the distributional sense, and then in $L^2(\Omega)$. Then, $\partial_{\bar z} u_f^\pm \in \Dom(\partial_{\mathrm{ah}})$ with $t_{\partial\Omega} \partial_{\bar z} u_f^\pm = 0 \in H^{1/2}(\partial\Omega)$. By \Cref{lemma:propertiesWirtinger} $(iv)$, $\partial_{\bar z} u_f^\pm \in H^1_0(\Omega)$, and hence $u_f^\pm \in \Dom(\Rodzin_0)$ with $(\Rodzin_0\pm i)u_f^\pm = f$, as desired.  

    Once self-adjointness of $\Rodzin_0$ is ensured, we characterize its spectrum. First, we show that $\ker(\Rodzin_0) = A^2(\Omega)$. Clearly $0\in \sigma_{\mathrm{ess}}(\Rodzin_0)$ is an eigenvalue of infinite multiplicity, because $\Rodzin_0 u = 0$ for every $u\in A^2(\Omega)$. In particular, $A^2(\Omega) \subseteq \ker(\Rodzin_0)$. Conversely, if $u\in \ker(\Rodzin_0)$, then $\partial_z (\partial_{\bar z} u) = \frac{1}{4}\Delta u = 0$ in $\Omega$ and $\partial_{\bar z} u = 0$ on $\partial\Omega$, hence $\partial_{\bar z} u$ is an antiholomorphic function with zero trace. By unique continuation, this leads to $\partial_{\bar z} u = 0$, that is, $u \in A^2(\Omega)$.

    Next, we show that $\sigma_{\mathrm{p}}(\Rodzin_0) \setminus \{0\} = \sigma(-\Delta_D)$, and that the multiplicity of every $\lambda \in \sigma_{\mathrm{p}}(\Rodzin_0)\setminus\{0\}$ as an eigenvalue of $\Rodzin_0$, denoted $m_0$, coincides with its multiplicity as an eigenvalue of $-\Delta_D$, denoted $m_D$. Let $\lambda \in \sigma_{\mathrm{p}}(\Rodzin_0)\setminus\{0\}$. Then, there exists $u\in \Dom(\Rodzin_0)\setminus A^2(\Omega)$ ---in particular, $\partial_{\bar z} u \in H^1_0(\Omega)$--- such that
    \begin{equation} \label{eq:KernelRodzin0}
        \begin{cases}
            -\Delta u = \lambda u & \text{in } \Omega, \\
            \partial_{\bar z} u = 0 & \text{on } \partial\Omega.
        \end{cases}
    \end{equation}
    Differentiating the PDE by $\partial_{\bar z}$, we deduce that
    \begin{equation} \label{eq:KernelRodzin0Laplacian}
        \begin{cases}
            -\Delta (\partial_{\bar z} u) = \lambda (\partial_{\bar z} u) & \text{in } \Omega, \\
            \partial_{\bar z} u = 0 & \text{on } \partial\Omega.
        \end{cases}
    \end{equation}
    In conclusion, $v:= \partial_{\bar z} u$ is an eigenfunction of the Dirichlet Laplacian $-\Delta_D$ with eigenvalue $\lambda$, hence $\lambda \in \sigma(-\Delta_D)$. To see that $m_D\geq m_0$, let $\{u_j\}_{j=1}^{m_0}$ be linearly independent eigenfunctions of $\Rodzin_0$ of eigenvalue $\lambda$. By \eqref{eq:KernelRodzin0Laplacian}, $\{v_j:= \partial_{\bar z} u_j\}_{j=1}^{m_0}$ are eigenfunctions of $-\Delta_D$ of eigenvalue $\lambda$. By way of contradiction, assume that they are linearly dependent. Then, we can write
    \begin{equation}
        \underset{j=1}{\overset{m_0}{\sum}} \, c_j v_j = 0 \, \text{ in } \Omega, \quad \text{for some } c_j\in \C, \, j=1, \dots, m_0 \text{ not all identically } 0.
    \end{equation}
    But then, by linearity of $\Rodzin_0$, in $\Omega$ we have
    \begin{equation}
        \underset{j=1}{\overset{m_0}{\sum}} \, \lambda c_j u_j = \Rodzin_0 \left( \underset{j=1}{\overset{m_0}{\sum}} \, c_j u_j \right) = 4\partial_z \left( \underset{j=1}{\overset{m_0}{\sum}} \, c_j \partial_{\bar z} u_j \right) = 4\partial_z \left( \underset{j=1}{\overset{m_0}{\sum}} \, c_j v_j \right) = 0,
    \end{equation}
    contradicting the linear independence of $\{u_j\}_{j=1}^{m_0}$. Hence, $\{v_j:= \partial_{\bar z} u_j\}_{j=1}^{m_0}$ are linearly independent eigenfunctions of $-\Delta_D$ of eigenvalue $\lambda$. This ensures that $m_D \geq m_0$.

    Conversely, let $\lambda \in \sigma(-\Delta_D)$. Then, there exists $v\in H^2(\Omega)\cap H^1_0(\Omega)\setminus\{0\}$ such that 
    \begin{equation} \label{eq:LaplacianRodzin0}
        \begin{cases}
            -\Delta v = \lambda v & \text{in } \Omega, \\
            v = 0 & \text{on } \partial\Omega.
        \end{cases}
    \end{equation}
    Notice that it can not happen that $\partial_z v = 0$, since otherwise $\lambda$ would be $0$, which is not in $\sigma(-\Delta_D)$. Set $u:= \partial_z v$, and notice that $u$ satisfies, inside $\Omega$,
    \begin{equation} \label{eq:LaplacianKernelRodzin0}
        \begin{split}
            -\Delta u &= -\Delta (\partial_z v) = \partial_z (-\Delta v) = \partial_z (\lambda v) = \lambda \partial_z v = \lambda u, \, \text{ and }\\
            \partial_{\bar z} u & = \partial_{\bar z} (\partial_z v) = \frac{1}{4}\Delta v = -\frac{\lambda}{4} v.
        \end{split}
    \end{equation}
    Since $\partial_{\bar z} u = -\frac{\lambda}{4} v \in H^1_0(\Omega)$, we can take traces to deduce that $\partial_{\bar z} u = 0$ in $H^{1/2}(\partial\Omega)$. Hence, $u$ solves \eqref{eq:KernelRodzin0}, and therefore $\lambda \in \sigma(\Rodzin_0)$. To see that $m_0\geq m_D$, let $\{v_j\}_{j=1}^{m_D}$ be linearly independent eigenfunctions of $-\Delta_D$ of eigenvalue $\lambda$. By \eqref{eq:LaplacianKernelRodzin0}, $\{u_j:= \partial_z v_j\}_{j=1}^{m_D}$ are eigenfunctions of $\Rodzin_0$ of eigenvalue $\lambda$. By way of contradiction, assume that they are linearly dependent. Then, we can write
    \begin{equation}
        0 = \underset{j=1}{\overset{m_D}{\sum}} \, c_j u_j \, \text{ in } \Omega, \quad \text{for some } c_j\in \C, \, j=1, \dots, m_D \text{ not all identically } 0.
    \end{equation}
    That is, for some $c_j\in \C$, we have
    \begin{equation}
        \begin{cases}
            0 = \underset{j=1}{\overset{m_D}{\sum}} \, c_j u_j = \underset{j=1}{\overset{m_D}{\sum}} \, c_j \partial_z v_j = \partial_z \left( \underset{j=1}{\overset{m_D}{\sum}} \, c_j v_j \right) & \text{in } \Omega, \\
            \underset{j=1}{\overset{m_D}{\sum}} \, c_j v_j = 0 & \text{on } \partial\Omega,
        \end{cases}
    \end{equation}
    because $v_j\in H^1_0(\Omega)$ for every $j$. By the unique continuation principle for antiholomorphic functions, this leads to 
    \begin{equation}
        \underset{j=1}{\overset{m_D}{\sum}} \, c_j v_j = 0 \, \text{ in } \Omega, \quad \text{for some } c_j\in \C, \, j=1, \dots, m_D,
    \end{equation}
    contradicting the linear independence of $\{v_j\}_{j=1}^{m_D}$. Hence, $\{u_j:= \partial_z v_j\}_{j=1}^{m_D}$ are linearly independent eigenfunctions of $\Rodzin_0$ of eigenvalue $\lambda$. This ensures that $m_0 \geq m_D$, and hence $m_0 = m_D$.

    We conclude showing that $\sigma(\Rodzin_0) = \sigma_{\mathrm{p}}(\Rodzin_0)$. It suffices to show that $L^2(\Omega)$ has an orthonormal basis consisting of eigenfunctions of $\Rodzin_0$. Indeed, once we show this, then by \cite[Footnote 1 in Section 10.1.1]{Kato1995} we have that $\sigma(\Rodzin_0)$ coincides with the closure of $\sigma_{\mathrm{p}}(\Rodzin_0) = \{0\} \cup \sigma(-\Delta_D)$, but this set is already closed.
    
    For every $\lambda\in \sigma(-\Delta_D)$ with multiplicity $m_\lambda$, take $\{v_j^\lambda\}_{j=1}^{m_\lambda}$ orthogonal eigenfunctions of $-\Delta_D$ of eigenvalue $\lambda$, and take an orthonormal basis $\{h_k\}_{k\in \N}$ of $A^2(\Omega)$ ---for example, the one of \Cref{prop:MinMaxSplope}. By the previous facts, and by orthogonality in $L^2(\Omega)$ of any pair of eigenfunctions of $\Rodzin_0$ of different eigenvalue ---due to self-adjointness---, the set
    \begin{equation} \label{eq:OrthonormalSystemR0}
        \{h_k\}_{k\in \N} \cup \left\{ u_j^\lambda:= \partial_z v_j^\lambda \right\}_{\substack{ \lambda\in \sigma(-\Delta_D) \\ j=1,\dots, m_\lambda }}
    \end{equation}
    is an orthogonal system, which we can assume to be normalized ---normalizing each $u_j^\lambda$. We now verify that this orthonormal system is a basis of $L^2(\Omega)$. To this end, let $f\in L^2(\Omega)$ be orthogonal to every function in \eqref{eq:OrthonormalSystemR0}. First, we show that $f\in A^2(\Omega)$, that is, that $\partial_{\bar z} f =0$. For this, since the eigenfunctions of the Dirichlet Laplacian form an orthogonal basis of $L^2(\Omega)$, every $v\in \mathcal C^\infty_c(\Omega)$ can be expanded in the form
    \begin{equation} \label{eq:auxExpan}
        v = \underset{\substack{ \lambda\in \sigma(-\Delta_D) \\ j=1,\dots, m_\lambda }}{\sum} \, c_j^\lambda v_j^\lambda  \, \text{ in } \Omega, \quad \text{for some } c_j^\lambda\in \C.
    \end{equation}
    Since $v\in \mathcal C^\infty_c(\Omega)\subset H^2(\Omega)\cap H^1_0(\Omega)=\mathrm{Dom}(-\Delta_D)$ and the operator $-\Delta_D$ diagonalizes in the basis $\{v_j^\lambda\}_{\substack{ \lambda\in \sigma(-\Delta_D) \\ j=1,\dots, m_\lambda }}$, we have that
    \begin{equation}
        -\Delta v = \underset{\substack{ \lambda\in \sigma(-\Delta_D) \\ j=1,\dots, m_\lambda }}{\sum} \, c_j^\lambda (-\Delta v_j^\lambda) = \underset{\substack{ \lambda\in \sigma(-\Delta_D) \\ j=1,\dots, m_\lambda }}{\sum} \, \lambda c_j^\lambda v_j^\lambda \quad \text{in } L^2(\Omega),
    \end{equation}
    hence the coefficients $c_j^\lambda\in \C$ decay fast enough to ensure that \eqref{eq:auxExpan} holds in $H^1(\Omega)$. Thus, 
    \begin{equation}
        \langle f, \partial_z v\rangle_{L^2(\Omega)} = \underset{\substack{ \lambda\in \sigma(-\Delta_D) \\ j=1,\dots, m_\lambda }}{\sum} \, c_j^\lambda \langle f, \partial_z v_j^\lambda \rangle_{L^2(\Omega)} = \underset{\substack{ \lambda\in \sigma(-\Delta_D) \\ j=1,\dots, m_\lambda }}{\sum} \, c_j^\lambda \langle f, u_j^\lambda \rangle_{L^2(\Omega)} = 0 \quad \text{for every } v\in \mathcal C^\infty_c(\Omega),
    \end{equation}
    by orthogonality of $f$, showing that $f\in A^2(\Omega)$. Knowing this, since $0 = \langle f, h_k\rangle_{L^2(\Omega)}$ for every $k\in \N$, we conclude that $f\equiv 0$. This shows that \eqref{eq:OrthonormalSystemR0} is a basis of $L^2(\Omega)$. 

    To sum up, we have $\sigma(\Rodzin_0) = \sigma_{\mathrm{p}}(\Rodzin_0) = \{0\}\cup \sigma(-\Delta_D)$. Moreover, by the first part of the proof we have $\sigma_{\mathrm{ess}}(\Rodzin_0) = \{0\}$ and $\sigma_{\mathrm{d}}(\Rodzin_0) = \sigma(-\Delta_D)$. We also have seen that $0$ is an eigenvalue of infinite multiplicity, with associeated eigenspace given by the Bergamn space $A^2(\Omega)$, which is item $(i)$. Finally, we have seen that every eigenvalue in $\sigma_{\mathrm{p}}(\Rodzin_0) \setminus \{0\}$ has the same multiplicity as an eigenvalue of the Dirichlet Laplacian, and that the associated eigenfunctions are the $\partial_z$ of the eigenfunctions of the Dirichlet Laplacian, which is item $(ii)$.
\end{proof}

Now that the operator $\Rodzin_0$ is described, we address the convergence of $\Rodzin_a$ to $\Rodzin_0$ in the strong resolvent sense, as $a$ tends to $0^+$. To this end, we shall need to approximate functions in $\Dom(\Rodzin_0)$ by functions in $H^1(\Omega)$ in a suitable norm.

\begin{lemma} \label{lemma:densityDomR0}
    The space $\Dom(\Rodzin_0)$ is a Hilbert space when endowed with the scalar product 
    \begin{equation} 
        \langle u, v \rangle := \langle u, v \rangle_{L^2(\Omega)} + \langle \partial_{\bar z} u, \partial_{\bar z} v \rangle_{L^2(\Omega)} + \langle \Delta u, \Delta v \rangle_{L^2(\Omega)} \quad \text{for } u,v \in \Dom(\Rodzin_0)
    \end{equation}
    and the associated norm $\|u\| := \sqrt{\langle u, u \rangle}$ for $u\in \Dom(\Rodzin_0)$. In addition, the set $\Dom(\Rodzin_0) \cap H^1(\Omega)$ is dense in $\Dom(\Rodzin_0)$ with respect to the norm $\|\cdot\|$.
\end{lemma}

\begin{proof}
    Using that $\Dom(\partial_{\mathrm{h}})$ is a Hilbert space when endowed with the scalar product 
    \begin{equation} 
        \langle u, v \rangle_{\mathrm{h}} = \langle u, v \rangle_{L^2(\Omega)} + \langle \partial_{\bar z} u, \partial_{\bar z} v \rangle_{L^2(\Omega)},
    \end{equation}
    given by \Cref{lemma:propertiesWirtinger} $(i)$, we verify that $\Dom(\Rodzin_0)\subset \Dom(\partial_{\mathrm{h}})$ is a Hilbert space when endowed with the scalar product $\langle \cdot, \cdot \rangle$. Let $\{u_j\}_j\subset \Dom(\Rodzin_0)$ be a Cauchy sequence with respect to $\|\cdot\|$. On the one hand, $\{u_j\}_j\subset \Dom(\partial_{\mathrm{h}})$ is a Cauchy sequence with respect to the norm endowed by $\langle\cdot,\cdot\rangle_{\mathrm{h}}$, and hence there exists $u\in \Dom(\partial_{\mathrm{h}})$ such that $u_j\to u$ and $\partial_{\bar z} u_j \to \partial_{\bar z} u$ strongly in $L^2(\Omega)$, as $j\to+\infty$. On the other hand, $\{\Delta u_j\}_j\subset L^2(\Omega)$ is a Cauchy sequence with respect to the norm $\|\cdot\|_{L^2(\Omega)}$, and hence there exists $u_\Delta\in L^2(\Omega)$ such that $\Delta u_j\to u_\Delta$ strongly in $L^2(\Omega)$, as $j\to+\infty$. Using this, one can straightforwardly verify that $\Delta u = u_\Delta$ first in the weak sense, and then in $L^2(\Omega)$. Since $\partial_{\bar z} u \in \mathrm{Dom}(\partial_{\mathrm{ah}})$, by \Cref{lemma:propertiesWirtinger} $(ii)$ we have that $\partial_{\bar z} u\in H^{-1/2}(\partial\Omega)$. We show that actually $\partial_{\bar z} u =0$ in $H^{-1/2}(\partial\Omega)$. For every $g\in H^{1/2}(\partial\Omega)$, denoting also by $g$ an extension in $\Omega$, by the Green's formulas \eqref{eq:GreenFormulas} and the property of the pairing \eqref{eq:Brezis}, we have
    \begin{equation}
        \begin{split}
            \dfrac{1}{2} \overline{\langle \partial_{\bar z} u, \nu g \rangle}_{H^{-1/2}(\partial\Omega),H^{1/2}(\partial\Omega)} & =  \frac{1}{4}\int_\Omega \Delta u \,\overline g + \int_\Omega \partial_{\bar z} u \,\overline{\partial_{\bar z} g} = \underset{j\to +\infty}{\lim} \, \frac{1}{4}\int_\Omega \Delta u_j \,\overline g + \int_\Omega \partial_{\bar z} u_j \,\overline{\partial_{\bar z} g} \\
            & = \underset{j\to +\infty}{\lim} \, \dfrac{1}{2} \overline{\langle \partial_{\bar z} u_j, \nu g \rangle}_{H^{-1/2}(\partial\Omega),H^{1/2}(\partial\Omega)} =0,
        \end{split}
    \end{equation}
    where in the last equality we have used that $\partial_{\bar z} u_j \in H^1_0(\Omega)$. Applying \Cref{lemma:propertiesWirtinger} $(iv)$ to $\partial_{\bar z} u\in \mathrm{Dom}(\partial_{\mathrm{ah}})$, we conclude that $\partial_{\bar z} u \in H^1_0(\Omega)$, and hence that $u\in \Dom(\Rodzin_0)$. In conclusion, we have shown that the Cauchy sequence $\{u_j\}_j$ converges to $u\in \Dom(\Rodzin_0)$ with respect to the norm $\|\cdot\|$, which concludes the proof that $(\Dom(\Rodzin_0), \langle\cdot,\cdot\rangle)$ is a Hilbert space.

    In order to show that $\Dom(\Rodzin_0) \cap H^1(\Omega)$ is dense in $\Dom(\Rodzin_0)$, let $f\in \Dom(\Rodzin_0)\subset\mathrm{Dom}(\partial_{\mathrm{h}})$ be such that 
    \begin{equation} \label{eq:auxiliarDensity}
        \langle f, u \rangle_{L^2(\Omega)} + \langle \partial_{\bar z} f, \partial_{\bar z} u \rangle_{L^2(\Omega)} + \langle \Delta f, \Delta u \rangle_{L^2(\Omega)} = 0 \quad \text{for every } u\in \Dom(\Rodzin_0) \cap H^1(\Omega).
    \end{equation}
    In particular, $\langle f, \Phi_{\mathrm{h}} g \rangle_{L^2(\Omega)}=0$ for every $g\in \mathcal C^1(\partial\Omega)$ ---recall that $\Phi_{\mathrm{h}}$ is the Cauchy integral operator of \Cref{lemma:SingleLayer}, and that $\Phi_{\mathrm{h}} g\in A^2(\Omega)\cap H^1(\Omega)$ if $g\in \mathcal C^1(\partial\Omega)$. By \Cref{rmk:enoughForGainRegularity}, we have that $f\in \Dom(\Rodzin_0)\cap H^1(\Omega)$, and hence \eqref{eq:auxiliarDensity} holds replacing $u$ by $f$, namely
    \begin{equation}
        \|f\|_{L^2(\Omega)}^2 + \|\partial_{\bar z} f\|_{L^2(\Omega)}^2 + \|\Delta f\|_{L^2(\Omega)}^2 = 0.
    \end{equation}
    In conclusion $f\equiv 0$, as desired.
\end{proof}

\begin{proof}[Proof of \Cref{thm:SRCRodzin0}]
    We start showing that $\Rodzin_a$ converges to $\Rodzin_0$, as $a\to0^+$, in the strong resolvent sense. Since $\Rodzin_a$ is self-adjoint for every $a\geq 0$, by \cite[Lemma 6.37]{Teschl2014} and \cite[Theorem VIII.19 (b)]{ReedSimon1980} it is enough to prove that the difference of resolvents at $\lambda=i$, namely,
    \begin{equation}
        W_{a,0} := (\Rodzin_a-i)^{-1} - (\Rodzin_0-i)^{-1},
    \end{equation}
    converges to zero weakly in $L^2(\Omega)$, as $a\to0^+$. Given $f,g\in L^2(\Omega)$, set
    \begin{equation}
        \begin{split}
            \varphi_a & := (\Rodzin_a-i)^{-1} f \in \Dom(\Rodzin_a), \\
            \psi & := (\Rodzin_0+i)^{-1}g \in \Dom(\Rodzin_0).
        \end{split}
    \end{equation}
    By density of $\Dom(\Rodzin_0) \cap H^1(\Omega)$ in $\Dom(\Rodzin_0)$ with respect to $\|\cdot\|_{L^2(\Omega)} + \|\partial_{\bar z} \cdot\|_{L^2(\Omega)} + \|\Delta\cdot\|_{L^2(\Omega)}$, given by \Cref{lemma:densityDomR0}, we can find $\{\psi_k\}_k \subset \Dom(\Rodzin_0) \cap H^1(\Omega)$ such that 
    \begin{equation}
        \|\psi-\psi_k\|_{L^2(\Omega)} +  \|\Delta\psi-\Delta\psi_k\|_{L^2(\Omega)} \rightarrow 0, \quad \text{as } k\to +\infty.
    \end{equation}
    For every $k\in \N$, set
    \begin{equation}
        g_k := (\Rodzin_0+i)\psi_k \in L^2(\Omega),
    \end{equation}
    and notice that, by the triangle inequality, as $k\to +\infty$
    \begin{equation} \label{eq:gkConvergesTog}
        \|g-g_k\|_{L^2(\Omega)} = \|(\Rodzin_0+i)\psi - (\Rodzin_0+i)\psi_k \|_{L^2(\Omega)} \leq \|\psi-\psi_k\|_{L^2(\Omega)} + \|\Delta\psi-\Delta\psi_k\|_{L^2(\Omega)} \rightarrow 0.
    \end{equation}
    Similarly as in \eqref{eq:auxNRCPairingInfty}, integration by parts and self-adjointness of $\Rodzin_0$ straightforwardly lead to
    \begin{equation}
        \langle W_{a,0}f, g_k \rangle_{L^2(\Omega)} = - \langle \varphi_a, 2 \overline \nu \partial_{\bar z} \psi_k \rangle_{L^2(\partial\Omega)} + \langle 2 \bar \nu \partial_{\bar z} \varphi_a, \psi_k \rangle_{L^2(\partial\Omega)}.
    \end{equation}
    Using the boundary conditions $2 \bar \nu \partial_{\bar z} \varphi_a + a \varphi_a=0$ and $\partial_{\bar z} \psi_k=0$ that $\varphi_a\in \Dom(\Rodzin_a)$ and $\psi_k\in \Dom(\Rodzin_0)$ satisfy, we get that
    \begin{equation} \label{eq:auxWRCPairing}
        \langle W_{a,0}f, g_k \rangle_{L^2(\Omega)} = -a \langle \varphi_a, \psi_k \rangle_{L^2(\partial\Omega)}.
    \end{equation}
    As we shall see afterward, in order to prove the weak resolvent convergence it is enough to bound the pairing in the right-hand side of \eqref{eq:auxWRCPairing} uniformly in $a$ ---but not necessarily in $k$. This is what we do next.
    
    Since both $\varphi_a$ and $\psi_k$ belong to $H^1(\Omega)$, by the property of the pairing \eqref{eq:Brezis}, the trace theorem from $H^1(\Omega)$ to $H^{1/2}(\partial\Omega)$, and \Cref{lemma:propertiesWirtinger}, we deduce that
    \begin{equation} \label{eq:boundWa0}
        \begin{split}
            \left| \langle \varphi_a, \psi_k \rangle_{L^2(\partial\Omega)} \right| \leq \|\varphi_a\|_{H^{-1/2}(\partial\Omega)} \|\psi_k\|_{H^{1/2}(\partial\Omega)} \leq C  \left( \|\varphi_a\|_{L^2(\Omega)} + \|\partial_{\bar z} \varphi_a \|_{L^2(\Omega)} \right) \|\psi_k\|_{H^1(\Omega)}
        \end{split}
    \end{equation}
    for some constant $C>0$ depending only on $\Omega$. By the resolvent estimates \eqref{eq:BoundedResolventL2L2} and \eqref{eq:L2boundDz} for $\lambda=i$, we have that
    \begin{equation}
        \|\varphi_a\|_{L^2(\Omega)} + \|\partial_{\bar z} \varphi_a \|_{L^2(\Omega)} \leq (1+\sqrt{2}) \|f\|_{L^2(\Omega)}.
    \end{equation}
    Combining this estimate with \eqref{eq:boundWa0} and \eqref{eq:auxWRCPairing}, we obtain
    \begin{equation} \label{eq:DensityPairing}
        \left|\langle W_{a,0}f, g_k \rangle_{L^2(\Omega)}\right| \leq K a \|f\|_{L^2(\Omega)} \|\psi_k\|_{H^1(\Omega)}, \quad \text{for every } k\in \N,
    \end{equation}
    for some constant $K>0$ depending only on $\Omega$. In summary, by the triangle inequality and this estimate, for every $k\in \N$ we have
    \begin{equation}
        \begin{split}
            \left|\langle W_{a,0}f, g \rangle_{L^2(\Omega)}\right| & \leq \left|\langle W_{a,0}f, g-g_k \rangle_{L^2(\Omega)}\right| + \left|\langle W_{a,0}f, g_k \rangle_{L^2(\Omega)}\right| \\
            & \leq \|W_{a,0}\|_{L^2(\Omega)\to L^2(\Omega)} \| f \|_{L^2(\Omega)} \|g-g_k\|_{L^2(\Omega)} + K a \|f\|_{L^2(\Omega)} \|\psi_k\|_{H^1(\Omega)} \\
            & \leq 2 \| f \|_{L^2(\Omega)} \|g-g_k\|_{L^2(\Omega)} + K a \|f\|_{L^2(\Omega)} \|\psi_k\|_{H^1(\Omega)},
        \end{split}
    \end{equation}
    where we have used that $\|W_{a,0}\|_{L^2(\Omega)\to L^2(\Omega)} \leq 2$, which follows from the triangle inequality and the resolvent estimate \eqref{eq:BoundedResolventL2L2} applied to $\lambda=i$.

    Choosing first $k$ big enough and then $a$ small enough ---depending on the choice of $k$---, we conclude that $W_{a,0}$ converges to zero weakly in $L^2(\Omega)$, as desired. That any $\Lambda' \in \{0\}\cup \sigma(-\Delta_D)=\sigma(\Rodzin_0)$ is the limit of elements of $\sigma(\Rodzin_a)$, as $a\to0^+$, then follows by \cite[Theorem VIII.24 (a)]{ReedSimon1980}.

    We conclude with the convergence of the resolvents of $\Rodzin_a$ to the resolvent of $\Rodzin_0$ in operator norm, as $a\to0^+$, once projected into the orthogonal of the Bergman space, $A^2(\Omega)^\perp$. Recall from \Cref{sec:BergmanSpace} that the orthogonal projection
    \begin{equation} 
        P^\perp: L^2(\Omega) \rightarrow A^2(\Omega)^\perp \subset L^2(\Omega)
    \end{equation}
    is a well defined, bounded, self-adjoint operator in $L^2(\Omega)$, with $\|P^\perp\|_{L^2(\Omega)\to L^2(\Omega)} = 1$. We want to show that, for every $\lambda\in \C\setminus\R$,
    \begin{equation}
        W_a^\lambda := P^\perp \left((\Rodzin_a-\lambda)^{-1} - (\Rodzin_0-\lambda)^{-1}\right)
    \end{equation}
    converges to zero, as $a\to0^+$, in the operator norm. Let $f,g \in L^2(\Omega)$, and set
    \begin{equation}
        \begin{split}
            \varphi_a & := (\Rodzin_a-\lambda)^{-1}f \in \Dom(\Rodzin_a), \\
            \psi & := (\Rodzin_0-\overline \lambda)^{-1} P^\perp g \in \Dom(\Rodzin_0).
        \end{split}
    \end{equation}
    Notice that $\psi$ satisfies $\langle \psi, h \rangle_{L^2(\Omega)} = 0$ for every $h\in A^2(\Omega)$. Indeed, since by definition $\Rodzin_0 \psi - \overline \lambda \psi = P^\perp g$ in $\Omega$, by self-adjointness of $\Rodzin_0$, the orthogonality of $P^\perp g$, and since $h\in A^2(\Omega)=\ker(\Rodzin_0)$, we have
    \begin{equation}
        \overline \lambda \langle \psi, h \rangle_{L^2(\Omega)} = \langle \Rodzin_0 \psi, h\rangle_{L^2(\Omega)} - \langle P^\perp g, h \rangle_{L^2(\Omega)} = \langle \psi , \Rodzin_0 h \rangle_{L^2(\Omega)} = 0.
    \end{equation}
    As a consequence, by \Cref{lemma:GainRegularityOrthogonal} $\psi \in \Dom(\Rodzin_0) \cap H^1(\Omega)$, with
    \begin{equation} \label{eq:BoundH1DomRodzin0}
        \|\psi\|_{H^1(\Omega)} \leq C  \|\partial_{\bar z} \psi\|_{L^2(\Omega)},
    \end{equation}
    for some constant $C>0$ depending only on $\Omega$. Similarly as in \eqref{eq:auxNRCPairingInfty}, integration by parts and self-adjointness of $P^\perp$ and of $\Rodzin_0$ straightforwardly lead to
    \begin{equation}
        \langle W_a^\lambda f, g \rangle_{L^2(\Omega)} = - \langle \varphi_a, 2 \overline \nu \partial_{\bar z} \psi \rangle_{L^2(\partial\Omega)} + \langle 2 \bar \nu \partial_{\bar z} \varphi_a, \psi \rangle_{L^2(\partial\Omega)}.
    \end{equation}
    Using the boundary conditions $2 \bar \nu \partial_{\bar z} \varphi_a + a \varphi_a=0$ and $\partial_{\bar z} \psi=0$ that $\varphi_a\in \Dom(\Rodzin_a)$ and $\psi\in \Dom(\Rodzin_0)$ satisfy, we get that
    \begin{equation} \label{eq:auxNRCProjPairing}
        \langle W_a^\lambda f, g \rangle_{L^2(\Omega)} = -a \langle \varphi_a, \psi \rangle_{L^2(\partial\Omega)}.
    \end{equation}
    As we shall see afterward, in order to prove the norm resolvent convergence it is enough to bound the pairing in the right-hand side of \eqref{eq:auxNRCProjPairing} uniformly in $a$. This is what we do next.
    
    Since both $\varphi_a$ and $\psi$ belong to $H^1(\Omega)$, by the property of the pairing \eqref{eq:Brezis}, the trace theorem from $H^1(\Omega)$ to $H^{1/2}(\partial\Omega)$, \Cref{lemma:propertiesWirtinger}, and \eqref{eq:BoundH1DomRodzin0}, we deduce that
    \begin{equation} \label{eq:boundWaLambda}
        \left| \langle \varphi_a, \psi \rangle_{L^2(\partial\Omega)} \right| \leq \|\varphi_a\|_{H^{-1/2}(\partial\Omega)} \|\psi\|_{H^{1/2}(\partial\Omega)} \leq C \left( \|\varphi_a\|_{L^2(\Omega)} + \|\partial_{\bar z} \varphi_a \|_{L^2(\Omega)} \right) \|\partial_{\bar z} \psi\|_{L^2(\Omega)}
    \end{equation}
    for some constant $C>0$ depending only on $\Omega$. First, by the resolvent estimates \eqref{eq:BoundedResolventL2L2} and \eqref{eq:L2boundDz}, we have that
    \begin{equation} \label{eq:auxiliarL2BoundPhia}
        \|\varphi_a\|_{L^2(\Omega)} + \|\partial_{\bar z} \varphi_a \|_{L^2(\Omega)} \leq C_\lambda \|f\|_{L^2(\Omega)},
    \end{equation}
    for some constant $C_\lambda>0$ depending only on $\lambda$. Next, by self-adjointness of $\Rodzin_0$ ---see, for example, \cite[equation V.3.13]{Kato1995}---, there holds
    \begin{equation} \label{eq:L2BoundOrthogonalA2}
        \|\psi\|_{L^2(\Omega)} = \|(\Rodzin_0-\overline \lambda)^{-1} P^\perp g \|_{L^2(\Omega)} \leq \dfrac{1}{|\mathrm{Im}(\overline \lambda)|} \|P^\perp g\|_{L^2(\Omega)} \leq \dfrac{1}{|\mathrm{Im}(\overline \lambda)|} \|g\|_{L^2(\Omega)}.
    \end{equation}
    Moreover, notice that $\psi$ solves the problem
    \begin{equation}
        \begin{cases}
            -\Delta \psi - \overline \lambda \psi = P^\perp g & \text{in } \Omega, \\
            \partial_{\bar z} \psi = 0 & \text{on } \partial\Omega.
        \end{cases}
    \end{equation}
    Multiplying the equality $-\Delta \psi - \overline \lambda \psi = P^\perp g$ by $\overline \psi$ and integrating in $\Omega$, the divergence theorem and the boundary condition lead to
    \begin{equation}
        \int_\Omega P^\perp g \, \overline \psi = 4\int_\Omega |\partial_{\bar z} \psi|^2 - \overline \lambda \int_\Omega |\psi|^2.
    \end{equation}
    Finally, by the triangle inequality and \eqref{eq:L2BoundOrthogonalA2}, we can bound
    \begin{equation} \label{eq:L2BoundDzOrthogonalA2}
        4\|\partial_{\bar z} \psi\|_{L^2(\Omega)}^2 \leq \left| \int_\Omega P^\perp g \, \overline \psi + \overline \lambda \int_\Omega |\psi|^2 \right| \leq \left(\dfrac{1}{|\mathrm{Im}(\overline \lambda)|} + \dfrac{|\lambda|}{|\mathrm{Im}(\overline \lambda)|^2} \right) \| g\|_{L^2(\Omega)}^2.
    \end{equation}
    Then, combining the estimates \eqref{eq:auxiliarL2BoundPhia}, \eqref{eq:L2BoundOrthogonalA2} and \eqref{eq:L2BoundDzOrthogonalA2} with \eqref{eq:boundWaLambda} and \eqref{eq:auxNRCProjPairing}, we obtain that
    \begin{equation}
        \left|\langle W_a^\lambda f, g \rangle_{L^2(\Omega)}\right| \leq K a \|f\|_{L^2(\Omega)} \|g\|_{L^2(\Omega)},
    \end{equation}
    for some constant $K>0$ depending only on $\Omega$ and $\lambda$. If we now divide both sides of this last inequality by $\|f\|_{L^2(\Omega)} \|g\|_{L^2(\Omega)}$ and take the supremum among all functions $f,g\in L^2(\Omega)\setminus\{0\}$, we conclude that
    \begin{equation}
        \| W_a^\lambda \|_{L^2(\Omega)\to L^2(\Omega)} \leq K a,
    \end{equation}
    for some $K>0$ depending only on $\Omega$ and $\lambda$. The convergence follows by taking $a\to 0^+$.
\end{proof}

\section*{Acknowledgments}

The author thanks Albert Mas and Tom\'as Sanz-Perela for enlightening discussions. The author also thanks the suggestions of the anonymous referee, which have enriched the paper.

\appendix

\section{Eigenvalues of the disk} \label{sec:EigenBall}

In this appendix, we explicitly derive the eigenvalues and eigenfunctions of $\Rodzin_a$ in a disk of radius $R>0$ centered at the origin, which will be denoted $D_R$. We recall from \Cref{thm:PropertiesMukOmega} that all the eigenvalues are strictly positive. Hence, we seek for the solutions to 
\begin{equation} \label{eq:EigenRodzin}
    \begin{cases}
        -\Delta u = \mu u & \text{in } D_R, \\
        2 \bar \nu \partial_{\bar z} u + au = 0 & \text{on } \partial D_R,
    \end{cases}
\end{equation}
with $\mu>0$. To this end, we introduce polar coordinates: write a point $x\in \R^2$ as $x=r e^{i\phi}\equiv(r,\phi)$ with $r=|x|\in[0,+\infty)$ and $\phi = \arg(x) \in \mathbb [0,2\pi)$. Since one can decompose the space $L^2(D_R)$ as
\begin{equation} \label{eq:L2Decomposition}
    L^2(D_R) = \underset{j\in \N\cup\{0\}}{\bigoplus} \ L_j^\pm,
\end{equation}
where
\begin{equation}
    L_j^\pm = \left\{ u \in L^2(D_R) : \ u(r,\phi) = \rho(r)e^{\pm ij\phi} \text{ with } \rho\in L^2 \left( (0,R), rdr \right) \right\},
\end{equation}
it is enough to look for eigenfunctions of the form of separated variables.

By the standard procedure of the method of separation of variables, the solutions $u$ to $-\Delta u = \mu u$ in $D_R$ of the form of separated variables are
\begin{equation} \label{eq:GeneralSolHelmholtz}
    u^{j, \pm} (r, \phi) = \left( c_1^{j, \pm} J_j(\sqrt \mu r) + c_2^{j, \pm} Y_j(\sqrt \mu r) \right) e^{\pm ij\phi}, \quad \text{with }  j\in \N\cup\{0\}, \ c_1^{j, \pm}, c_2^{j, \pm} \in \C,
\end{equation}
where $J_j$ and $Y_j$ are the Bessel functions of first and second kind of order $j$, respectively; see~\cite[Chapter 9]{Abramowitz1964}.

Recall that the Bessel function $Y_j$ is singular at the origin for every $j\in \N\cup \{0\}$. By elliptic interior regularity, since the functions $u^{j, \pm}$ given by \eqref{eq:GeneralSolHelmholtz} must solve $-\Delta u^{j, \pm} = \mu u^{j, \pm}$ across the origin, we deduce that $c_2^{j, \pm}=0$, and thus $u^{j, \pm}$ is of the form
\begin{equation} \label{eq:EigenBall}
    u^{j, \pm} (r, \phi) = c^{j, \pm} J_j(\sqrt \mu r) e^{\pm ij\phi}, \quad \text{with }  j\in \N\cup\{0\}, \ c^{j, \pm}\in \C,
\end{equation}
where $r\in[0,R)$ and $\phi \in [0,2\pi)$.

We now impose the boundary condition $2\overline{\nu} \partial_{\bar z} u^{j, \pm} + a u^{j, \pm} = 0$ on $\partial D_R$ to characterize the eigenvalues $\mu$, so that $u^{j, \pm}$ solves \eqref{eq:EigenRodzin} in $D_R$. Since the boundary condition can be written as $\partial_\nu u^{j, \pm} + i \partial_\tau u^{j, \pm} + a u^{j, \pm} = 0$, the outward unit normal vector at $\partial D_R$ is $\nu=e_r$, and the unit tangent vector ---so that $\{\nu,\tau\}$ is positively oriented--- is $\tau=e_\phi$, the boundary condition writes in polar coordinates as
\begin{equation}
    \left( \partial_r + \frac{i}{R} \partial_\phi + a \right) \Bigg |_{r=R} c^{j, \pm} J_j(\sqrt \mu r) e^{\pm ij\phi} = 0.
\end{equation}
Using the recurrence relation \cite[formula (9.1.27)]{Abramowitz1964}, which asserts that for any linear combination $\mathscr C_p(\zeta) = c_1 J_p(\zeta) + c_2 Y_p(\zeta)$, with $p\in \R$, $\zeta\in \C$, and $c_1, c_2 \in \C$, it holds
\begin{equation} \label{eq:DerivativeBessel}
    \frac{d}{d\zeta} \mathscr C_p(\zeta) = - \mathscr C_{p+1} (\zeta) + \frac{p}{\zeta} \mathscr C_p(\zeta),
\end{equation}
we get that
\begin{equation}
    0 = \left[-\sqrt \mu J_{j+1}(\sqrt \mu R) + \left(\frac{j}{R} \mp \frac{j}{R} + a \right) J_j(\sqrt \mu R) \right] e^{\pm ij\phi}.
\end{equation}
This gives rise to two scenarios.
\begin{enumerate}[label=$(\roman*)$]
        \item For $u^{j, +}$ with $j\in \N\cup\{0\}$, we must have
\begin{equation} \label{eq:EigenEquationBallPlus}
    a = \sqrt \mu \frac{J_{j+1}(\sqrt \mu R)}{J_j(\sqrt \mu R)}.
\end{equation}
        \item For $u^{j, -}$ with $j\in \N\cup\{0\}$, we must have
\begin{equation} \label{eq:EigenEquationBallMinus}
    a + \frac{2j}{R} = \sqrt \mu \frac{J_{j+1}(\sqrt \mu R)}{J_j(\sqrt \mu R)}.
\end{equation}
    \end{enumerate}

The following lemma was stated and proven in \cite[Lemma B.1]{Mas2022} to collect the results on the Bessel functions needed to describe the solutions 
to \eqref{eq:EigenEquationBallPlus} and \eqref{eq:EigenEquationBallMinus}.

\begin{lemma} \label{Lemma:Bessel}
    $($\cite[Lemma B.1]{Mas2022}$)$
    Let $J_p$ be the Bessel function of the first kind of order $p \geq 0$, and denote the $k-$th positive zero of this function by $z_{p, k}$. Then,
	\begin{enumerate}[label=$(\roman*)$]
		\item \label{itemIncrSeq} the positive zeros of $J_p$ are simple and form an infinite increasing sequence;
		\item \label{itemInterlaced} the positive zeros of two consecutive Bessel functions are interlaced, meaning that
		\begin{equation}
			0< z_{p, 1} < z_{p + 1, 1} < z_{p, 2} < z_{p + 1, 2} < \cdots;
		\end{equation}
		\item \label{itemQuotient} the quotient of two consecutive Bessel functions can be expressed as
		\begin{equation}
			\label{Eq:QuotientBessel}
			\dfrac{J_{p + 1}(x)}{J_{p}(x)} = \sum_{k \geq 1} \dfrac{2x}{z_{p, k}^2 - x^2} \quad \text{ for } x \in \R \setminus \{z_{p, k}\}_{k\in \N}.
		\end{equation}
		As a consequence, $J_{p + 1}/J_{p}$ is odd, strictly increasing in each interval contained in $\R \setminus \{z_{p, k}\}_{k\in \N}$, it is positive in the intervals $(0,z_{p, 1})$ and $(z_{p + 1, k},z_{p, k + 1})$ for $k\geq 1$, and negative in the intervals $(z_{p, k},z_{p + 1, k})$ for $k\geq 1$.
	\end{enumerate}
\end{lemma}

With the help of \Cref{Lemma:Bessel}, we can now describe the eigenfunctions and eigenvalues of $\Rodzin_a$ in the disk $D_R$.

\begin{proposition} \label{prop:EigenDisk}
    Let $a>0$ and $R>0$ be given. For each $j \in \N\cup\{0\}$, the following holds:
    \begin{enumerate}[label=$(\roman*)$]
        \item \label{itemZeros1} Equation \eqref{eq:EigenEquationBallPlus} has infinitely many positive solutions, which can be enumerated as $\mu_k^{j,+}(a)$ with $k\in \N\cup\{0\}$. In addition,
        \begin{itemize}
            \item $\mu_0^{j,+}(a) \in \left( 0, (z_{j,1}/R)^2 \right)$;
            \item $\mu_k^{j,+}(a) \in \left( (z_{j+1,k}/R)^2, (z_{j,k+1}/R)^2 \right)$ for $k\in \N$.
        \end{itemize}
        \item \label{itemZeros2} Equation \eqref{eq:EigenEquationBallMinus} has infinitely many positive solutions, which can be enumerated as $\mu_k^{j,-}(a)$ with $k\in \N\cup\{0\}$. In addition,
        \begin{itemize}
            \item $\mu_0^{j,-}(a) = \mu_0^{j,+}(a)$;
            \item $\mu_0^{j,-}(a) > \mu_0^{j,+}(a)$ and $\mu_k^{j,-}(a) \in \left( (z_{j+1,k}/R)^2, (z_{j,k+1}/R)^2 \right)$ for $k\in \N$.
        \end{itemize}
    \end{enumerate}
    As a consequence, the following functions form an orthogonal basis of $L^2(D_R)$,
    \begin{equation} \label{eq:AllEigenfunctionsDisk}
        u_{D_R}^{j, k, \pm} (r, \phi) := J_j \left( \sqrt{\mu_k^{j, \pm}(a)} \, r \right) e^{\pm ij\phi} \quad \text{with } j,k\in\N\cup\{0\},
    \end{equation}
    and are eigenfunctions of $\Rodzin_a$ in $D_R$ with eigenvalue $\mu_k^{j, \pm}(a)$. In particular, the first eigenvalue is $\mu_{1,D_R}(a) = \mu_0^{0,+}(a)$, namely, the unique solution $\mu$ to
    \begin{equation}
        a = \sqrt{\mu} \frac{J_1 \left(\sqrt{\mu} R \right)}{J_0 \left(\sqrt{\mu} R \right)}
    \end{equation}
    in $(0, \Lambda_{D_R})$, where $\Lambda_{D_R} = (z_{0,1}/R)^2$. The associated eigenfunction is (up to a multiplicative constant)
    \begin{equation}
        u_{1,D_R}(a) (r,\phi) = J_0 \left( \sqrt{\mu_{1,D_R}(a)} \, r \right).
    \end{equation}
\end{proposition}

Solving numerically equations \eqref{eq:EigenEquationBallPlus} and \eqref{eq:EigenEquationBallMinus} for some choices of indexes, we obtain the plot of \Cref{fig:EigenRodzinBall} in \Cref{sec:QualitativeDisk}. In the same fashion, \Cref{fig:Eigenfunctions} shows the plot of the radial part of some eigenfunctions of the $\overline\partial$-Robin Laplacian, for some choices of $a>0$. Since the angular part of the eigenfunctions are the exponentials $e^{\pm i j \phi}$, we omit the plot of the argument. 
\vspace{-0.3cm}
\begin{figure}[h]
    \centering
    \includegraphics[width=0.95\linewidth]{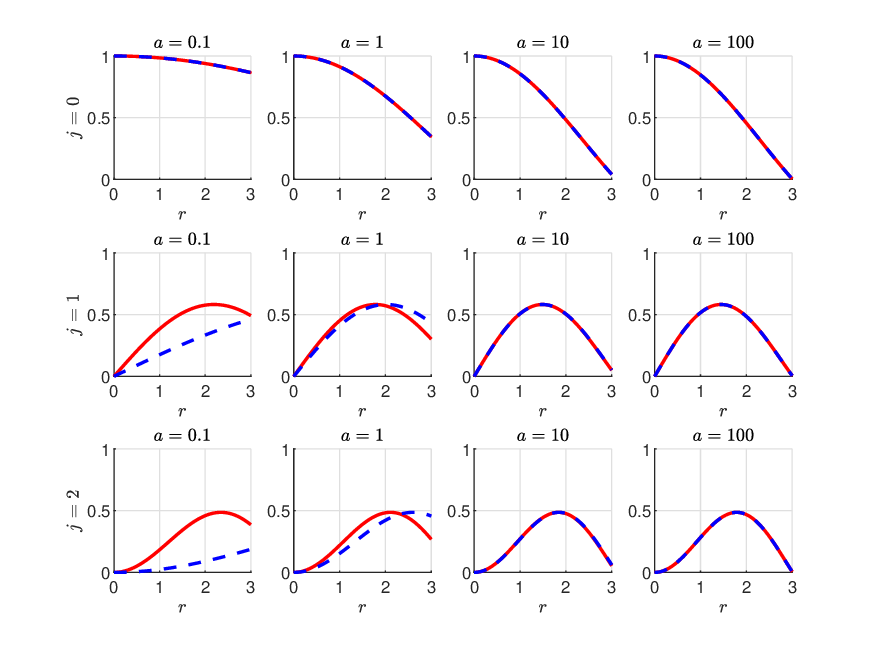}
    \vspace{-0.4cm}
    \caption{Plot of the radial part of the eigenfunctions \eqref{eq:AllEigenfunctionsDisk} for $k=0$, $a=0.1$ (first column), $a=1$ (second column), $a=10$ (third column), $a=100$ (fourth column), and $j=0$ (first row), $j=1$ (second row), $j=2$ (third row). The eigenfunctions with angular sign $-$ are plotted with continuous red curves, and the ones with angular sign $+$ are plotted with dashed blue curves.}
    \label{fig:Eigenfunctions}
\end{figure}

\begin{proof}[Proof of \Cref{prop:EigenDisk}]
    Statements \ref{itemZeros1} and \ref{itemZeros2} directly follow from \Cref{Lemma:Bessel}. As a consequence of this and the polar decomposition \eqref{eq:L2Decomposition}, it is clear that \eqref{eq:AllEigenfunctionsDisk} is an orthogonal basis of $L^2(\Omega)$ consisting of eigenfunctions of $\Rodzin_a$ in $D_R$ with eigenvalues $\mu_k^{j, \pm}(a)$.
    
    It only remains to show that the first eigenvalue is $\mu_{1,D_R}(a) = \mu_0^{0,+}(a)$, because then the rest of the statement follows by the previous and the well known fact that $\Lambda_{D_R} = (z_{0,1}/R)^2$ \cite[Proposition 1.2.14]{Henrot2006}. To this end, we want to show that $\mu_0^{0,+}(a)$ is the smallest positive solution to \eqref{eq:EigenEquationBallPlus} and \eqref{eq:EigenEquationBallMinus}.
    
    Since the zeros of the Bessel functions are interlaced ---\Cref{Lemma:Bessel} \ref{itemInterlaced}---, for every $j\in \N\cup\{0\}$ it follows that $z_{j,k} < z_{j+1,k}$ for all $k\in\N$, and therefore
    \begin{equation}
        \frac{2x}{z_{j,k}^2-x^2} > \frac{2x}{z_{j+1,k}^2-x^2} \quad \text{for every } x\in (0,z_{0,1}).
    \end{equation}
    Hence, by \Cref{Lemma:Bessel} \ref{itemQuotient} it follows that
    \begin{equation}
        \frac{J_{j+1}(x)}{J_j(x)} > \frac{J_{j+2}(x)}{J_{j+1}(x)} \quad \text{for every } x\in (0,z_{0,1}) \text{ and } j\in\N\cup\{0\}.
    \end{equation}
    As a consequence, the smallest positive $\mu$ solving \eqref{eq:EigenEquationBallPlus} is $\mu_0^{0,+}(a)$, and the smallest positive $\mu$ solving \eqref{eq:EigenEquationBallMinus} is $\mu_0^{0,-}(a) = \mu_0^{0,+}(a)$. Therefore, $\mu_{1,D_R}(a) = \mu_0^{0,+}(a)$, as desired.
\end{proof}

As a direct consequence of \Cref{prop:EigenDisk}, we deduce the simplicity of the first eigenvalue of $\Rodzin_a$ when the underlying domain is a disk.

\begin{corollary} \label{cor:MuDiskIsMuRobin}
    Let $D\subset \R^2$ be a disk. For every $a>0$, the first eigenvalue of the $\overline\partial$-Robin Laplacian $\Rodzin_a$ in $D$ is simple and coincides with the first eigenvalue of the Robin Laplacian $-\Delta_a$ in $D$. In addition, the associated eigenfunction for $\Rodzin_a$ is (up to a multiplicative constant) the same as the first eigenfunction for $-\Delta_a$.
\end{corollary}

\begin{proof}
    The description of $\mu_{1,D_R}(a)$ and $u_{1, D_R}(a)$ in \Cref{prop:EigenDisk} coincide with the description of $\mu_{1,D_R}^{\mathrm{Rob}}(a)$ and of the first eigenfunction of the Robin Laplacian in \cite[pg. 982 and Figure~3]{Freitas2021}. 
\end{proof}

\end{document}